\documentclass[11pt,leqno,a4paper]{article}

\usepackage{anysize}
\marginsize{3.5cm}{3.5cm}{2.5cm}{2.5cm}

\usepackage[]{hyperref}
\hypersetup{
    colorlinks=true,       
    linkcolor=red,          
    citecolor=blue,        
    filecolor=magenta,      
    urlcolor=cyan           
}

\usepackage{pstricks}
\usepackage{tikz}
\usepackage{amsmath}
\usepackage{amsfonts,amssymb}
\usepackage{enumerate}
\usepackage{mathrsfs}
\usepackage{amsthm}
\usepackage{a4wide}

\theoremstyle{plain}
\newtheorem{theo}{Theorem}[section]
\newtheorem{prop}[theo]{Proposition}
\newtheorem{lemm}[theo]{Lemma}
\newtheorem{coro}[theo]{Corollary}
\newtheorem{assu}[theo]{Assumption}
\newtheorem{defi}[theo]{Definition}
\theoremstyle{definition}
\newtheorem{rema}[theo]{Remark}

\DeclareMathOperator{\cnx}{div}

\DeclareMathOperator{\Hess}{Hess}
\DeclareMathOperator{\supp}{supp}
\DeclareMathOperator{\di}{d}

\DeclareSymbolFont{pletters}{OT1}{cmr}{m}{sl}
\DeclareMathSymbol{s}{\mathalpha}{pletters}{`s}

\def\htilde{\widetilde{h}}
\def\lDx#1{\langle D_x\rangle^{#1}\,}
\def\mU{\mathcal{U}}

\def\B{B }

\def\defn{\mathrel{:=}}

\def\Deltayx{\Delta_{x,y}}

\def\eps{\varepsilon}

\def\la{\left\lvert}
\def\lA{\left\lVert}
\def\le{\leq}
\def\les{\lesssim}

\def\ma{a}
\def\mez{\frac{1}{2}}
\def\partialx{\nabla}
\def\partialyx{\nabla_{x,y}}
\def\ra{\right\rvert}
\def\rA{\right\rVert}

\def\tdm{\frac{3}{2}}

\def\xN{\mathbf{N}}
\def\xR{\mathbf{R}}

\def\gain{\mu}

\numberwithin{equation}{section}

\title{Strichartz estimates for gravity water waves}
\author{
T. Alazard, 
N. Burq, 
C. Zuily}

\date{\empty}

\begin{document}
 \maketitle
 \begin{abstract}

We prove Strichartz estimates for gravity water waves, 
in arbitrary dimension and in fluid domains with general bottoms. We consider rough solutions 
such that, initially, the first order derivatives of the velocity field are not controlled in 
$L^\infty$-norm, or the initial free surface has a curvature 
not controlled in $L^2$-norm. 
\end{abstract}

\section{Introduction}
This paper continues the analysis started in \cite{ABZ3,ABZ4}. 
Our goal in this series of papers is to study gravity water waves 
having rough initial data. 
In \cite{ABZ3} we proved that the Cauchy problem is well-posed 
under the minimal assumptions to insure, in terms of 
Sobolev embeddings, that at time $t=0$ the trace $\underline{v}$ 
of the velocity at the free surface 
belongs to $W^{1,\infty}(\xR^d)$ and the free surface elevation $\eta$ 
belongs to~$W^{\tdm,\infty}(\xR^d)$. 
Here $d$ is the dimension of the free surface, so that $d=1$ for two-dimensional waves 
and $d=2$ for three-dimensional waves.

In this paper, we shall prove {\em a priori\/} 
estimates for solutions such that, initially, the $W^{1,\infty}(\xR^d)$-norm of 
$\underline{v}$ and the $W^{\tdm,\infty}(\xR^d)$-norm of $\eta$ are not bounded. 
With regards to the analysis of the possible singularities in the water waves equations, 
the most important conclusion is that, for two-dimensional gravity waves (for $d=1$), one can 
consider solutions such that, initially, the $L^2(\xR)$-norm of the curvature is not bounded.

The analysis is in two steps. 
The first step is to reduce the water waves equations to a wave 
type equation. This task was performed in 
\cite{ABZ3}. This reduction was based on two facts. Firstly, 
the Craig-Sulem-Zakharov reduction to a system on the boundary, introducing the Dirichlet-Neumann operator. The second fact (following Lannes~\cite{LannesJAMS}Ê
and \cite{ABZ1,AM}), is that one can use microlocal analysis 
to study the Dirichlet-Neumann operator in non smooth domains. For our purposes, the key point 
is that we proved in \cite{ABZ3} tame estimates depending linearly on H\"older norms, by using paradifferential calculus. 
Our goal here is to use Strichartz 
estimates to control these H\"older norms, and hence 
to improve the analysis of the Cauchy problem by combining these Strichartz estimates 
with the results proved in \cite{ABZ3}. 

The main difficulty will therefore consist in proving 
Strichartz estimates for gravity waves at a slightly lower level of regularity 
where we proved the existence of the solutions in \cite{ABZ3}. 
For the equations with surface tension, Strichartz estimates were proved in 
\cite{CHS,ABZ2}. Without surface tension, by contrast with other dispersive equations, 
one new point is that high frequencies propagate slower than low frequencies.

\subsection{Equations and assumptions on the fluid domain}

We consider the incompressible Euler equation in a 
time dependent fluid domain $\Omega$  
contained in a fixed container $\mathcal{O}$, located between a free surface 
and a fixed bottom. 
As in our previous papers, 
the only assumption we shall make on the bottom is that it is separated from the free surface 
by a ÓstripÓ of fixed length. 
Namely, we assume that, for each time $t$ in $[0,T]$, 
$$
\Omega(t)=\left\{ (x,y)\in \mathcal{O} \, :\, y < \eta(t,x)\right\},
$$
where $\mathcal{O}\subset \xR^d\times \xR$ is a given open connected set. The spatial 
coordinates are $x\in \xR^d$ (horizontal) and $y\in \xR$ (vertical) with $d\ge 1$. 
We assume that the free surface
$$
\Sigma(t)=\left\{ (x,y)\in \xR^d\times \xR \, :\, y = \eta(t,x)\right\},
$$ 
is separated from the bottom $\Gamma=\partial\Omega(t)\setminus \Sigma(t)$ 
by a curved strip. This means that we study the case where there exists $h>0$ such that, for 
any $t$ in $[0,T]$, 
\begin{equation}\label{n1}
\Omega_{h}(t)\defn \left\{ (x,y)\in \xR^d\times \xR \, :\, \eta(t,x)-h < y < \eta(t,x)\right\} \subset \mathcal{O}.
\end{equation}
\textbf{Examples.} $i)$ $\mathcal{O}=\xR^d\times \xR$ corresponds to the infinite depth case (then 
$\Gamma=\emptyset$); $ii)$ The finite depth case corresponds to the case where $\mathcal{O}=\{(x,y)\in \xR^d\times \xR\,:\, y>b(x)\}$ for some continuous function $b$ such that $\eta(t,x)-h>b(x)$ for any time $t$ (then $\Gamma=\{y=b(x)\}$). 
Notice that no regularity assumption is required on $b$. $iii)$ 
See the picture below.

We consider a potential flow such that the velocity $v$ is given by $\nabla_{x,y}\phi$ 
for some harmonic function $\phi\colon \Omega\rightarrow \xR$. 
The equations by which the motion is determined are the following:
\begin{equation}\label{n2}
\left\{
\begin{aligned}
&\Deltayx\phi=0 &&  \text{in }\Omega \\[1ex]
&\partial_{t}\phi+\frac{1}{2}\la \partialyx\phi\ra^2  +P+g y = 0
&&\text{in }\Omega \\[1ex]
&\partial_{t} \eta = \partial_{y}\phi -\partialx\eta\cdot\partialx \phi &&\text{for }y=\eta(t,x) \\[1ex]
&P = 0
&&\text{for }y=\eta(t,x)\\[1ex]
&\partial_\nu \phi=0 &&\text{on }\Gamma,
\end{aligned}
\right.
\end{equation}
where $g>0$ is acceleration due to gravity, $P$ is the pressure and 
$\nu$ denotes the normal vector to $\Gamma$ (whenever it exists; 
for general domains, one solves the boundary value problem by a variational argument, 
see \cite{ABZ1,ABZ3}). 
The system~\eqref{n2} is augmented with initial data at time $0$. 

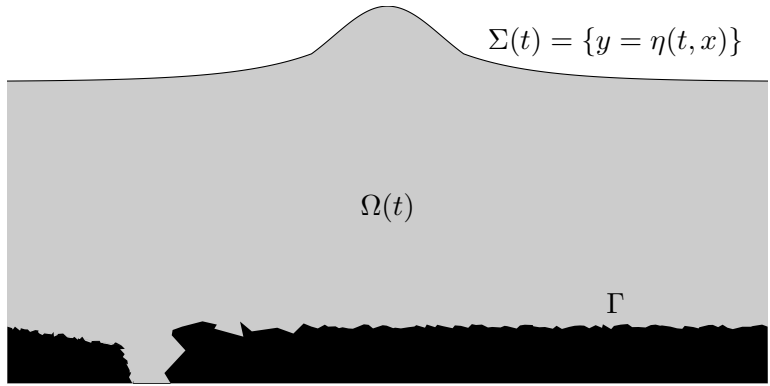
\usetikzlibrary{fadings}
\usetikzlibrary{decorations}
\usepgflibrary{decorations.pathmorphing}

\tikzfading[name=fade out, inner color=transparent!0,
  outer color=transparent!100]

\begin{center}
\begin{tikzpicture}[scale=1,samples=100]
    
\clip (-5,-3) rectangle (5,2);

    \filldraw [white!80!black,draw=black]
plot [domain=-6:-1] ({\x},{1+exp(\x)}) --
plot [domain=-1:1] ({\x},{1+exp(-\x*\x)}) --  
plot [domain=1:6] ({\x},{1+exp(-\x)}) -- (6,0) -- (-6,0) ;
    \fill[color=white!80!black,draw=none] (-12,-6) rectangle (9,1);
    \fill[black]
    decorate [decoration={random steps,segment length=1pt,amplitude=1pt}] 
    {(-5,-2.25) -- (-3.5,-2.5)}
    decorate [decoration={random steps,segment length=1pt,amplitude=1pt}] 
    {(-3.5,-2.5) -- (-3,-4.25)}
    decorate [decoration={random steps,segment length=5pt,amplitude=4pt}] 
    {-- (-2.7,-2.25) -- (-1,-2.25)}
    decorate [decoration={random steps,segment length=2pt,amplitude=1pt}] 
    {-- (-1,-2.25) -- (5,-2.25)}
    -- (5,-3) -- (-5,-3) -- (-5,-2.25);
     \node at (0,-1) [above] {$\Omega(t)$};
\node at (3,1.2) [above] {$\Sigma(t)=\{y=\eta(t,x)\}$};
\node at (3,-2.2) [above] {$\Gamma$};
\end{tikzpicture}  
\end{center}

\subsection{Regularity thresholds for the water waves}\label{S:21}

A well-known property of smooth solutions of \eqref{n2} is that their energy 
is conserved
$$
\frac{d}{dt} 
\left\{\mez \int_{\Omega(t)} \la \nabla\phi (t,x,y)\ra^2 \, dxdy + \frac{g}{2}\int_{\xR^d} \eta(t,x)^2 \, dx \right\} =0.
$$
However, we do not know if weak solutions exist at this level of regularity 
(even the meaning of the equations is not clear at that level of regularity). 
This is the only known coercive quantity (see~\cite{BO}).

There is another regularity threshold given by a scaling invariance which holds 
in the infinite depth case (that is when $\mathcal{O}=\xR^d\times \xR$). 
If $\phi$ and $\eta$ are solutions of the gravity water waves equations, then $\phi_\lambda$ and $\eta_\lambda$ defined by
$$
\phi_\lambda(t,x,y)=\lambda^{-3/2} \phi (\sqrt{\lambda}t,\lambda x , \lambda y),\quad 
\eta_\lambda(t,x)=\lambda^{-1} \eta(\sqrt{\lambda} t,\lambda x),
$$
solve the same equations. Then the critical case corresponds to $\eta_0$ Lipschitz and $\phi_0$ in $\dot{C}^{3/2}$ (one can replace the H\"older spaces by 
other spaces having the same invariance by scalings). 

According to the scaling argument, 
one expects that the problem is ``ill-posed'' for initial data such that the free surface is not Lipschitz. 
See e.g.\ \cite{CaARMA,CCT} for such ill-posedness 
results for semi-linear equations. However, the water waves equations are not semi-linear and it 
is not clear that the scaling argument is the only relevant regularity threshold to determine 
the optimal regularity in the analysis of the Cauchy problem (we refer the reader to 
the discussion in Section~$1.1.2$ of the recent famous result by Klainerman-Rodnianski-Szeftel~\cite{KRS}). 
In particular, it remains an open problem to prove an ill-posedness result for the gravity water waves equations. 
We refer to the recent paper by Chen, Marzuola, Spirn and Wright~\cite{CMSW} for a related result with surface tension. 

There are several additional criteria which appear in the mathematical analysis 
of the water waves equations and 
some comments are in order. The first results on the water waves equations required that 
the initial data are smooth enough, so that at least 
the curvature belongs to $L^\infty(\xR^d)$ as well as the first order derivatives of the velocity. 
The literature on the subject is now well known, 
starting with the pioneering work of Nalimov~\cite{Nalimov} (see also 
Yosihara~\cite{Yosihara} and Craig~\cite{Craig1985}) who showed the unique solvability in 
Sobolev spaces under a smallness assumption. 
Wu proved that the Cauchy problem is well posed 
without smallness assumption~(\cite{WuJAMS,WuInvent}) and 
several extensions of this result were obtained 
by different methods by several authors (see~\cite{ABZ3,CS,LannesJAMS,LindbladAnnals,MR,SZ}). 

Let us mention that it seems compulsory to assume that the 
gradient of the velocity is bounded (see~\cite{ChLi,ABZ3,WaZh}). 
Below we shall consider solutions such that, 
even though the initial velocity field is only  $C^{1- \epsilon}$, 
the solution itself is still $L^2((-T,T);C^{1+\epsilon'})$. 

Eventually, we recall from the appendix of \cite{ABZ4} that 
the Cauchy problem for the linearized water equations, which can be written under the form 
$\partial_t u +i \sqrt{g} \la D_x\ra^{1/2} u=0$, is ill-posed on H\"older spaces. So we need to work in Sobolev spaces, 
although the continuation criteria mentioned above are most naturally stated working in H\"older spaces.

\subsection{Reformulation of the equations}

Following Zakharov~(\cite{Zakharov1968}) 
and Craig and Sulem 
(\cite{CrSu}) 
we reduce the water waves equations 
to a system on the free surface. To do so, notice that since the velocity potential 
$\phi$ is harmonic, it is fully determined by the knowledge of $\eta$ and the knowledge of its trace at the free surface, denoted by $\psi$. 
Then one uses the 
Dirichlet-Neumann operator which maps a function defined on the free surface to the normal derivative of its harmonic extension. 

Namely, if $\psi=\psi(t,x) \in\xR$ is defined by 
$$
\psi(t,x)=\phi(t,x,\eta(t,x)),
$$
and if the Dirichlet-Neumann operator is defined by 
\begin{align*}
(G(\eta) \psi)  (t,x)&=
\sqrt{1+|\partialx\eta|^2}\,
\partial _n \phi\arrowvert_{y=\eta(t,x)}\\
&=(\partial_y \phi)(t,x,\eta(t,x))-\partialx \eta (t,x)\cdot (\partialx \phi)(t,x,\eta(t,x)),
\end{align*}
then one obtains the following system for two unknowns of the variables $(t,x)$,
\begin{equation}\label{n10}
\left\{
\begin{aligned}
&\partial_{t}\eta-G(\eta)\psi=0,\\[1ex]
&\partial_{t}\psi+g \eta
+ \smash{\frac{1}{2}\la\partialx \psi\ra^2  -\frac{1}{2}
\frac{\bigl(\partialx  \eta\cdot\partialx \psi +G(\eta) \psi \bigr)^2}{1+|\partialx  \eta|^2}}
= 0.
\end{aligned}
\right.
\end{equation}
We refer to \cite{ABZ1,ABZ3} for a precise construction of $G(\eta)$ 
in a domain with a general bottom. 

For general domains, it is proved in \cite{Bertinoro} that 
if a solution $(\eta,\psi)$ of System~\eqref{n10} belongs 
to $C^0([0,T];H^s(\xR^d))$ for some $T>0$ and $s>1+d/2$, then one can define a velocity potential $\phi$ and 
a pressure $P$ satisfying \eqref{n2}. 
Below we shall always consider solutions such that  $(\eta,\psi)$ belongs 
to $C^0([0,T];H^s(\xR^d))$ for some $s>1+d/2$ (which is the scaling index). It is thus sufficient to solve the 
Craig--Sulem--Zakharov formulation~\eqref{n10} of the water waves equations.

\subsection{Local in time existence for low regularity solutions}

We shall work below with the horizontal and vertical traces of the velocity 
on the free boundary, namely
$$
B= (\partial_y \phi)\arrowvert_{y=\eta},\quad 
V = (\nabla_x \phi)\arrowvert_{y=\eta}.$$ 
These can be defined only in terms of $\eta$ and $\psi$ by means of the formula 
\begin{equation}\label{defi:BV}
B\defn \frac{\partialx \eta \cdot\partialx \psi+ G(\eta)\psi}{1+|\partialx  \eta|^2},
\qquad
V\defn \partialx \psi -B \partialx\eta.
\end{equation}
Also, recall that the Taylor coefficient $\ma$ defined 
by
$$a=-\partial_y P\arrowvert_{y=\eta}
$$
can be defined in terms of 
$\eta,\psi$ only (see \cite{Bertinoro} or Definition~$1.5$ in \cite{ABZ3}). 
In our previous paper we prove the following result. 
\begin{theo}[from \cite{ABZ3}]
\label{T1}
Let $d\ge 1$, $s>1+d/2$ and consider an initial data $(\eta_{0},\psi_{0})$ such that
\begin{enumerate}
\item \label{assu1}
$\eta_0\in H^{s+\mez}(\xR^d),\quad \psi_0\in H^{s+\mez}(\xR^d),\quad V_0\in H^{s}(\xR^d),\quad B_0\in H^{s}(\xR^d)$,
\item there exists $h>0$ such that condition \eqref{n1} holds initially for $t=0$,
\item (Taylor sign condition) there exists $c>0$ such that, for all $x\in \xR^d$, 
$\ma_0(x)\ge c$.
\end{enumerate}
Then there exists $T>0$ such that 
the Cauchy problem for \eqref{n10} 
with initial data  $(\eta_{0},\psi_{0})$ has a unique solution 
$(\eta,\psi)\in C^0\big([0,T];H^{s+\mez}(\xR^d)\times H^{s+\mez}(\xR^d)\big)$ 
such that 
\begin{enumerate}
\item we have $(V,B)\in C^0\big([0,T];H^{s}(\xR^d)\times H^{s}(\xR^d)\big)$,
\item the condition \eqref{n1} holds for 
$0\le t\le T$, with $h$ replaced with $h/2$,
\item for all $0\le t\le T$ and for all $x\in \xR^d$, 
$\ma(t,x)\ge c/2$.
\end{enumerate}
\end{theo}

The same result holds for periodic initial data. In \cite{ABZ4} we proved that 
it holds in fact in the more general setting of uniformly local Sobolev spaces.

\subsection{A priori estimates}

Our goal in this paper is to 
take benefit of the dispersive properties of the water-waves system and improve 
the regularity thresholds just exhibited. 
This improvement should be seen as an illustration of the fact that the water-wave system is, after suitable 
paralinearizations, a {\em quasi-linear} wave type equation and consequently, 
the machinery developed in this context by Bahouri-Chemin~\cite{BaCh} 
and Tataru~\cite{TataruNS} 
applies to this context (see also 
Staffilani-Tataru~\cite{StTa} and Burq-G\'erard-Tzvetkov~\cite{BGT1}; we also refer the reader to 
the recent 
book by Bahouri-Chemin-Danchin~\cite{BCD} for references on Strichartz estimates for 
equations with non smooth coefficients). 
For the sake of simplicity, 
we restrict our attention to a priori estimates and make the following assumptions.
\begin{assu}\label{A:2}
Let $T_0$ in $(0,1]$. We consider smooth solutions of \eqref{n10} such that
\begin{enumerate}[i)]
\item $(\eta,\psi)$ belongs to $C^1([0,T_0]; H^{s_0}(\xR^d)\times H^{s_0}(\xR^d))$ for some 
$s_0$ large enough;
\item there exists $h>0$ such that \eqref{n1} holds for any $t$ in $[0,T_0]$ (this is the assumption 
that there exists a curved strip of width $h$ separating the free surface from the bottom);
\item there exists $c>0$ such that the Taylor coefficient $a(t,x)=-\partial_y P\arrowvert_{y=\eta(t,x)}$ is bounded from below by $c$ 
from any $(t,x)$ in $[0,T_0]\times \xR^d$.
\end{enumerate}
\end{assu}

For $\rho= k + \sigma$ with $k\in\xN$ and $\sigma \in (0,1)$, recall that one denotes 
by~$W^{\rho,\infty}(\xR^d)$ 
the space of functions whose derivatives up to order~$k$ are bounded and uniformly H\"older continuous with exponent~$\sigma$. Hereafter, we always consider indexes $\rho\not\in\xN$.

Define, for $T$ in $(0,T_0]$, the norms
\begin{equation*}
\begin{aligned}
M_s(T)&\defn 
\lA (\psi,\eta,B,V)\rA_{C^0([0,T];H^{s+\mez}\times H^{s+\mez}\times H^s\times H^s)},\\
Z_r(T)&\defn \lA \eta\rA_{L^p([0,T];W^{r+\mez,\infty})}
+\lA (B,V)\rA_{L^p([0,T];W^{r,\infty}\times W^{r,\infty})},\end{aligned}
\end{equation*}
where $p=4$ if $d=1$ and $p=2$ for $d\ge 2$, and set
$$
M_{s,0}\defn \lA (\psi(0),\eta(0),B(0),V(0))\rA_{H^{s+\mez}\times H^{s+\mez}\times H^s\times H^s}.
$$

The main result of this paper is the following
\begin{theo}\label{T2}
Let $T_0>0$ and $\gain$ be such that $\gain=\frac{1}{24}$ if $d=1$ and 
$\gain<\frac{1}{12}$ for $d\ge 2$. Consider two real numbers $s$ and $r$ satisfying
\begin{equation*}
s>1+\frac{d}{2}-\gain, \quad 1<r<s+\gain-\frac{d}{2},\quad 
r\not\in\mez \xN. 
\end{equation*}
For any $A>0$ there exist $B>0$ and $T_1\in (0,T_0]$ such that, 
for all smooth solution 
$(\eta,\psi)$ of \eqref{n10} defined on the time interval~$[0,T_0]$ and satisfying 
Assumption~\ref{A:2} on this time interval, if 
$M_{s,0}\le A$ then $M_s(T_1)+Z_r(T_1)\le B$.
\end{theo}
\begin{rema}In terms of Sobolev embeddings, $M_{s,0}$ only controls the 
$W^{1-\mu,\infty}$-norm (resp.\ $W^{\tdm-\mu,\infty}$-norm) 
of the traces $(V,B)$ of the velocity on the free surface (resp.\ $\eta$). 
Moreover, for $d=1$, $s+\mez$ might be strictly smaller than $2$ so that the $L^2(\xR)$-norm 
of the curvature of 
the initial free surface might be infinite.
\end{rema}  

The proof of Theorem~\ref{T2}  relies on the paradifferential reduction established in \cite{ABZ3} where we proved  
that the equations can be reduced to a very simple form
\begin{equation}\label{eqpara}
\Big(  \partial_t + \mez( T_{V } \cdot \nabla + \nabla \cdot T_{V }) +i T_{\gamma } \Big)u  =f 
\end{equation}
where $T_V$ is a paraproduct and $T_\gamma$ is a para-differential operator of 
$\frac{1}{2}$ with symbol 
$$
\gamma=\sqrt{\ma \lambda},\quad 
\lambda=\sqrt{\big(1+ \vert \nabla \eta\vert^2\big)\vert \xi \vert^2 
- \big(\xi \cdot \nabla \eta\big)^2}.
$$
Here $a$ is the Taylor coefficient and 
$\lambda$ is the principal symbol of the Dirichlet-Neumann operator 
(see Section~\ref{s:2} below for the definition of paradifferential operators). 
Using this reduction, we combine the Sobolev estimates 
already proved in \cite{ABZ3} and estimates in H\"older spaces coming 
from Strichartz ones.

\subsection{Strichartz estimates for gravity water waves}
Most of this paper is devoted to the proof of the following result, which 
is the main step in the proof of Theorem~\ref{T2}.

\begin{theo}\label{T4}
Let $I= [0,T]$, $d\geq 1$. 
Let $\gain$ be such that 
$\gain=\frac{1}{24}$ if $d=1$ and $\gain<\frac{1}{12}$ if $d\ge 2$.

Let  $s>1+\frac{d}{2}-\gain$ and $f \in L^\infty(I; H^s(\xR^d))$. 
Let $u\in C^0(I; H^s(\xR^d))$ be a solution of \eqref{eqpara}. 
Then one can find $k = k(d)$ such that
\begin{equation*}
\begin{aligned}
\Vert  u  \Vert_{L^p(I;  W^{s-\frac{d}{2}+\gain, \infty}(\xR^d))}
\leq    \mathcal{F}\big(\Vert V\Vert_{E_0} +  \mathcal{N}_k(\gamma)\big) 
\Big\{  \Vert  f\Vert_{L^p(I; H^{s }(\xR^d))} 
+ \Vert  u \Vert_{C^0(I; H^s(\xR^d))}\Big\}
\end{aligned}
\end{equation*}
where   $p=4$ if $ d=1 $ and $p=2$ if $d \geq 2.$ Here $E_0 = L^p(I; W^{1,\infty}(\xR^d))^d$ and $\mathcal{N}_k(\gamma) = \sum_{\vert \beta  \vert \leq k} \Vert D_\xi^\beta \gamma\Vert_{L^\infty(I\times \xR^d \times \mathcal{C})}$ 
with $\mathcal{C}=\{\frac{1}{10}\le\la\xi\ra\le10\}$.  
\end{theo}

The first step of the proof of Theorem~\ref{T4} 
is classical in the context of quasi-linear wave equations (see the works by Lebeau~\cite{Lebeau92}, 
Smith~\cite{Smith}, Bahouri-Chemin~\cite{BaCh}, 
Tataru~\cite{TataruNS} and Blair~\cite{Blair}). 
It consists, after a dyadic decomposition at frequency $h^{-1}$, 
in regularizing the coefficients at scale $h^{-\delta}, \delta \in (0,1)$. 
Then, we need to straighten the vector field $\partial_t +V\cdot \nabla$ 
by means of a para-change of variables (see~\cite{AliXEDP,Alipara}). 
Finally, we are able to write a parametrix for the reduced system, which 
allows to prove Strichartz estimates using the usual strategy ($TT^*$ method) 
on a small time interval $|t| \leq h^{\frac{\delta}{2}}$. To conclude, we glue the estimates.

\section{Symmetrization of the equations}\label{s:2}

 In this section we recall some results from~\cite{ABZ3} which are used in the sequel. 
 
 \subsection{Paradifferential operators}

We begin by recalling notations for Bony's paradifferential operators (see \cite{Bony,MePise,Meyer,Taylor}). Except in the last section, we shall 
not use paradifferential calculus in this paper, though 
we use in an essential way the paradifferential reduction made in \cite{ABZ3}. 

For~$k\in\xN$, we denote by $W^{k,\infty}({\mathbf{R}}^d)$ the usual Sobolev spaces.
For $\rho= k + \sigma$, $k\in \xN, \sigma \in (0,1)$ denote 
by~$W^{\rho,\infty}({\mathbf{R}}^d)$ 
the space of functions whose derivatives up to order~$k$ are bounded and uniformly H\"older continuous with 
exponent~$\sigma$. 

\begin{defi}\label{T:5}
Given~$\rho\in [0, 1]$ and~$m\in\xR$,~$\Gamma_{\rho}^{m}({\mathbf{R}}^d)$ denotes the space of
locally bounded functions~$a(x,\xi)$
on~${\mathbf{R}}^d\times({\mathbf{R}}^d\setminus 0)$,
which are~$C^\infty$ functions of $\xi$ outside the origin 
and
such that, for any~$\alpha\in\xN^d$ and any~$\xi\neq 0$, the function
$x\mapsto \partial_\xi^\alpha a(x,\xi)$ is in $W^{\rho,\infty}({\mathbf{R}}^d)$ and there exists a constant
$C_\alpha$ such that,
\begin{equation*}
\forall\la \xi\ra\ge \mez,\quad 
\lA \partial_\xi^\alpha a(\cdot,\xi)\rA_{W^{\rho,\infty}(\xR^d)}\le C_\alpha
(1+\la\xi\ra)^{m-\la\alpha\ra}.
\end{equation*}
\end{defi}
Given a symbol~$a$ in one such symbol class, one defines
the paradifferential operator~$T_a$ by
\begin{equation}\label{eq.para}
\widehat{T_a u}(\xi)=(2\pi)^{-d}\int \chi(\xi-\eta,\eta)\widehat{a}(\xi-\eta,\eta)\psi(\eta)\widehat{u}(\eta)
\, d\eta,
\end{equation}
where
$\widehat{a}(\theta,\xi)=\int e^{-ix\cdot\theta}a(x,\xi)\, dx$
is the Fourier transform of~$a$ with respect to the first variable; 
$\chi$ and~$\psi$ are two fixed~$C^\infty$ functions such that:
\begin{equation}\label{cond.psi}
\psi(\eta)=0\quad \text{for } \la\eta\ra\le 1,\qquad
\psi(\eta)=1\quad \text{for }\la\eta\ra\geq 2,
\end{equation}
and~$\chi(\theta,\eta)$ 
satisfies, for some small enough positive numbers $\eps_1<\eps_2$,
$$
\chi(\theta,\eta)=1 \quad \text{if}\quad \la\theta\ra\le \eps_1\la \eta\ra,\qquad
\chi(\theta,\eta)=0 \quad \text{if}\quad \la\theta\ra\geq \eps_2\la\eta\ra,
$$
and
$$
\forall (\theta,\eta)\,:\qquad \la \partial_\theta^\alpha \partial_\eta^\beta \chi(\theta,\eta)\ra\le 
C_{\alpha,\beta}(1+\la \eta\ra)^{-\la \alpha\ra-\la \beta\ra}.
$$

Given a symbol~$a\in \Gamma^m_{\rho}(\xR^d)$, we set
\begin{equation}\label{defi:norms}
M_{\rho}^{m}(a)= 
\sup_{\la\alpha\ra\le 1+2d+\rho~}\sup_{\la\xi\ra \ge 1/2~}
\lA (1+\la\xi\ra)^{\la\alpha\ra-m}\partial_\xi^\alpha a(\cdot,\xi)\rA_{W^{\rho,\infty}(\xR)}.
\end{equation}

\subsection{Change of unknowns and symmetrization}
Consider two real numbers $s$ and $r$ such that $s>\frac{3}{4}+\frac{d}{2}$ and $r>1$. 
Consider a smooth solution of \eqref{n10} defined on some time interval $[0,T]$ and 
satisfying Assumption~\ref{A:2}. 
Introduce
$$
\lambda(t,x,\xi)=\Bigl(\big(1+ \vert \nabla \eta(t,x)\vert^2\big)\vert \xi \vert^2 
- \big(\xi \cdot \nabla \eta(t,x)\big)^2\Bigr)^\mez.
$$
(Notice that if $\eta(t,\cdot)$ is in $W^{\rho+1,\infty}(\xR^d)$ then $\lambda(t,\cdot,\cdot)$ belongs to 
$\Gamma^1_\rho(\xR^d)$.)

Set $\zeta=\nabla\eta$ and introduce
$$
U_s \defn \lDx{s} V+T_\zeta \lDx{s}\B,\quad \zeta_s\defn \lDx{s}\zeta.
$$
We proved in \cite{ABZ3} that
\begin{equation}\label{syst:s}
\left\{
\begin{aligned}
&(\partial_t+T_V\cdot\partialx)U_s+T_\ma\zeta_s=f_1,\\
&(\partial_{t}+T_V\cdot\partialx)\zeta_s=T_\lambda U_s +f_2,
\end{aligned}
\right.
\end{equation}
where $a(t,x)=-\partial_y P\arrowvert_{y=\eta(t,x)}$ 
is the Taylor coefficient and where, for each time $t\in [0,T]$, 
\begin{multline*}
\lA (f_1(t),f_2(t))\rA_{L^2\times H^{-\mez}}\\
\le C\left( \lA (\eta,\psi)(t)\rA_{H^{s+\mez}},\lA (V,B)(t)\rA_{H^s}\right)
\left\{1+\lA \eta(t)\rA_{W^{r+\mez,\infty}}+\lA (V,B)(t)\rA_{W^{r,\infty}}\right\}.
\end{multline*}

To prove an $L^2$ estimate for System~\eqref{syst:s}, in \cite{ABZ3}, we performed a 
symmetrization of the non-diagonal part of~\eqref{syst:s}. Here we used 
the fact that the Taylor coefficient $\ma$ is a positive function. 
Let us mention that this is automatically satisfied 
for infinitely deep fluid domain: this was 
proved by Wu~(see~\cite{WuInvent,WuJAMS}) for smooth enough 
domains. As mentioned in \cite{ABZ3}, this 
remains valid for any~$C^{1,\alpha}$-domain, $0<\alpha<1$ (since 
the Hopf Lemma is true for such domains, see~\cite{Safonov} and the references therein). 

\begin{prop}[from~\cite{ABZ3}]\label{psym}
Let 
$s>\frac{3}{4}+\frac{d}{2}$, $r>1$. Introduce the symbols
$$
\gamma=\sqrt{\ma \lambda},\quad q= \sqrt{\frac{\ma}{\lambda}},
$$
and set $\theta_s=T_{q} \zeta_s$. Then 
\begin{align}
&\partial_{t} U_s+T_{V}\cdot\partialx  U_s  + T_\gamma \theta_s = F_1,\label{systreduit1}
\\
&\partial_{t}\theta_s+T_{V}\cdot\partialx \theta_s - T_\gamma U_s =F_2,
\label{systreduit2}
\end{align}
for some source terms $F_1,F_2$ satisfying 
\begin{multline*}
\lA (F_{1}(t),F_{2}(t))\rA_{L^{2}\times L^{2}}\\
\le C\left( \lA (\eta,\psi)(t)\rA_{H^{s+\mez}},\lA (V,B)(t)\rA_{H^s}\right)
\left\{1+\lA \eta(t)\rA_{W^{r+\mez,\infty}}+\lA (V,B)(t)\rA_{W^{r,\infty}}\right\}.
\end{multline*}
\end{prop}

One has the following corollary, which is the starting point of the proof of the 
Strichartz estimate.

\begin{coro}\label{T10}
With the above notations, set
\begin{equation}\label{eq:r0}
u=\lDx{-s}(U_s-i \theta_s)=\lDx{-s}(U_s-iT_{\sqrt{a/\lambda}}\nabla \eta_s),
\end{equation}
and recall the notation
$$
\gamma(t,x,\xi)=\sqrt{\ma(t,x)\lambda(t,x,\xi)}.
$$
Then $u$ satisfies the complex-valued equation
\begin{equation}\label{eq:r1}
\partial_t u +\mez \bigl( T_V\cdot \partialx +\partialx\cdot T_V\bigr)u + i T_\gamma u =f ,
\end{equation}
where $f$ satisfies
for each time $t\in [0,T]$,
\begin{equation}\label{eq:r2}
\lA f(t)\rA_{H^s}\le \mathcal{F}\big( \lA (\eta,\psi)(t)\rA_{H^{s+\mez}},\lA (V,B)(t)\rA_{H^s}\big) 
\left\{1+\lA \eta(t)\rA_{W^{r+\mez,\infty}}
+\lA (V,B)(t)\rA_{W^{r,\infty}}\right\}.
\end{equation}
\end{coro}
\begin{proof}
It immediately follows from Proposition~\ref{psym} that $u=U_s-i \theta_s$ 
satisfies
\begin{equation}\label{troisieme}
(\partial_t +T_V\cdot \partialx + i T_\gamma ) (U_s-i \theta_s)=h ,
\end{equation}
where, for each time $t\in [0,T]$,
$$
\lA h(t)\rA_{L^2}\le \mathcal{F}\big( \lA (\eta,\psi)(t)\rA_{H^{s+\mez}},\lA (V,B)(t)\rA_{H^s}\big)
\left\{1+\lA \eta(t)\rA_{W^{r+\mez,\infty}}
+\lA (V,B)(t)\rA_{W^{r,\infty}}\right\}.
$$
By commuting $\lDx{-s}$ to~\eqref{troisieme} we thus obtain
$$
\partial_t u +T_V\cdot \partialx u + i T_\gamma u =f' ,
$$
where the $H^s$-norm of $f'(t)$ is bounded by the right-hand side of \eqref{eq:r2}. 
This gives \eqref{eq:r1} with $f=f'+\mez T_{\cnx V} u$ and 
it remains to prove 
that the $H^s$-norm of $T_{\cnx V} u$ is bounded by the right-hand side of \eqref{eq:r2}. 
Since a paraproduct by an $L^\infty$-function acts on any Sobolev spaces, we have 
$\lA T_{\cnx V} u\rA_{H^s}\les \lA \cnx V\rA_{L^\infty}\lA u\rA_{H^s}$. Therefore, 
to complete the proof, it is sufficient to prove that
\begin{equation}\label{n201}
\lA u(t)\rA_{H^s}\le 
\mathcal{F}\big( \lA (\eta,\psi)(t)\rA_{H^{s+\mez}},\lA (V,B)(t)\rA_{H^s}\big) .
\end{equation}
To do so, write, by definition of $u$,
\begin{align*}
\lA u(t)\rA_{H^s}&\le \lA U_s\rA_{L^2}+\lA \theta_s\rA_{L^2}\\
&\le \lA \langle D_x\rangle^s V\rA_{L^2}+\lA T_\zeta \langle D_x\rangle^s B\rA_{L^2} 
+\lA T_q \zeta_s\rA_{L^2}\\
&\le \lA V\rA_{H^s}+\lA \zeta\rA_{L^\infty}\lA B\rA_{H^s} 
+\mathcal{M}^{-\mez}_0(q) \lA \zeta\rA_{H^{s-\mez}}
\end{align*}
where we used that $\lA T_p v\rA_{H^{\mu-m}}\les \mathcal{M}^m_0(p)\lA v\rA_{H^\mu}$ 
for any paradifferential operator with symbol $p$ in $\Gamma^m_0$ (see \eqref{defi:norms} 
for the definition of $\mathcal{M}^m_0(p)$). Now we have, using the Sobolev embeding, 
$$
\lA \zeta\rA_{L^\infty}+\mathcal{M}^{-\mez}_0(q)
\le \mathcal{F}\big( \lA \eta\rA_{W^{1,\infty}},\lA a\rA_{L^\infty}\big)
\le \mathcal{F}\big( \lA \eta\rA_{H^{s+\mez}},\lA a\rA_{H^{s-\mez}}\big)
$$
so \eqref{n201} follows from the Sobolev estimates for $a$ proved in \cite{ABZ3} (see Proposition~$4.6$).
\end{proof}

\section{Smoothing the paradifferential symbols}

 From now on we assume  $s>1 + \frac{d}{2} - \sigma_d$  where $\sigma_1= \frac{1}{24},\, \sigma_d =  \frac{1}{12}$ if $d \geq 2 $ and we set $s_0 :=\mez - \sigma_d >0.$ Then  $s-\mez >\frac{d}{2} +s_0.$ Now, with $I = [0,T], $ we introduce the spaces
 \begin{equation}\label{EFG}
 \left\{
 \begin{aligned}
 &E:=C^0( I; H^{s - \mez}(\xR^d)) \cap L^2(I; W^{\mez, \infty}(\xR^d)),\\
 & F:= C^0( I; H^{s}(\xR^d)) \cap L^2(I; W^{1, \infty}(\xR^d)),\\
 &G:= C^0( I; W^{s_0, \infty}(\xR^d)) \cap  L^2(I; W^{\mez, \infty}(\xR^d)) 
\end{aligned}
\right.
 \end{equation}
 endowed with their natural norms.
 
 We shall assume that 
 \begin{equation}
 \begin{aligned}
  &(i) \quad \ma \in E, \quad  \nabla \eta \in E, \quad V \in  F,\\
  &(ii) \quad  \exists c>0 : \ma(t,x) \geq c,\quad  \forall (t,x) \in I \times \xR^d.
\end{aligned}
 \end{equation}
 Let us recall that 
\begin{equation}\label{U}
\begin{aligned}
&\gamma (t,x,\xi) =   \big(\ma ^2\mU(t,x,\xi)\big)^{\frac{1}{4}}\\
&\mU(t,x,\xi) \defn (1+ \vert \nabla \eta\vert^2(t,x))\vert \xi \vert^2 - (\xi \cdot \nabla \eta(t,x))^2.
\end{aligned}
 \end{equation}
   
 Now we have,  for $\xi \in \mathcal{C}_0 := \{\xi : \mez \leq \vert \xi \vert \leq2\}$  considered as a parameter,
$$
\ma^2\mU\in G 
$$
uniformly in $\xi$.

\begin{lemm}\label{estgamma}
There exists $\mathcal{F}:\xR^+ \to \xR^+$   such that $\Vert \gamma \Vert_G \leq \mathcal{F}( \Vert \nabla \eta \Vert_E) $ for all $\xi \in \mathcal{C}_0.$
\end{lemm}
\begin{proof}
By the Cauchy-Schwartz inequality we have $\mathcal{U}(t,x,\xi) \geq \vert \xi \vert^2$ from which we deduce that 
\begin{equation}\label{gamma}
\gamma(t,x,\xi) \geq c_0>0  \quad  \forall    (t,x,\xi) \in I\times \xR^d \times \mathcal{C}_0. 
\end{equation}
Moreover $\gamma \in C^0(I; L^\infty(\xR^d \times \mathcal{C}_0)).$ On the other hand, since
$$
\gamma^4(t,x,\xi) - \gamma^4(t,y,\xi) = (\gamma(t,x,\xi) -\gamma(t,y,\xi))\sum_{j=0}^3 
( \gamma(t,x,\xi))^{3-j}(\gamma(t,y,\xi))^j
$$
we have, using \eqref{gamma},
$$
\frac{\vert \gamma(t,x,\xi) 
-\gamma(t,y,\xi)\vert}{\vert x-y\vert^\sigma} 
\leq \frac{1}{ 4c_0^3} \frac{\vert (\ma^2\mU)(t,x,\xi) -(\ma^2\mU)(t,y,\xi)\vert}{\vert x-y\vert^\sigma}\cdot
$$
Taking $\sigma = s_0$ and $\sigma = 1/2$ we deduce the lemma.
\end{proof}

Guided by works by Lebeau~\cite{Lebeau92}, Smith \cite{Smith}, 
Bahouri-Chemin~\cite{BaCh}, Tataru~\cite{TataruNS} and Blair~\cite{Blair}, 
we smooth out the symbols of the operator. 

\begin{defi}
Now let $\psi \in C_0^\infty(\xR^d), \psi(\xi) 
= 1 \text{ if } \vert \xi \vert \leq \mez,  \psi(\xi) = 0 \text{ if } \vert \xi \vert \geq 1.$ With $ h= 2^{-j}$ 
and $\delta >0,$ which will be chosen later on, we set
\begin{equation}
\gamma_\delta (t,x,\xi) = \psi(h^\delta D_x)\gamma(t,x,\xi).
\end{equation}
\end{defi}

\begin{lemm}\label{estgamma-delta}
$(i)\quad  \forall \alpha, \beta \in \xN^d \quad \exists C_{\alpha, \beta} >0 : \forall t \in I, \forall \xi \in \xR^d$
$$\vert D_x^\alpha D^\beta_\xi \gamma_\delta (t,x,\xi) \vert \leq C_{\alpha, \beta} h^{-\delta \vert \alpha \vert} 
\Vert D^\beta_\xi  \gamma(t, \cdot, \xi) \Vert_{L^\infty(\xR^d)}.$$
$(ii) \quad  \forall \alpha, \beta \in \xN^d, \vert \alpha \vert \geq 1 \quad \exists C_\alpha >0 : \forall t \in I, \forall \xi \in \xR^d$
$$ \vert D_x^\alpha D^\beta_\xi  \gamma_\delta (t,x,\xi) \vert \leq C_{\alpha, \beta} h^{-\delta (\vert \alpha \vert - \mez)} 
\Vert D^\beta_\xi  \gamma(t, \cdot, \xi) \Vert_{W^{\mez,\infty}(\xR^d)}.$$
\end{lemm}

\begin{proof}
$(i)$ follows from the fact that
$$\gamma_\delta (t,x,\xi) = h^{-\delta d}\widehat{\psi}\Big(\frac{\cdot}{h^\delta}\Big) \ast \gamma (t,\cdot, \xi).$$
$(ii)$ We write
$$
D_x^\alpha D^\beta_\xi  \gamma_\delta(t,x,\xi)
= \sum_{k =-1}^{+\infty}\Delta_k D_x^\alpha \psi(h^\delta D_x) D^\beta_\xi \gamma(t,x,\xi):= \sum_{k =-1}^{+\infty}v_k 
$$
where $\Delta_k$ denotes the usual Littlewood-Paley  frequency localization (see section 5 below).

If $\mez 2^{k} \geq h^{-\delta} = 2^{j\delta}$ we have $\Delta_k \psi(h^\delta D_x) =0$. 
Therefore
$$
D_x^\alpha D^\beta_\xi \gamma_\delta(t,x,\xi) = \sum_{k =-1}^{2+ [j\delta]} v_k.
$$
Now 
$$
v_k = 2^{k \vert \alpha \vert}  \varphi_1(2^{-k}D_x) \psi(h^\delta D_x) \Delta_k D^\beta_\xi  \gamma(t,x,\xi) 
 $$
where $\varphi_1(\xi)$ is supported in $\{\frac{1}{3} \leq \xi \vert \leq 3\}.$    Therefore,
$$
\Vert v_k\Vert_{L^\infty(\xR^d)} 
\leq 2^{k \vert \alpha \vert} \Vert  {\Delta}_k  D^\beta_\xi  \gamma(t,\cdot,\xi)\Vert_{L^\infty(\xR^d)} 
\leq C2^{k \vert \alpha \vert} 2^{- \frac{k}{2}} 
\Vert D^\beta_\xi \gamma(t, \cdot, \xi)\Vert_{W^{\mez, \infty}(\xR^d)}.
$$
It follows that
$$
\Vert D_x^\alpha D^\beta_\xi \gamma_\delta(t,\cdot,\xi)\Vert _{L^\infty(\xR^d)} 
\leq C  \sum_{k =-1}^{2+ [j\delta]} 2^{k(\vert \alpha \vert - \mez)} \Vert D^\beta_\xi  \gamma (t,\cdot, \xi) \Vert_{W^{\mez, \infty}(\xR^d)}.
$$
Since $\vert \alpha\vert - \mez >0$ we deduce that
$$
\Vert D_x^\alpha D^\beta_\xi  \gamma_\delta(t,\cdot,\xi)\Vert _{L^\infty(\xR^d)} 
\leq C  2^{j\delta(\vert \alpha \vert - \mez)}\Vert D^\beta_\xi \gamma (t,\cdot, \xi) \Vert_{W^{\mez, \infty}(\xR^d)}.
$$
This completes the proof.
\end{proof}

We introduce now the Hessian matrix of $\gamma_\delta,$
\begin{equation}\label{Hess}
\Hess_\xi (\gamma_{\delta}) (t,x,\xi) = \Big( \frac{\partial^2 \gamma_\delta}{\partial \xi_j \partial \xi_k}(t,x,\xi)\Big).
\end{equation}

Our purpose is to prove the following result.
\begin{prop}\label{hess}
There exist $c_0>0, h_0 >0$ such that 
$$
\vert \det   \Hess_\xi (\gamma_{\delta}) (t,x,\xi)\vert \geq c_0 $$
for all $t \in I, x\in \xR^d, \xi \in \mathcal{C}_0, 0< h \leq h_0.$
\end{prop}

With the notation in \eqref{U} we can write
$$
\mU(t,x,\xi) = \langle A(t,x) \xi, \xi \rangle,
$$
where $A(t,x)$ is a symmetric matrix. 
Since we have
\begin{equation}\label{inf}
\vert \xi \vert^2 \leq \mU(t,x,\xi) \leq C(1+ \Vert \nabla \eta \Vert^2_{L^\infty(I \times \xR^d)}) \vert \xi \vert ^2 
\end{equation}
we see that the eigenvalues of $A$ are greater than one;  therefore we have
\begin{equation}\label{detA}
\det A(t,x) \geq 1 \quad \forall (t,x) \in I \times \xR^d.
\end{equation}
We shall need the following lemma.
\begin{lemm}
With $\alpha = \frac{1}{4}$ we have
$$
\vert \det  \Hess_\xi (\gamma)(t,x,\xi) \vert = a^{\frac{d}{2}}(2\alpha)^d\vert 2\alpha -1\vert \det A(t,x) 
\big(\mU(t,x,\xi)\big)^{(\alpha-1)d}.
$$
\end{lemm}

\begin{proof}
Here $t$ and $x$ are fixed parameters.
The matrix $A$ being symmetric one can find an orthogonal matrix $B$ such that 
$B^{-1}AB = D = \text{diag}(\mu_j)$ where the $\mu_j's$ are the eigenvalues of $A.$ Setting 
$C=\text{diag}(\sqrt{\mu_j})$ and $M= CB^{-1}$ we see that $\mU(t,x,\xi) = \vert M\xi \vert^2$ 
which implies that
\begin{equation}
\gamma(t,x,\xi) = a^\mez g(M\xi) \quad \text{where} \quad g(\zeta) = \vert \zeta \vert^{2\alpha},
\end{equation}
so that $\text{Hess}_\xi (\gamma)(t,x,\xi) = a^{\frac{1}{2}}  \,  {}^t\!M \big( \text{Hess}_\zeta( g) (M\xi)\big) M$. 
Since $\vert \det M \vert^2 = \vert \det C \vert^2 = \det A$ we obtain
\begin{equation}\label{Hgamma}
\vert \det  \text{Hess}_\xi (\gamma)(t,x,\xi) \vert  = a^{\frac{d}{2}} \det A(t,x) \, \vert \det \text{Hess}_\zeta (g) (M(t,x)\xi) \vert.
\end{equation}
Now we have,
\begin{equation}\label{Hg}
\frac{\partial^2 g}{\partial \zeta_j \partial \zeta_k} (\zeta) 
= 2\alpha \vert \zeta \vert^{2\alpha -2}\big(\delta_{jk} + 2(\alpha -1)\omega_j \omega_k\big), \quad \omega_j = \frac{\zeta_j}{\vert \zeta \vert}.
\end{equation}
Let us consider the function $F: \xR \to \xR$ defined by
$$F(\lambda) = \det \big( \delta_{jk} + \lambda \omega_j \omega_k\big).$$
It is a polynomial in $\lambda$ and we have
\begin{equation}\label{F}
F(0) = 1.
\end{equation}
Let us denote by $C_j(\lambda)$ the $j^{th}$ column of this determinant. Then
$$
F'(\lambda) = \sum_{k=1}^d \det \big(C_1(\lambda), \ldots, C'_k(\lambda),\ldots C_d(\lambda) \big).
$$
We see easily that
$$
\det \big(C_1(0), \ldots, C'_k(0),\ldots C_d(0) \big) = \omega_j^2
$$
which ensures that
\begin{equation}\label{F'}
F'(0) = 1.
\end{equation}
Now since $C_j(\lambda)$ is linear with respect to $\lambda$ 
we have $C''_j(\lambda) = 0$ therefore
$$
F''(\lambda) = \sum_{j=1}^d \sum_{k=1,k\neq j}^d 
\det \big(C_1(\lambda),\ldots, C'_j(\lambda),\ldots,C'_k(\lambda),\ldots, C_d(\lambda) \big).
$$
Now $C'_j(\lambda) = \omega_j(\omega_1,\ldots,\omega_d)$ 
and $C'_k(\lambda) = \omega_k(\omega_1,\ldots,\omega_d)$. 
It follows that $F''(\lambda) = 0$ for all $\lambda \in \xR.$ 
We deduce from \eqref{F}, \eqref{F'} that $F(\lambda) = 1 + \lambda$ and from \eqref{Hg} that
$$
\det \text{Hess}_\zeta(g)(\zeta) = \big(2\alpha \vert \zeta \vert^{2\alpha -2}\big)^d (2\alpha -1).
$$
The lemma follows then from \eqref{Hgamma} since $\mU(t,x,\xi) = \vert M(t,x)\xi \vert^2.$
\end{proof}

\begin{coro}\label{hess1}
One can find $c_0 >0$ such that  
$$
\vert \det  \Hess_\xi (\gamma)(t,x,\xi) \vert \geq c_0,
$$
for all $t\in I, x \in \xR^d, \xi \in \mathcal{C}_0.$ 
\end{coro}
\begin{proof}
This follows from the preceeding lemma and from \eqref{inf}, \eqref{detA}.
\end{proof}
\begin{proof}[Proof of Proposition \ref{hess}]
 
Recall that  we have   for all $\alpha \in \xN^d$
$$\sup_{t\in I} \sup _{\vert \xi \vert \leq 2} \Vert D^\alpha_\xi \gamma (t,\cdot,\xi)\Vert_{W^{s_0, \infty}(\xR^d)} < +\infty.$$
For fixed $j,k\in \{1,\ldots,d\}$ we write
\begin{equation}\label{perturb}
\frac{\partial^2 \gamma_\delta}{\partial \xi_j \partial \xi_k}(t,x,\xi) 
= \frac{\partial^2 \gamma}{\partial \xi_j \partial \xi_k}(t,x,\xi) 
- (I-\psi(h^\delta D_x))\frac{\partial^2 \gamma}{\partial \xi_j \partial \xi_k}(t,x,\xi).
\end{equation}
Setting $\gamma_{jk}= \frac{\partial^2 \gamma}{\partial \xi_j \partial \xi_k}$ 
we  have, since $\psi(0) =1$, 
$$
(I-\psi(h^\delta D_x))\gamma_{jk}(t,x,\xi) 
= h^{-\delta d}\int_{\xR^d} \overline{\mathcal{F}}\psi \Big(\frac{y}{h^\delta}\Big)
\big[\gamma_{jk}(t,x,\xi) - \gamma_{jk}(t,x-y,\xi)\big] dy,
$$
where $\overline{\mathcal{F}}$ denotes the inverse Fourier transform with respect to $x$. 
Then 
\begin{equation*}
\begin{aligned}
&\Vert (I-\psi(h^\delta D_x))\gamma_{jk}(t,\cdot,\xi)  \Vert_{L^\infty (\xR^d)} \\
&\qquad \leq h^{-\delta d}\Big( \int_{\xR^d} \vert \overline{\mathcal{F}}\psi 
\big(\frac{y}{h^\delta}\big)\vert  \vert y \vert^{s_0} dy\Big) \Vert  \gamma_{jk}(t,\cdot,\xi) \Vert_{W^{s_0,\infty}(\xR^d)}\\
&\qquad \leq C h^{\delta s_0} \sup_{t\in I} \sup_{\vert \xi \vert \leq 2}
\Vert \gamma_{jk}(t,\cdot,\xi)\Vert_{W^{s_0,\infty}(\xR^d)}. 
\end{aligned}
\end{equation*}
Then Proposition \ref{hess} follows from Corollary \ref{hess1} and \eqref{perturb} 
if $h_0$ is small enough.
\end{proof}

\section{The pseudo-differential symbol}

In this paragraph, we study the pseudo-differential symbol of $T_{\gamma_\delta}$

Let $\chi \in C_0^\infty(\xR^d \times \xR^d) $ be such that
\begin{equation}\label{chi}
\begin{aligned}
&(i)\,  &&\chi(-\zeta, \eta) = \chi(\zeta,\eta),\\
&(ii)\,  &&\chi \text{ is homogeneous of order zero }, \\
&(iii)\, && \chi(\zeta, \eta) =1 \text{ if } \vert \zeta \vert \leq \eps_1 \vert \eta \vert,  
\quad\chi(\zeta, \eta) =0 \text{ if } \vert \zeta \vert \geq \eps_2 \vert \eta \vert 
\end{aligned}
\end{equation}
where $0<\eps_1 <\eps_2$ are small constants. 
Then modulo an operator of order zero we have
\begin{equation}\label{pseudo}
T_{\gamma_\delta}u  = \sigma_{\gamma_\delta}(t,x, D_x)u 
\end{equation}
with
$$
\sigma_{\gamma_\delta}(t,x, \eta) = \Big( \int_{\xR^d} 
e^{i x\cdot \zeta} \chi(\zeta, \eta) \widehat{\gamma}_\delta(t, \zeta, \eta) \,d\zeta \Big) \psi_0(\eta)
$$
where $\widehat{\gamma}_\delta$ denotes the Fourier transform of $\gamma_\delta$ 
with respect to  the variable $x$ and $\psi_0$ is a cut-off function such that $\psi_0(\eta) = 0$ 
for $\vert \eta \vert \leq \frac{1}{4}, \psi_0(\eta) = 1$  for $\vert \eta \vert \geq \frac{1}{3}.$ (Notice that $\chi$ can be so chosen that 
Remark $2.2$ in \cite{ABZ3} remains true.)

We shall use in the sequel the following result.
\begin{lemm}\label{l0.6}
For all $N\in \xN, M>0$ and all $\alpha, \beta \in \xN^d$ there exists $C  >0$ 
such that
$$\sup_{\vert \eta \vert \leq M} \vert D_\mu^\alpha 
D_\eta^\beta \widehat{\chi}(\mu, \eta)\vert \leq \frac{C}{\langle \mu \rangle^N}$$
where the Fourier transform is taken with respect to the first variable of $\chi$.
\end{lemm}
\begin{proof}
Recall that we have $\supp \chi \subset \{(\zeta, \eta): \vert \zeta \vert \leq \eps_2 \vert \eta \vert\}.$ Then for $\vert \eta \vert \leq M$ we see easily that for any $k\in \xN$ we can write
$$
\vert \mu\vert^{2k} D_\mu^\alpha D_\eta^\beta \widehat{\chi}(\mu, \eta) 
= \int_{\xR^d} e^{-i \mu\cdot \zeta} 
(-\Delta_\zeta)^k\big [(-\zeta)^\alpha D_\eta^\beta \chi(\zeta, \eta)\big] \,d\zeta.
$$
The result follows then from the fact that when $\vert \eta \vert \leq M$ 
the domain of integration is contained in the set $\{\zeta : \vert \zeta \vert\leq \eps_2 M\}.$
\end{proof}
\section{Frequency localization}
Let us recall the Littlewood-Paley decomposition. 
There exists a function $\psi\in C^\infty_0(\xR^d)$ such that 
$\psi(\xi)=1$ for $\la \xi\ra \le 1/2$ and $\psi(\xi)=0$ for $\la \xi\ra\ge 1$, and a function 
$\varphi\in C^\infty_0(\xR^d)$ whose support is contained 
in the shell $\mathcal{C}_0 := \{\xi : \mez \leq \vert \xi \vert \leq2\}$ such that
$$
\psi(\xi)+\sum_{k=0}^{N-1}\varphi(2^{-k}\xi)=\psi(2^{-N}\xi).
$$
We set
$$
\Delta_{-1}u=\psi(D)u,\quad 
\Delta_j u = \varphi(2^{-j}D)u,\quad S_j u =\sum_{k=-1}^{j-1} \Delta_k u=\psi(2^{-j}D)u.
$$

Let us consider the operator
\begin{equation*}
L = \partial_t +\mez(T_V\cdot \nabla + \nabla \cdot T_V)+i T_\gamma
\end{equation*}
where
$$
V = (V_1,\ldots,V_d), \quad V_j \in C^0(I; H^s(\xR^d))\cap L^2(I; W^{1,\infty}(\xR^d)).
$$
  
Let $u \in L^\infty(I; H^s(\xR^d))$ be such that
$$
Lu = f \in L^2(I;H^s(\xR^d)).
$$
Then we have
\begin{equation}\label{commuta}
 L \Delta_j u = \Delta_j f + \mez \big( [T_V, \Delta_j]\cdot \nabla u +\nabla\cdot  [T_V, \Delta_j] u\big) +i [T_\gamma, \Delta_j]u.
\end{equation}
Now we have the following lemma.
\begin{lemm}\label{TV-S}
\begin{align*}
 T_V \cdot \nabla \Delta_j u&= S_j(V)\cdot\nabla \Delta_ju + R_ju\\
 \nabla \cdot T_V \Delta_j u= &= \nabla \cdot S_j(V) \Delta_j u + R'_j u
 \end{align*}
where $R_ju, R'_ju$ have their  spectrum contained in a ball $B(0, C2^j)$ and satisfies the estimate
$$
\Vert R_ju \Vert_{H^s(\xR^d)} + \Vert R'_ju \Vert_{H^s(\xR^d)} \leq   C\Vert V \Vert_{W^{1,\infty}(\xR^d)} \Vert u \Vert_{H^s(\xR^d)}.
$$
\end{lemm}
\begin{proof}
It is sufficient to  prove the first line. Indeed we have
$ \nabla \cdot T_V \Delta_j u = T_V \cdot \nabla \Delta_j u + T_{\text{div}V} \Delta_j u$ and $\Vert T_{\text{div}V} \Delta_j u\Vert_{H^s(\xR^d)} \leq C \Vert V \Vert_{W^{1,\infty}(\xR^d)} \Vert u \Vert_{H^s(\xR^d)}.
.$

Since $\Delta_k \Delta_j = 0$ if $\vert k-j\vert \geq 2$ we can write
 \begin{align*}
 T_V \cdot \Delta_j \nabla u  &= \sum_{\vert j-k \vert \leq 1} S_{k-3}(V) \cdot \Delta_k \Delta_j \nabla u = S_j(V)\cdot\sum_{\vert j-k \vert \leq 1}   \Delta_k \Delta_j \nabla u + R_ju\\
 &   = S_j(V)\Delta_j \nabla u +R_j u, 
  \end{align*}
  where $R_j u = \sum_{\vert j-k \vert \leq 1} \big( S_{k-3}(V)-S_j(V) \big)  \cdot \Delta_k \Delta_j \nabla u.$ We have three terms in the sum defining $R_ju$.  Each of them is a finite sum of terms of the form $A_j = \Delta_{j+\mu}(V)\cdot   \Delta_{j+ \nu} \Delta_j \nabla u.$ Since the spectrum of $A_j$ is contained in a ball of radius $C2^j$ we can write
  \begin{align*}
  \Vert A_j \Vert_{H^s(\xR^d)} &\leq C_1 2^{js}\Vert F_j \Vert_{L^2(\xR^d)} \leq C_1 2^{js}\Vert \Delta_{j+\mu}V\Vert_{L^\infty(\xR^d)} \Vert \Delta_{j }\nabla u\Vert_{L^2(\xR^d)} \\
  &\leq C_2 2^{js} 2^{-j} \Vert V \Vert_{W^{1,\infty}(\xR^d)} 2^{-js} 2^j \Vert u \Vert_{H^s(\xR^d)} \leq C_2\Vert V \Vert_{W^{1,\infty}(\xR^d)}  \Vert u \Vert_{H^s(\xR^d)}
  \end{align*}
  which completes the proof of the lemma.
\end{proof}
It follows from the lemma that we have 
\begin{multline}
 \big(\partial_t +\mez( S_j(V)\cdot \nabla + \nabla  \cdot  S_j(V))+i T_\gamma\big) \Delta_j u =\\
   \Delta_j f+\mez \big( [T_V, \Delta_j]\cdot \nabla u +\nabla\cdot  [T_V, \Delta_j] u\big) +i [T_\gamma, \Delta_j]u +R_j u + R'_ju
 \end{multline}
 As already mentioned, we need to smooth out the symbols. To do so, 
we replace 
$S_j(V)$ by $S_{\delta j}(V)=\psi(2^{-j\delta}D)V$ and $\gamma$ by $\gamma_\delta = \psi(2^{-j\delta}D)\gamma.$  We set
\begin{equation}\label{Ldelta}
L_\delta = \partial_t +\mez( S_{j\delta}(V)\cdot \nabla + \nabla  \cdot  S_{\delta j}(V)) +i T_{\gamma_\delta}.
\end{equation}
Then we have
\begin{equation}\label{Ldeltau}
L_\delta \Delta_ju = F_j
\end{equation}
where 
  \begin{multline}\label{Fj}
F_j = \Delta_jf +[T_V, \Delta_j]\cdot \nabla u + i [T_\gamma, \Delta_j]u +R_j u + R'_ju + \\ \mez \big\{ \big(S_{j\delta}(V)  - S_j(V)\big)\cdot\nabla \Delta_ju +\nabla \cdot \big(S_{j\delta}(V)  - S_j(V)\big) \Delta_ju\big\}
  +i\big(T_{\gamma_\delta} - T_\gamma \big)\Delta_j u.
\end{multline}
  Then, for all fixed $t,$ the function $F_j(t,\cdot)$ 
has its spectrum contained in a ball $B(0, C2^j)$.  

\section{Straightening the vector field}

We want to straighten the vector field $\partial_t + S_{j\delta}(V)\cdot \nabla.$
Consider the system of differential equations
\begin{equation}\label{eqdif}
\left \{
\begin{aligned}
&\dot{X}_k(s) =  S_{j\delta}(V_k)(s,X(s)), \quad 1 \leq k \leq d,\quad X=(X_1,\ldots,X_d) \\
&X_k(0) =x_k.
\end{aligned}
\right.
\end{equation}
For $k=1,\ldots,d$ we have $S_{j\delta}(V_k) \in L^\infty(I;H^\infty(\xR^d))$ and
$$
\vert S_{j\delta}(V_k)(s,x) \vert \leq C\Vert V_k \Vert_{L^\infty(I\times \xR^d)} \quad   \forall (s,x) \in I \times \xR^d.
$$
Therefore System \eqref{eqdif} has a unique solution defined on $I$ 
which will be denoted $X(s; x,h) (h=2^{-j})$ or sometimes simply $X(s).$

 We shall set 
 \begin{equation}\label{E0}
 E_0 = L^p(I; W^{1,\infty}(\xR^d))^d
 \end{equation}
  where $p= 4$ if $d=1$, $p=2$ if $d \geq 2$, endowed with its natural norm.
\begin{prop}\label{estX}
For fixed $(s,h)$ the map $x \mapsto X(s;x,h)$ 
belongs to $C^\infty (\xR^d, \xR^d).$ Moreover there exist functions $\mathcal{F}, \mathcal{F}_\alpha: \xR^+ \to \xR^+$ such that  
\begin{equation*}
\begin{aligned}
&(i) \quad &&\Big\Vert \frac{\partial X}{\partial x}(s; \cdot, h) - Id 
\Big\Vert_{L^\infty(\xR^d)} \leq \mathcal{F}(\Vert V \Vert_{E_0})\vert s \vert^\mez, \\  
&(ii)\quad  &&\Vert (\partial_x^\alpha X) (s; \cdot, h)  \Vert_{L^\infty(\xR^d)} 
\leq \mathcal{F}_\alpha(\Vert V \Vert_{E_0}) h^{-\delta(\vert \alpha \vert -1)} \vert s \vert^\mez, \quad    \vert \alpha \vert \geq 2 \\
\end{aligned}
\end{equation*}
for   all $(s,h) \in I\times (0, h_0]$.
\end{prop}
\begin{proof}
To prove $(i)$ we differentiate the system \label{eqdiff} with respect to $x_l$ and we obtain
\begin{equation*}
\left\{
\begin{aligned}
&\dot{\frac{\partial X_k}{\partial x_l}}(s) = 
 \sum_{q = 1}^d S_{j\delta}\Big (\frac{\partial V_k}{\partial x_q}\Big)(s,X(s))\frac{\partial X_q}{\partial x_l}(s)\\
&\frac{\partial X_k}{\partial x_l}(0) = \delta_{kl}
\end{aligned}
\right. 
\end{equation*}
from which we deduce 
\begin{equation}\label{dX}
\frac{\partial X_k}{\partial x_l}(s) = \delta_{kl} 
+ \int_0^s  \sum_{q = 1}^d S_{j\delta}
\Big (\frac{\partial V_k}{\partial x_q}\Big)(\sigma,X(\sigma))\frac{\partial X_q}{\partial x_l}(\sigma) \, d \sigma.
\end{equation}
Setting 
$\vert \nabla X\vert = \sum_{k,l =1}^d \vert \frac{\partial X_k}{\partial x_l} \vert$ 
we obtain from \eqref{dX}
$$
\vert \nabla X(s) \vert \leq C_d + \int_0^s \vert \nabla V(\sigma, X(\sigma)) \vert  \, \vert \nabla X(\sigma) \vert \, d \sigma.
$$
The Gronwall inequality implies that
\begin{equation}\label{Gron}
 \vert \nabla X(s) \vert \leq \mathcal{F} (\Vert V \Vert_{E_0}) \quad \forall s \in I. 
 \end{equation}
Coming back to \eqref{dX} and using \eqref{Gron} we can write
$$
\la \frac{\partial X }{\partial x}(s) - Id \ra \leq  \mathcal{F} (\Vert V \Vert_{E_0}) \int_0^s \Vert \nabla V(\sigma, \cdot)  \Vert_{L^\infty(\xR^d)} 
\, d\sigma \leq \mathcal{F}_1 (\Vert V \Vert_{E_0})\vert s \vert^\mez.
$$
Notice that in dimension one we have used the inequality $\vert s \vert^{\frac{3}{4}} \leq C\vert s \vert^\mez $ when $s\in I.$ 

To prove $(ii)$ we shall show by induction on $\vert \alpha \vert$ that the estimate
$$
\Vert(\partial_x^\alpha X)(s;\cdot, h)\Vert_{L^\infty(\xR^d)} \leq \mathcal{F}_\alpha(\Vert V \Vert_{E_0}) h^{-\delta(\vert \alpha \vert -1)}
$$
for  $1 \leq \vert \alpha \vert \leq k$ implies $(ii)$ for  $\vert \alpha \vert = k+1.$ 
The above estimate is true for  $\vert \alpha \vert =1$ by $(i)$. 
Let us differentiate  $\vert \alpha \vert$ times the system \eqref{eqdif}. 
We obtain
\begin{equation}\label{rec1}
 \frac{d}{ds}\big(\partial_x^\alpha X\big)(s) = S_{j\delta}(\nabla V)(s,X(s)) \partial_x^\alpha X +(1) 
 \end{equation}
where the term $(1)$ is a finite linear combination of terms of the form
$$
A_\beta (s,x) = \partial_x^\beta \big(S_{j\delta}(V)\big)(s,X(s)) \prod_{i=1}^q \big(\partial_x^{L_i} X(s)\big)^{K_i}
$$
where 
$$
2 \leq \vert \beta \vert \leq \vert \alpha \vert, \quad  1 
\leq q \leq \vert \alpha \vert, \quad \sum_{i=1}^q \vert K_i \vert L_i = \alpha, \quad \sum_{i=1}^q K_i = \beta.
$$
Then  $1 \leq \vert L_i \vert < \vert \alpha \vert$ 
which allows us to use the induction. Therefore we can write
\begin{equation*}
\begin{aligned}
\Vert A_\beta (s,\cdot) \Vert_{L^\infty(\xR^d)} 
&\leq \big\Vert \partial_x^\beta \big(S_{j\delta}(V)\big)(s,\cdot)\big\Vert_{L^\infty(\xR^d)} \prod_{i=1}^q 
\Big\Vert \partial_x^{L_i} X(s, \cdot)\Big\Vert _{L^\infty(\xR^d)}^{\vert K_i \vert}\\[1ex]
&\leq C h^{-\delta (\vert \beta \vert -1)}\Vert V(s,\cdot) \Vert_{W^{1, \infty}(\xR^d)} 
h^{-\delta \sum_{i=1}^q \vert K_i\vert(\vert L_i \vert -1)}\mathcal{F} (\Vert V \Vert_{E_0}) \\[1ex]
&\leq \mathcal{F} (\Vert V \Vert_{E_0})h^{-\delta(\vert \alpha \vert -1)}\Vert V(s,\cdot) \Vert_{W^{1, \infty}(\xR^d)}.
\end{aligned}
\end{equation*}
It follows then from \eqref{rec1} that
\begin{equation}
\begin{aligned}
\vert \partial_x^\alpha X(s) \vert &\leq   \mathcal{F} (\Vert V \Vert_{E_0})h^{-\delta(\vert \alpha \vert -1)} 
\int_0^s\Vert V(\sigma,\cdot) \Vert_{W^{1, \infty}(\xR^d)} \,d\sigma\\
&\quad + C \int_0^s\Vert V(\sigma,\cdot) \Vert_{W^{1, \infty} \vert (\xR^d)}\vert \partial_x^\alpha X(\sigma) \vert \, d\sigma.
\end{aligned}
\end{equation}
The H\" older and Gronwall inequalities imply immediately $(ii)$.
\end{proof}

\begin{coro}
There exist  $T_0>0,h_0>0 $ such that for $t\in [0,T_0]$ and $0<h \leq h_0$ the map $x\mapsto X(t; x,h)$ 
from $\xR^d$ to $\xR^d$ is a $C^\infty$ diffeomorphism.
\end{coro}
\begin{proof}
This follows from a result by Hadamard (see~\cite{Be}). 
Indeed if $T_0$ is small enough, Proposition~\ref{estX} 
shows that the matrix $\big( \frac{\partial X_k}{\partial x_j}(t; x,h) \big)$ 
is invertible. On the other hand since
$$
\vert X(t; x,h) - x \vert \leq
\int_0^t \vert S_{j\delta}(V)(\sigma, X(\sigma))\vert \, d \sigma 
\leq T_0 \Vert V \Vert_{L^\infty([0,T_0]\times \xR^d)}
$$
we see that the map  $x\mapsto X(t; x,h)$ is proper.
\end{proof}
  \section{Reduction to a semi-classical form}
Now according to \eqref{pseudo} and \eqref{Ldeltau} we see that the function 
$ {U}\defn \Delta_ju$ is a solution of the equation 
$$ \big( \partial_t +\mez \big\{S_{j\delta}(V)\cdot \nabla + \nabla  \cdot S_{j\delta}(V)\big\} +i  \sigma_{\gamma_\delta}(t,x, D_x)\varphi_1(hD_x)\big)
 {U}(t,x) =F_j(t,x), \quad h=2^{-j}
$$
where 
\begin{equation}\label{phi1}
\varphi_1 \in C^\infty(\xR^d), \quad   
 \supp \varphi_1 \subset \{ \xi: \frac{1}{4} \leq \vert \xi \vert \leq 4\}, \quad  \varphi_1 = 1 \text{ on }  \{ \xi: \frac{1}{3} \leq \vert \xi \vert \leq 3\}  
\end{equation}
and $F_j$ has been defined in \eqref{Fj}. Notice that 
$$\mez \big(S_{j\delta}(V)\cdot \nabla + \nabla  \cdot S_{j\delta}(V)\big\} = S_{j\delta}(V)\cdot \nabla + \mez S_{j\delta}(\text{div}V).$$
We shall set 
\begin{equation}\label{a}
a(t,x,h,\xi) =  \sigma_{\gamma_\delta}(t,x, \xi)\varphi_1(h\xi).
\end{equation}
We make now the change of variable $x = X(t;y,h).$ Let us set
\begin{equation}
v_h(t,y) = {U}(t,X(t;y,h)) \quad t \in [0,T_0].
\end{equation}
Then it follows from \eqref{eqdif} that
\begin{equation}\label{Dv}
\partial_t v_h(t,y) = -i \big( a(t,x,h,D_x) {U} \big) (t, X(t;y,h)) + F_j(t,X(t;y,h)).
\end{equation}
Our next purpose is to give another expression to the quantity
\begin{equation}\label{A}
A= \big(a(t,x,h,D_x) {U} \big) (t, X(t;y,h)).
\end{equation}
In  the computation below, $t\in [0,T_0]$ and $h$ being  fixed, we will omit them. We have 
$$
A = (2\pi)^{-d} \iint e^{i(X(y)-x')\cdot \eta } a(X(y), \eta) {U}(x') \,dx'\, d\eta.
$$
{\bf Notations}: we set
\begin{equation}\label{notation}
\begin{aligned}
H(y,y') &= \int_0^1 \frac{\partial X}{\partial x}(\lambda y +(1-\lambda) y') \, d\lambda, 
\quad M(y,y') =\big(\, {}\!^tH(y,y') \big)^{-1} \\
M_0(y)  &= \Big(\, {}\!^t  \Big(\frac{\partial X}{\partial x}(y)\Big)  \Big)^{-1}, 
\quad  J(y,y') = \Big\vert \det \Big(\frac{\partial X}{\partial x}(y')\Big)\Big\vert \vert \det M(y,y') \vert.
\end{aligned}
\end{equation}
Let us remark that $M,M_0$ are well defined by Proposition~\ref{estX}. 
Moreover $M_0(y) = M(y,y)$ and $J(y,y) =1.$

Now, in the integral defining $A,$ we make the change of variables $x' =X(y')$.  
Then using the equality $X(y) - X(y') = H(y,y')(y-y')$ and setting $\eta = M(y,y')\zeta$ we obtain,
$$
A = (2\pi)^{-d} \iint e^{i (y-y')\cdot \zeta } a\big(X(y), M(y,y')\zeta\big) J(y,y') v_h(y') \,dy' \,d\zeta.
$$
Now we set
\begin{equation}\label{htilde}
z = h^{-\mez} y, \quad w_h(z) = v_h(h^\mez z), \quad \htilde  = h^\mez.
\end{equation}
Then
$$A = (2\pi)^{-d} \iint e^{i (\htilde z-y')\cdot \zeta } 
a\big(X(\htilde z), M\big(\htilde z,y'\big)\zeta\big) J\big(\htilde z,y'\big) v_h(y') \, dy' d\zeta.$$
Then setting $y' = \htilde z'$ and $\htilde \zeta = \zeta'$ 
we obtain
\begin{equation}\label{A1}
 A = (2\pi)^{-d} \iint e^{i (z-z')\cdot \zeta'  } 
 a\big(X(\htilde x), M\big(\htilde z,\htilde z'\big)\htilde ^{-1}\zeta'\big)
 J\big(\htilde z,\htilde z'\big) w_h(z') \,dz' \,d\zeta'.
 \end{equation}
Our aim is to reduce ourselves to a semi-classical form, after multiplying 
the equation by $\htilde .$ However this not straightforward since the symbol $a$ 
is not homogeneous in $\xi$ although $\gamma$ is homogeneous of order $\mez.$ We proceed as follows.
 
First of all on the support of the function $\varphi_1$ 
in the definition of $a$ (see \eqref{a}) the function $\psi_0$ 
appearing in the definition of $\sigma_{\gamma_\delta}$ (see \eqref{pseudo}) is equal to one. 
 
Therefore we can write for $X\in \xR^d, \rho\in \xR^d$,  (skipping the variable $t$),
\begin{equation*} 
\begin{aligned}
\sigma_{\gamma_\delta}(X, \htilde ^{-1}\rho) 
&=\int e^{iX\cdot \zeta}\chi\big(\zeta, \htilde ^{-1}\rho\big)
\widehat{\gamma}_\delta(\zeta, \htilde ^{-1}\rho) \, d \zeta\\
&= \iint e^{i(X-y)\cdot\zeta} \chi
\big(\zeta, \htilde ^{-1}\rho\big) {\gamma}_\delta(y, \htilde ^{-1}\rho)\,dy \,d\zeta\\
 &=\int \widehat{\chi}\big(\mu, \htilde ^{-1}\rho\big){\gamma}_\delta
 \big(X-\mu, \htilde ^{-1}\rho\big)\,d\mu.
\end{aligned}
\end{equation*}
Now since $\chi$ is homogeneous of degree zero we have,
$$\widehat{\chi}(\mu, \lambda \eta) = \lambda^d \widehat \chi(\lambda \mu, \eta), $$
which follows from the fact that $\chi(\zeta, \lambda \eta) = \chi(\lambda\lambda^{-1}\zeta, \lambda \eta) = \chi(\lambda^{-1}\zeta,   \eta).$

Applying this equality with $\lambda = \htilde ^{-2}$ and $\eta = \htilde  \rho$ we obtain,
\begin{equation*} 
\begin{aligned}
 \sigma_{\gamma_\delta}(X, \htilde ^{-1}\rho)
 &= \htilde ^{-2d}\int \widehat{\chi}(\htilde ^{-2} \mu, \htilde  \rho){\gamma}_\delta(X-\mu, \htilde ^{-1}\rho)\,d\mu\\
 &= \int \widehat{\chi}(\mu' , \htilde  \rho){\gamma}_\delta(X-\htilde ^{2}\mu'  , \htilde ^{-1}\rho)
 \,d\mu'. 
 \end{aligned}
\end{equation*}
Using the fact that $\gamma$ and $\gamma_\delta$ are  homogeneous of order $\mez$ in $\xi$  we obtain 
$$ \htilde  \sigma_{\gamma_\delta}(X, \htilde ^{-1}\rho) 
= \int \widehat{\chi}(\mu  , \htilde  \rho){\gamma}_\delta(X-\htilde ^{2}\mu   , \htilde  \rho)\, d\mu
$$ 
and since $\htilde ^{-1}h=\htilde $, we obtain
$$
\htilde  a(X,\htilde ^{-1}\rho) =  \Big(\int \widehat{\chi}(\mu  , \htilde  \rho){\gamma}_\delta(X-\htilde ^{2}\mu,\htilde\rho)\,d\mu\Big) 
\varphi_1(\htilde \rho).
$$
It follows then from \eqref{Dv},\eqref{A},\eqref{A1} that the function $w_h$ defined in \eqref{htilde} 
is solution of the equation
\begin{equation}\label{Eqfin}
(\htilde \partial_t + \htilde c+iP)w_h(t,z) = \htilde F_j(t,X(t,\htilde z,h))
\end{equation}
where $c(t,z, \htilde) =  \mez S_{j\delta}(\text{div}V)(t, X(t,\htilde z))$ and 
\begin{equation}\label{operateur}
Pw(t,z) = (2\pi \htilde )^{-d}\iint e^{i{\htilde}^{-1}(z-z')\cdot\zeta}
\widetilde{p}(t,z,z',\zeta,\htilde )w(t,z') \, dz' \,d\zeta
\end{equation}
with
\begin{equation}\label{ptilde}
\begin{aligned}
 \widetilde{p}(t,z,z',\zeta,\htilde ) =    \int \widehat{\chi}\big(\mu  , &M(t,\htilde z,\htilde z')\zeta\big)
 {\gamma}_\delta\big (t, X(t,\htilde z)-\htilde ^{2}\mu, M(t,\htilde z,\htilde z')\zeta\big)d\mu \\
 & \times \varphi_1( M(t,\htilde z,\htilde z')\zeta)J(t, \htilde z,\htilde z').
 \end{aligned}
\end{equation}
We shall set in what follows
\begin{equation*}
p(t,z,\zeta,\htilde ) = \widetilde{p}(t,z,z,\zeta,\htilde ).
\end{equation*}
Since $M(t,\htilde z,\htilde z)= M_0(t,\htilde z)$ and $J(t,\htilde z,\htilde z) = 1$ we obtain
\begin{equation}\label{encorep}
p(t,z,\zeta,\htilde ) =    \int \widehat{\chi}\big(\mu  , M_0(t, \htilde z)\zeta\big)
{\gamma}_\delta\big (t, X(t,\htilde z)-\htilde ^{2}\mu, M_0(t,\htilde z) \zeta\big)\,d\mu \\
\cdot\varphi_1( M_0(t,\htilde z)\zeta).
\end{equation}

Notice that since the function $\chi$ is even with respect to its first variable the symbol $p$ is real.

Summing up we have proved that
\begin{equation}\label{semiclass}
   \htilde \big\{ \partial_t +\mez \big( S_{j\delta}(V)\cdot \nabla  + \nabla  \cdot S_{j\delta}(V)\big) +i  T_{\gamma_\delta} \big\}{U}(t,x) = \big(\htilde \partial_t + \htilde c + iP\big)w_h(t,z)
  \end{equation}
where 
\begin{equation}\label{c=}
x = X(t, \htilde z), \quad c(t,z, \htilde) =  \mez S_{j\delta}(\text{div}V)(t, X(t,\htilde z)), \quad w_h(t,z) = U(t, X(t, \htilde z)).
\end{equation}

 \subsection{ Estimates on the pseudo-differential symbol} 
 Let $I_{\htilde }:=[0,\htilde ^{ \delta}]$. We introduce   norms on the para-differential symbol $\gamma.$ For $k \in \xN$ we set
  \begin{equation}\label{norme:gamma}
\mathcal{N}_k(\gamma) = \sum_{\vert \beta  \vert \leq k} \Vert D_\xi^\beta \gamma\Vert_{L^\infty(I_{\htilde}\times \xR^d \times \mathcal{C}_3)}.  
\end{equation}

We   estimate now the derivatives of the  symbol of the operator  appearing in the right hand side of \eqref{semiclass}. Recall that $E_0$ is defined in \eqref{E0}.
\begin{lemm}\label{est:c}
For any $\alpha \in \xN^d$ there exists $\mathcal{F}_\alpha: \xR^+ \to \xR^+$ such that for $t\in I_{\htilde}$
$$\Vert (D_z^\alpha c)(t, \cdot)\Vert_{L^\infty(\xR^d)}  \leq \mathcal{F}_\alpha (\Vert V \Vert_{E_0}) \htilde^{\vert \alpha \vert (1-2 \delta)}\Vert V(t, \cdot) \Vert_{W^{1, \infty}(\xR^d)}.$$
\end{lemm}
\begin{proof}
By the Faa-di-Bruno formula $D_z^\alpha c$ is a finite linear combination of terms of the form 
\begin{equation}\label{(1)=}
 (1) = \htilde^{\vert \alpha \vert} D_x^a \big[S_{j\delta}(\text{div}V)\big](t, X(t,\htilde z)) \prod_{j=1}^r \Big(\big(D_x^{l_j}X\big) (t,\htilde z)\Big)^{p_j} 
 \end{equation}
where $1\leq \vert a \vert \leq \vert \alpha \vert, \quad \vert l_j \vert \geq 1, \quad \sum_{j=1}^r \vert p_j\vert l_j = \alpha, \quad \sum_{j=1}^r p_j = a.$ Now for fixed $t$ we have 
\begin{equation}\label{est:Sjdelta}
 \vert D_x^a \big[S_{j\delta}(\text{div}V)\big](t, \cdot)\vert \leq C_{\alpha} \htilde^{-2 \delta \vert a \vert}\Vert V(t, \cdot) \Vert_{W^{1, \infty}(\xR^d)}
 \end{equation}
and by Proposition \ref{estX}
$$\vert (D_x^{l_j}X\big) (t, \cdot)\vert \leq   \mathcal{F}_\alpha (\Vert V \Vert_{E_0}) \htilde^{-2\delta (\vert l_j \vert -1)}.$$
Then the product appearing in the term $(1)$ is bounded by $ \mathcal{F}_\alpha (\Vert V \Vert_{E_0}) \htilde ^M$  where $ M= -2 \delta\sum_{j=1}^r \vert p_j \vert  (\vert l_j \vert -1)  = -2 \delta \vert \alpha \vert + 2 \delta \vert a \vert$. Then the lemma follows from \eqref{(1)=} and \eqref{est:Sjdelta}.

\end{proof}

\begin{lemm}\label{derp}
For every $k \in \xN$   there exist $ \mathcal{F}_{k}: \xR^+ \to \xR^+$ such that,  
 $$\vert  D_z ^\alpha D_\zeta^\beta p(t,z,\zeta,\htilde )\vert \leq  \mathcal{F}_{ k}(\Vert V \Vert_{E_0})\sum_{\vert a \vert \leq   k}  \sup_{\xi \in \mathcal{C}_3} \Vert D^a_\xi \gamma(t,\cdot,\xi)\Vert_{L^{\infty}(\xR^d)} \htilde^{\vert \alpha \vert(1-2 \delta)} 
 $$
 for   all $\vert \alpha\vert + \vert \beta \vert \leq k$  and all $(t,z, \zeta,\htilde) \in I_{\htilde }\times \xR^d \times  \mathcal{C}_1\times (0, \htilde _0]$.
 \end{lemm}
 \begin{coro}\label{derp'}
For every $k \in \xN$   there exist $ \mathcal{F}_{k}: \xR^+ \to \xR^+$ such that, 
$$\int_0^s \vert  D_z ^\alpha D_\zeta^\beta p(t,z,\zeta,\htilde )\vert \, dt \leq   \mathcal{F}_{ k}(\Vert V \Vert_{E_0})\,  \mathcal{N}_k(\gamma) \, \htilde^{\vert \alpha \vert(1-2 \delta) + \delta} 
 $$
  for all $\vert \alpha\vert + \vert \beta \vert \leq k$  and all $(s,z, \zeta,\htilde) \in I_{\htilde }\times \xR^d \times  \mathcal{C}_1\times (0, \htilde _0]$.
 \end{coro}

 \begin{proof}[Proof of Lemma \ref{derp}]
  Here $t$ is considered as a  fixed parameter which will be skipped, keeping in mind that the estimates should be uniform with respect to $t \in [0, \htilde^\delta].$
On the other hand we recall that, by Proposition \ref{estX} and Lemma \ref{estgamma-delta}, we have (since $h= \widetilde{h}^2$)
  \begin{align}
 & \vert D_x^\alpha X(x) \vert \leq  \mathcal{F}_\alpha(\Vert V \Vert_{E_0})\, \widetilde{h}^{-2 \delta(\vert \alpha \vert -1)}, \quad \vert \alpha \vert \geq 1, \beta \in \xN^d \label{est2} \\
 & \vert  D_x^\alpha D_\xi^\beta \gamma_\delta( x,\xi)\vert \leq C_{\alpha, \beta}\,\htilde^{-2\delta \vert \alpha \vert}  \Vert D_\xi^\beta \gamma( \cdot, \xi) \Vert_{L^\infty(\xR^d)}, \quad \alpha, \beta \in \xN^d.\label{est3} 
 \end{align}
Set
$$
F(\mu, z, \zeta, \htilde) = \widehat{\chi}\big(\mu  , M_0(z)\zeta\big) \varphi_1( M_0(z)\zeta){\gamma}_\delta\big ( X(z)-\htilde ^{2}\mu, M_0(z) \zeta\big),  
$$
the lemma will follow immediately from the fact that for every $N \in \xN$ we have
\begin{equation}\label{est:F}
  \vert  D_z ^\alpha D_\zeta^\beta F(\mu, z, \zeta, \htilde) \vert \leq \mathcal{F}_{ \alpha,\beta}(\Vert V \Vert_{E_0})\sum_{\vert a \vert \leq   \vert \alpha\vert+ \vert \beta \vert}  \sup_{\xi \in \mathcal{C}_3} \Vert D^a_\xi \gamma(\cdot,\xi)\Vert_{L^{\infty}(\xR^d)} \htilde^{  -2\delta \vert\alpha \vert }C_N \langle \mu \rangle^{-N}.
\end{equation}
If we call $ m_{ij}(z) $ the entries of the matrix $M_0(z)$ we see easily that $ D_\zeta^\beta F$ is a finite linear combination of terms of the form 
\begin{equation}\label{DbetaF}
  (D_\xi^{\beta_1}(\widehat{\chi} \varphi_1)) (\mu, M_0(z)\zeta) \cdot (D_\xi^{\beta_2}\gamma_\delta)(X(z)-\htilde ^{2}\mu, M_0(z) \zeta)\cdot P_{\vert \beta\vert}(m_{ij}(z)) := G_1 \cdot G_2 \cdot G_3 
  \end{equation}
where $P_{\vert \beta \vert}$ is a polynomial of order $\vert \beta \vert.$

The estimate \eqref{est:F} will follow from the following ones.
\begin{align} 
&\vert D^\alpha_z G_1  \vert \leq \mathcal{F}_{\alpha, \beta} (\Vert V \Vert_{E_0})\widetilde{h}^{- 2\delta \vert \alpha \vert} C_N \langle \mu \rangle^{-N},\label{est:G1}\\ 
 &\vert D^\alpha_z G_2  \vert \leq \mathcal{F}_{\alpha, \beta} (\Vert V \Vert_{E_0}))\sum_{\vert a \vert \leq   \vert \alpha\vert+ \vert \beta \vert}  \sup_{\xi \in \mathcal{C}_3} \Vert D^a_\xi \gamma(\cdot,\xi)\Vert_{L^{\infty}(\xR^d)} \htilde^{-2\delta \vert\alpha \vert } \label{est:G2}\\
 &\vert D^\alpha_z G_3  \vert \leq \mathcal{F}_{\alpha, \beta} (\Vert V \Vert_{E_0})\widetilde{h}^{- 2\delta \vert \alpha \vert}. \label{est:G3}
\end{align}
 Using the  equality $^t{}\! \big(\frac{\partial X}{\partial z}\big)( z) M_0(  z) \zeta = \zeta,$  Proposition \ref{estX} and an induction we see that
\begin{equation}\label{est1}
\vert D^\alpha_z  m_{ij}(z) \vert \leq \mathcal{F}_\alpha (\Vert V \Vert_{E_0})\widetilde{h}^{- 2\delta \vert \alpha \vert    + \frac{\delta}{2}} 
\end{equation}
from which  \eqref{est:G3} follows since $G_3$ is  polynomial.
Now according to the Faa-di-Bruno formula $D_z^\alpha G_1$ is a finite linear combination of terms of the form 
$$ D_\xi^{\beta_1 + b}(\widehat{\chi} \varphi_1)) (\mu, M_0(z)\zeta)\prod_{j=1}^r \Big(D_z^{l_j}M_0(z)\zeta)\Big)^{p_j}, 1\leq \vert b \vert \leq \vert \alpha \vert, \vert l_j \vert \geq 1, \sum_{j=1}^r \vert p_j \vert l_j = \alpha, \sum _{j=1}^r p_j =b. 
$$
 Then \eqref{est:G1} follows immediately from Lemma \ref{l0.6} and \eqref{est1}. By the same formula we see that $D_z^\alpha G_2$ is a linear combination of terms of the form
 $$(D^a_z D_\xi^{\beta_2+b}\gamma_\delta)(X(z)-\htilde ^{2}\mu, M_0(z) \zeta)\prod_{j=1}^r \Big(D_z^{l_j}X(z)\Big)^{p_j}\Big( D_z^{l_j}M_0(z)\zeta\Big)^{q_j}$$
 where $1 \leq \vert a \vert + \vert b \vert \leq \vert \alpha \vert, \quad \sum_{j=1}^r (\vert p_j \vert + \vert q_j \vert) l_j = \alpha, \quad \sum_{j=1}^r p_j = a,  \quad\sum_{j=1}^r q_j = b.$
 Then \eqref{est:G2} follows  \eqref{est2}, \eqref{est3} and \eqref{est1}. The proof is complete.
      \end{proof}
\begin{rema}\label{rem:est:ptilde}
By exactly the same method we can show that we have the estimate
\begin{equation}\label{est:ptilde}
 \vert  D_z ^{\alpha_1}   D_{z'} ^{\alpha_2}  D_\zeta^\beta \widetilde{p}(t,z,z',\zeta,\htilde )\vert \leq  \mathcal{F}_{ k}(\Vert V \Vert_{E_0} \, \mathcal{N}_k(\gamma)) \, \htilde^{ (\vert \alpha_1 \vert + \vert \alpha_2\vert) (1-2 \delta)} 
\end{equation} 
 for   all $\vert \alpha_1\vert + \vert \alpha_2\vert +\vert \beta \vert \leq k$  and all $(t,z,z', \zeta,\htilde) \in I_{\htilde }\times \xR^d  \times \xR^d \times  \mathcal{C}_1\times (0, \htilde _0]$.
  \end{rema}
 
\begin{prop}\label{Hessp}
There exist $T_0>0, c_0>0, \htilde _0>0$ such that
$$
\la \det \Big( \frac{\partial^{2}p}{\partial \zeta_j \partial \zeta_k} (t,z,\zeta,\htilde ) \Big) \ra \geq c_0
$$
for any $t\in [0,T_0], z\in \xR^d, \zeta \in \mathcal{C}_0= \{\mez \leq \vert \zeta \vert \leq 2\}, 0<\htilde  \leq \htilde _0.$  
\end{prop}
\begin{proof}
By Proposition \ref{estX}, \eqref{notation} and  \eqref{phi1} we have $\varphi_1( M_0(t,\htilde z)\zeta) =1.$ 
Let us set
$$
M_0(t,\htilde z) =(m_{ij}) \text{ and }
M_0(t,\htilde z) \zeta= \rho.
$$
Then we have,
$$
\frac{\partial^{2}p}{\partial \zeta_j \partial \zeta_k} (t,z,\zeta,\htilde ) = A_1 +A_2+A_3,
$$
where
\begin{equation}
\begin{aligned}
A_1&=\sum_{l,r=1}^d\int  \frac{\partial^{2}\widehat{\chi}}{\partial \zeta_l \partial\zeta_r}(\mu,\rho) 
m_{lj}m_{rk} {\gamma}_\delta\big (t, X(t,\htilde z)-\htilde ^{2}\mu, \rho\big)\,d\mu,\\
A_2&= 2\sum_{l,r=1}^d\int  \frac{\partial\widehat{\chi}}{\partial \zeta_l }(\mu,\rho)  
\frac{\partial {\gamma_\delta}}{\partial \zeta_r}
\big (t, X(t,\htilde z)-\htilde ^{2}\mu, \rho\big)m_{lj}m_{rk}\,d\mu,\\
A_3&= \sum_{l,r=1}^d\int \widehat{\chi}(\mu, \rho) \frac{\partial^{2}\gamma_\delta}{\partial \zeta_l \partial\zeta_r}
\big (t, X(t,\htilde z)-\htilde ^{2}\mu, \rho\big)m_{lj}m_{rk} \, d\mu.
\end{aligned}
\end{equation}
Now we notice that by \eqref{chi} we have 
$$\int (\partial_\zeta^\alpha \widehat{\chi})(\mu, \rho) d\mu = (2\pi)^d  (\partial_\zeta^\alpha  {\chi})(0,\rho)= \left \{ \begin{array}{ll} 0, & \alpha \neq 0\\   (2\pi)^d, & \alpha =0. \end{array} \right.$$
Using this remark we can write
$$
 A_1 =\sum_{l,r=1}^d\int  \frac{\partial^{2}\widehat{\chi}}{\partial \zeta_l \partial \zeta_r}(\mu,\rho) 
 m_{lj}m_{rk}
 \Big[{\gamma}_\delta\big (t, X(t,\htilde z)-\htilde ^{2}\mu, \rho\big) - {\gamma}_\delta\big (t, X(t,\htilde z) , \rho\big)\Big]\, d\mu.
$$
Now recall (see \eqref{EFG} and Lemma \ref{estgamma}) that for bounded $\vert \zeta\vert $ (considered as a parameter) 
we have for all $\alpha\in \xN^d$
$$
\partial_\zeta ^\alpha \gamma_\delta \in L^\infty(I; H^{s-\mez}(\xR^d))
\subset L^\infty (I; W^{s_0,\infty}(\xR^d)), \quad s_0>0,
$$
uniformly in~$\zeta$. 
Since, by Proposition~\ref{estX},  $\Vert M_0(t, \htilde z)\Vert$ is uniformly bounded we can write
$$
\vert A_1 \vert \leq C \htilde ^{2s_0}\sum_{l,r=1}^d\int \vert \mu \vert^{s_0}
\la  \frac{\partial^{2}\widehat{\chi}}{\partial \zeta_l \partial \zeta_r}(\mu,\rho) \ra d\mu,
$$
the integral in the right hand side being bounded by Lemma~\ref{l0.6}.
 
By exactly the same argument we see that we have the following   inequality
$$
\vert A_2 \vert \leq C\htilde ^{2s_0}.
$$
Moreover one can write
$$
A_3 =  \sum_{l,r=1}^d\frac{\partial^{2}\gamma_\delta}{\partial \zeta_l \partial \zeta_r}\big (t, X(t,\htilde z) , \rho\big)
m_{lj}m_{rk} \Big(\int \widehat{\chi}(\mu, \rho) \, d\mu\Big) + \mathcal{O}(\htilde ^{2s_0}).
$$
Gathering the estimates we see that
$$
(2\pi)^{-d} \Big(\frac{\partial^{2}p}{\partial \zeta_j \partial \zeta_k} (t,z,\zeta,\htilde )\Big) 
=   {}\!^t\!M_0(t,\htilde z)\text{Hess}_\zeta (\gamma_\delta)(t,z,\zeta,\htilde )M_0(t, \htilde z) + \mathcal{O}(\htilde ^{2s_0}).
$$
Then our claim follows from Proposition~\ref{hess} 
and Proposition~\ref{estX} if $\htilde _0 $ is small enough.
\end{proof}

\section{The parametrix}

Our aim is to construct a parametrix for the operator 
$L = \htilde \partial_t +\htilde c+ iP$ on a time interval of size $\htilde ^\delta$  where $\delta = \frac{2}{3}.$  This parametrix will be of the following form
$$
\mathcal{K}v(t,z) = (2\pi\htilde )^{-d} \iint e^{i{\htilde}^{-1}(\phi(t,z,\xi,\htilde ) -y\cdot \xi)} \widetilde{b}(t,z,y,\xi,\htilde )v(y)\, dyd\xi.
$$
Here $\phi$  is a real valued phase such that 
$\phi\arrowvert _{t=0} = z\cdot \xi,$ $\widetilde{b} $ is  of the form
\begin{equation}\label{tildeb}
\widetilde{b}(t,z,y,\xi,\htilde) = b(t,z,\xi,\htilde ) \Psi_0\Big( \frac{\partial \phi}{\partial \xi}(t,z,\xi,\htilde ) -y\Big ) 
\end{equation}
where 
$b\arrowvert_{t=0} = \chi(\xi), \chi \in C_0^\infty(\xR^d\setminus\{0\})$ and 
$\Psi_0\in C_0^\infty(\xR^d)$ is such that $ \Psi_0(t) = 1$ if $\vert t \vert \leq 1$.
\subsection{Preliminaries}

An important step in this construction is to compute the expression
\begin{equation}\label{J}
J (t,z,y,\xi,\htilde) = e^{-i{\htilde}^{-1}\phi(t,z,\xi,\htilde )}P(t,z,D_z)\big(e^{i{\htilde}^{-1}\phi(t,\cdot,\xi,\htilde )}\widetilde{b}(t,z,y,\xi,\htilde) \big).
\end{equation}
In this computation since $(t,y, \xi,\htilde )$ are fixed we shall 
skip them and   write $\phi = \phi(z), \widetilde{b} = \widetilde{b} (z).$

Using \eqref{operateur} we obtain
$$
J= (2\pi\htilde )^{-d}\iint e^{i{\htilde}^{-1}(\phi(z') - \phi(z) 
+(z-z')\cdot \zeta)} \widetilde{p} (z,z',\zeta)\widetilde{b} (z') \, dz\, 'd\zeta.
$$
Then we write
\begin{equation}\label{teta}
\phi(z') - \phi(z) = \theta (z,z')\cdot(z'-z), \quad  \theta (z,z') 
= \int_0^1 \frac{\partial \phi}{\partial z}(\lambda z + (1-\lambda) z') \, d \lambda.
\end{equation}
Using this equality and setting $\zeta - \theta (z,z') = \eta$ in the integral we obtain
$$
J=(2\pi\htilde )^{-d}\iint e^{i{\htilde}^{-1}(z-z')\cdot\eta}\widetilde{p}(z,z',\eta + \theta(z,z')) \widetilde{b} (z') \, dz' d\eta.
$$
The phase that we will obtain will be uniformly bounded, 
say $\vert \frac{\partial \phi}{\partial z}\vert \leq C_0.$
It also can be seen that, due to the cut-off $\varphi_1$ in the expression of $\widetilde{p}$ and to Proposition~\ref{estX},  
we also have $\vert \eta +\theta(z,z') \vert \leq C_0$.  
Therefore $\vert \eta \vert \leq 2C_0.$ Let $\kappa \in C_0^\infty(\xR^d)$ be such that 
$\kappa(\eta) = 1$ if $\vert \eta \vert \leq 2C_0$. Then we can write
$$
J =(2\pi\htilde )^{-d}\iint e^{i{\htilde}^{-1}(z-z')\cdot\eta}\kappa(\eta)
\widetilde{p}(z,z',\eta + \theta(z,z')) \widetilde{b} (z') \, dz' \, d\eta.
$$
By the Taylor formula we can write
\begin{equation*}
\begin{aligned}
 &\widetilde{p}(z,z',\eta +\theta(z,z')) = 
 \sum_{\vert \alpha \vert \leq {N-1}} \frac{1}{\alpha!}(\partial^\alpha_\eta \widetilde{p})(z,z',\theta(z,z'))\eta^\alpha +r_N \\
 &r_N = \sum_{\vert \alpha \vert =N} \frac{N}{\alpha!}
 \int_0^1 (1-\lambda)^{N-1}(\partial^\alpha_\eta \widetilde{p})(z,z',\eta+ \lambda\theta(z,z'))\eta^\alpha \, d\lambda.
\end{aligned}
\end{equation*}
It follows that
\begin{equation}\label{J=}
\left\{
\begin{aligned}
&J =J_N + R_N\\
&J_N =  \sum_{\vert \alpha \vert \leq {N-1}} \frac{(2\pi\htilde )^{-d}}{\alpha!}
\iint e^{i{\htilde}^{-1}(z-z')\cdot\eta}\kappa(\eta)(\partial^\alpha_\eta \widetilde{p})(z,z',\theta(z,z'))\eta^\alpha \widetilde{b} (z') \, dz' d\eta \\
&R_N = (2\pi \htilde )^{-d} \iint e^{i{\htilde}^{-1}(z-z')\cdot\eta} 
\kappa(\eta) r_N(z,z',\eta) \widetilde{b} (z') \, dz' d\eta.
 \end{aligned}
 \right.
\end{equation}
Using the fact that $\eta^\alpha e^{i{\htilde}^{-1}(z-z')\cdot\eta}= (-\htilde D_{z'})^\alpha e^{i{\htilde}^{-1} (z-z')\cdot\eta}$ and integrating by parts in the integral with respect to $z$ we get
$$
J_N =  (2\pi\htilde )^{-d}\sum_{\vert \alpha \vert \leq {N-1}}\frac{\htilde ^{\vert \alpha \vert}}{\alpha!}
\iint e^{i{\htilde}^{-1}(z-z')\cdot\eta}\kappa(\eta)D_{z'}^\alpha
\big[(\partial^\alpha_\eta \widetilde{p})(z,z',\theta(z,z')) \widetilde{b} (z')\big] \, dz' d\eta.
$$
Therefore we can write
$$
J_N =  (2\pi\htilde )^{-d}\sum_{\vert \alpha \vert 
\leq {N-1}}\frac{\htilde ^{\vert \alpha \vert}}{\alpha!}
\int \widehat{\kappa}\big(\frac{z'-z}{\htilde }\big)
D_{z'}^\alpha\big[(\partial^\alpha_\eta \widetilde{p})(z,z',\theta(z,z')) \widetilde{b} (z')\big] \, dz'.
$$ 
Let us set
\begin{equation}\label{falpha}
f_\alpha(z,z',\htilde ) = D_{z'}^\alpha\big[(\partial^\alpha_\eta \widetilde{p})(z,z',\theta(z,z')) \widetilde{b} (z')\big] 
\end{equation}
and then, $z'-z = \htilde \mu$ in the integral. We obtain
$$
J_N = (2\pi)^{-d}\sum_{\vert \alpha \vert 
\leq {N-1}}\frac{\htilde ^{\vert \alpha \vert}}{\alpha!}
\int \widehat{\kappa}(\mu)f_\alpha(z,z+ \htilde \mu, \htilde )\, d\mu.
$$
By the Taylor formula we can write
$$
J_N = (2\pi)^{-d}\sum_{\vert \alpha \vert \leq {N-1}}
\frac{\htilde ^{\vert \alpha \vert}}{\alpha!}\sum_{\vert \beta \vert \leq {N-1}}
\frac{\htilde ^{\vert \beta \vert}}{\beta!}
\Big(\int \mu^\beta \widehat{\kappa}(\mu) \, d\mu\Big)\big(\partial_{z'}^\beta f_\alpha\big)(z,z,\htilde ) +S_N,
$$
with
\begin{equation}\label{SN}
S_N = (2\pi)^{-d}\sum_{\substack{\vert \alpha \vert \leq {N-1} \\ \vert \beta \vert=N}} 
N \frac{\htilde ^{\vert \alpha \vert+\vert \beta \vert}}{\alpha! \beta!} 
 \int \int_0^1 
 (1-\lambda)^{N-1}\mu^\beta \widehat{\kappa}(\mu) \big(\partial_{z'}^\beta f_\alpha\big) 
 (z,z+\lambda\htilde \mu,\htilde )  \, d\lambda \, d\mu.
\end{equation}
Noticing that
$$
\int \mu^\beta \widehat{\kappa}(\mu) \, d\mu = (2\pi)^d (D^\beta \kappa)(0) 
= \left \{ \begin{array}{ll} 0 & \text{if } \beta \neq 0 \\(2\pi)^d & \text{if } \beta =0 \end{array} \right.
$$
we conclude that
$$
J_N = \sum_{\vert \alpha \vert \leq {N-1}}\frac{\htilde ^{\vert \alpha \vert}}{\alpha!} f_\alpha(z,z,\htilde ) +S_N.
$$
It follows from \eqref{J=}, \eqref{falpha} and \eqref{SN} that 
\begin{equation}\label{Jfinal}
J =  \sum_{\vert \alpha \vert \leq {N-1}}\frac{\htilde ^{\vert \alpha \vert}}{\alpha!} 
D_{z'}^\alpha\big[(\partial^\alpha_\eta \widetilde{p})(z,z',\theta(z,z')) \widetilde{b} (z')\big]\arrowvert_{z'=z} +R_N +S_N
\end{equation}
where $R_N$ and $S_N$ are defined in \eqref{J=} and \eqref{SN}.

Reintroducing the variable $(t,y,\xi,\htilde)$ we conclude from \eqref{J} that 
\begin{equation}\label{conjugaison}
e^{-i{\htilde}^{-1}\phi(t,z,\xi,\htilde )}(\htilde  \partial_t +\htilde c+iP)\big(e^{i{\htilde}^{-1}\phi(t,z,\xi,\htilde )}\widetilde{b} \big)
=\Big[i\frac{\partial \phi}{\partial t} \widetilde{b}  + iJ + \htilde \frac{\partial \widetilde{b} }{\partial t} +\htilde c \widetilde{b}\Big]
\big(t,z,y,\xi,\htilde\big).
\end{equation}
We shall gather the terms the right hand side of \eqref{conjugaison} according to the power of $\htilde $. 
The term corresponding to  $\htilde ^0$ leads to the eikonal equation.

\subsection{The eikonal equation}
It is the equation
\begin{equation}\label{eikonale}
\frac{\partial \phi}{\partial t} + p\Big(t,z, \frac{\partial \phi}{\partial z}, \htilde \Big) =0 \quad \phi(0,z,\xi,\htilde ) = z \cdot \xi
\end{equation}
where $p$ is defined by the formula
\begin{equation}\label{p}
  p(t,z,\zeta,\htilde ) =    \int \widehat{\chi}\big(\mu  , M_0(t, \htilde z)\zeta\big)
{\gamma}_\delta\big (t, X(t,\htilde z)-\htilde ^{2}\mu, M_0(t,\htilde z) \zeta\big)  
\cdot\varphi_1( M_0(t,\htilde z)\zeta)\,d\mu.
\end{equation}
We set
 $$
q(t,z,\tau,\zeta,\htilde ) = \tau + p(t,z,\zeta,\htilde ) 
$$
and for $j\geq 1$ we denote by $\mathcal{C}_j$  the ring
$$
\mathcal{C}_j = \{\xi \in \xR^d: 2^{-j} \leq \vert \xi \vert \leq 2^j\}.
$$
Moreover in all what follows we shall have
\begin{equation}\label{delta}
  \delta = \frac{2}{3}.
  \end{equation} 
 
    \subsubsection{The solution of the eikonal equation}
Recall that $I_{\htilde} = [0, \htilde^\delta] $ is the time interval,  (where $\delta = \frac{2}{3} $) and $\mathcal{C}_j$   the ring $\{2^{-j} \leq \vert \xi \vert \leq 2^j\}.$   Consider the null-bicharacteristic flow of $q$. It is defined by the system 
\begin{equation}\label{bicar}
\left\{
\begin{aligned}
\dot{t}(s) &=1,\quad t(0) =0,\\
\dot{z}(s) &= \frac{\partial p}{\partial \zeta}\big(t(s),z(s), \zeta(s), \htilde \big), \quad z(0) =z_0,\\
\dot{\tau}(s) &=- \frac{\partial p}{\partial t}\big(t(s),z(s),\zeta(s), \htilde \big), \quad \tau(0) = - p(0,z_0,\xi,\htilde ),\\
\dot{\zeta}(s) &=- \frac{\partial p}{\partial z}\big(t(s),z(s), \zeta(s), \htilde \big), \quad \zeta(0) =\xi.
\end{aligned}
\right.
\end{equation}
Then $t(s) =s$  and this system has a unique solution defined on $ I_{\htilde},$ depending on $(s,z_0,\xi,\tilde{h}).$

We claim that for all fixed $ s \in I_{\htilde }$ and $ \xi \in \xR^d,$ the map
\begin{equation}\label{diffeo}
z_0 \mapsto z(s; z_0,\xi ,\htilde )
\end{equation}
is a global diffeomorphism from $\xR^d$ to $\xR^d$. 
This will follow from the facts that this map is proper and the 
matrix $\big( \frac{\partial z}{\partial z_0}(s;z_0,\xi,\htilde )\big)$ is invertible. 
Let us begin by the second point. 

Let us set $m(s) =(s, z(s), \zeta(s),\htilde )$. 
Differentiating System \eqref{bicar} with respect to $z_0$, we get
\begin{equation}\label{Dbicar}
\begin{aligned}
\dot{\frac{\partial z}{\partial z_0}}(s) 
&= p''_{z \zeta}(m(s)) \frac{\partial z}{\partial z_0}(s) + p''_{\zeta \zeta} (m(s))\frac{\partial \zeta}{\partial z_0}(s), \quad \frac{\partial z}{\partial z_0}(0)= Id\\
\dot{\frac{\partial \zeta}{\partial z_0}}(s) 
&= -p''_{zz}(m(s))\frac{\partial z}{\partial z_0}(s) -p''_{z\zeta}(m(s))\frac{\partial \zeta}{\partial z_0}(s), 
\quad \frac{\partial \zeta}{\partial z_0}(0)=0.
\end{aligned}
\end{equation}
Setting $U(s) =(\frac{\partial z}{\partial z_0}(s),\frac{\partial \zeta}{\partial z_0}(s))$ and 
\begin{equation}\label{matrice}
\mathcal{A}(s)=\begin{pmatrix} 
 \hphantom{-}p''_{z\zeta}(m(s))&    \hphantom{-} p''_{\zeta \zeta}(m(s))\\[1ex] -p''_{zz}(m(s)) &- p''_{z\zeta}(m(s))   
\end{pmatrix}
. 
\end{equation}
The system \eqref{Dbicar} can be written as $\dot{U}(s) = \mathcal{A}(s)U(s), U(0) = (Id,0).$ 
 Lemma \ref{derp} gives
\begin{equation}
  \vert  p''_{\zeta \zeta}(m(s))\vert  +    \vert  p''_{z \zeta}(m(s))\vert   +  \vert  p''_{zz}(m(s))\vert \leq \\   \mathcal{F} (\Vert V \Vert_{E_0})
   \sum_{\vert \beta \vert \leq  2}  \sup_{\xi \in \mathcal{C}_3}\Vert D^\beta_\xi \gamma(t,\cdot,\xi)\Vert_{L^\infty (\xR^d)}  \htilde^{2(1-2\delta)}, 
 \end{equation}
therefore
$$
\Vert \mathcal{A}(s) \Vert 
\leq \mathcal{F} (\Vert V \Vert_{E_0})
   \sum_{\vert \beta \vert \leq  2}  \sup_{\xi \in \mathcal{C}_3}\Vert D^\beta_\xi \gamma(t,\cdot,\xi)\Vert_{L^ \infty (\xR^d)} \htilde^{2(1-2\delta)}. 
  $$ 
Using the equality $2(1-2\delta) + \delta =0$, we deduce  that for $ s \in I_{\htilde }= (0, \widetilde{h}^\delta) $ we have
\begin{equation}\label{calA}
\int_0^s \left\Vert \mathcal{A}(\sigma)  \right\Vert d \sigma 
\leq  \mathcal{F} (\Vert V \Vert_{E_0}) \mathcal{N}_2(\gamma).
 \end{equation}
The Gronwall inequality shows that $\Vert U(s)\Vert$ 
is uniformly bounded on $I_{\htilde }.$ 
Coming back to   \eqref{Dbicar} we see that we have
\begin{equation}\label{estDx}
  \la \frac{\partial \zeta}{\partial z_0} (s) \ra \leq  \mathcal{F} \big(\Vert V \Vert_{E_0} +   \mathcal{N}_2(\gamma)\big) , \quad 
  \la \frac{\partial z}{\partial z_0}(s) - Id \ra \leq
 \mathcal{F} \big(\Vert V \Vert_{E_0} +   \mathcal{N}_2(\gamma)\big)   \htilde^{\frac{\delta}{2}}.
  \end{equation}
Taking $\htilde $ small enough we obtain the invertibility of the 
matrix  $\big( \frac{\partial z}{\partial z_0}(s;z_0,\xi,\htilde )\big)$.

Now we have
$$
\vert z(s; z_0,\xi,\htilde ) -z_0 \vert \leq \int_0^s \vert \dot{z}(\sigma,x_0,\xi,\htilde ) \vert \, d\sigma.
$$
Since the right hand side is uniformly bounded for $s \in \big[0, \htilde ^{\delta}\big]$, 
we see that our map is proper. 
Therefore we can write
\begin{equation}\label{diffeo1}
z(s;z_0,\xi,\htilde ) = z \Longleftrightarrow z_0=\kappa(s;z,\xi,\htilde ). 
\end{equation}
Let us set for $t\in\big[0, \htilde ^{\delta}\big]$
\begin{equation}\label{phi}
\phi(t,z,\xi,\htilde ) = z\cdot\xi - 
\int_0^t p\big(\sigma,z,\zeta(\sigma;\kappa(\sigma;z,\xi,\htilde ),\xi,\htilde ), \tilde{h} \big)\, d\sigma.
\end{equation}
\begin{prop}\label{eqeiko}
The function $\phi$ defined in \eqref{phi} is the solution of the eikonal equation \eqref{eikonale}.
\end{prop}
\begin{proof}
The initial condition is trivially satisfied. Moreover we have
$$
\frac{\partial \phi}{\partial t}\big(t,z,\xi,\htilde\big) 
= - p\big(t,z,\zeta(t;\kappa\big(t;z,\xi,\htilde\big),\xi,\htilde), \tilde{h}\big).
$$
Therefore it is sufficient to prove that
\begin{equation}\label{eiksuite}
\frac{\partial \phi}{\partial z}\big(t,z,\xi,\htilde\big) 
= \zeta\big(t;\kappa\big(t;z,\xi,\htilde\big),\xi,\htilde\big).
\end{equation}
Let us consider the Lagrangean manifold
\begin{equation}\label{Lagrange}
\Sigma =\Big\{\big(t,z(t; z_0,\xi ,\htilde ),\tau(t; z_0,\xi ,\htilde ),\zeta(t; z_0,\xi ,\htilde )\big):   
t \in I_{\htilde}  , (z_0,\xi ) \in \xR^{2d} \Big\}.
\end{equation}
According to \eqref{diffeo1} we can write
$$
\Sigma = \Big\{(t,z, \tau(t;\kappa(t;z,\xi,\htilde ),\xi,\htilde ), \zeta(t;\kappa(t;z,\xi,\htilde ),\xi,\htilde ))
: t\in I_{\htilde }, (z,\xi)\in \xR^{2d} \Big\}.
$$
Let us set 
\begin{equation*}
\begin{aligned}
F_0\big(t,z,\xi,\htilde\big) &= \tau\big(t;\kappa(t;z,\xi,\htilde ),\xi,\htilde\big),\\
F_j\big(t,z,\xi,\htilde\big)  &= \zeta_j\big(t;\kappa(t;z,\xi,\htilde ),\xi,\htilde\big).
\end{aligned}
\end{equation*}
Since the symbol $q$ is constant along its bicharacteristic   and $q(0, z(0), \tau(0), \zeta(0), \htilde)= 0$ we have
$$
F_0\big(t,z,\xi,\htilde\big) = - p\big(t,z,\zeta(t;\kappa(\sigma;z,\xi,\htilde ),\xi,\htilde\big), \htilde \big).
$$
Now $\Sigma$ beeing Lagrangean we have
$$
\di t \wedge \di F_0 + \di z \wedge \di F =0.
$$
Thus $\partial_{z_j} F_0 - \partial_t F_j =0$ since it is the coefficient of $\di t\wedge \di z_j$ in the above expression. 
Therefore using \eqref{phi} we can write
\begin{equation*}
\begin{aligned}
\frac{\partial \phi}{\partial z_j}(t,z,\xi,\htilde ) &=\xi_j - \int_0^t \frac{\partial}{\partial z_j} 
\big[p\big(\sigma,z,\zeta(\sigma;\kappa(\sigma;z,\xi,\htilde ),\xi,\htilde )\big)\big]\, d \sigma\\
&= \xi_j + \int_0^t \frac{\partial}{\partial \sigma}\big[ \zeta_j\big(\sigma;\kappa(\sigma;z,\xi,\htilde\big),\xi,\htilde )\big] \, d \sigma\\
&= \zeta_j\big(t;\kappa(t;z,\xi,\htilde ),\xi,\htilde\big).
\end{aligned}
\end{equation*}
    \end{proof}
    
\subsubsection{The Hessian of the phase}

Let us recall that the phase $\phi$ is the solution of the problem
\begin{equation}\label{phase}
\left\{
\begin{aligned}
&\frac{\partial \phi}{\partial t} (t,z,\xi,\htilde ) + p\Big(t,z,\frac{\partial \phi}{\partial z}(t,z,\xi,\htilde ), \htilde \Big) =0\\
&\phi \arrowvert_{t=0} = z \cdot \xi.
\end{aligned}
\right.
\end{equation}
On the other hand the map $(t,z,\xi) \mapsto \phi(t,z,\xi,\htilde )$ is $C^1$ in time and $C^\infty$ in $(x,\xi)$. 
Differentiating twice, with respect to $\xi$, the above equation we obtain
\begin{equation*}
\begin{aligned}
\frac{\partial }{\partial t} \Big( \frac{\partial^2 \phi }{\partial  \xi_i \partial \xi_j}\Big) 
&= - \sum_{k,l = 1}^d  
\frac{\partial^2 p}{\partial  \zeta_k \partial \zeta_l}\Big(t,z,\frac{\partial \phi}{\partial z},\htilde \Big)  
\frac{\partial^2 \phi }{\partial  z_k \partial \xi_i}   \frac{\partial^2 \phi }{\partial z_l \partial \xi_j} \\
& \quad- \sum_{k=1}^d \frac{\partial p}{ \partial \zeta_k}\Big(t,z,\frac{\partial \phi}{\partial z},\htilde \Big)
\frac{\partial^3 \phi}{\partial z_k \partial \xi_i \partial \xi_j}\cdot
\end{aligned}
\end{equation*}
By the initial condition in \eqref{phase} we have
$$
\frac{\partial^2 \phi}{\partial z_k \partial \xi_i}\Big  \arrowvert_{t=0} 
= \delta_{k i}\quad   \frac{\partial^2 \phi}{\partial z_l \partial \xi_j} \Big  \arrowvert_{t=0} 
= \delta_{lj}, \quad \frac{\partial^3 \phi}{\partial z_k \partial \xi_i \partial \xi_j} \Big   \arrowvert_{t=0} = 0, 
\quad  \frac{\partial^2 \phi }{\partial  \xi_i \partial \xi_j} \Big   \arrowvert_{t=0} =0.
$$
It follows that 
$$
\frac{\partial }{\partial t} \Big( \frac{\partial^2 \phi }{\partial  \xi_i \partial \xi_j}\Big)\Big\arrowvert_{t=0} 
= -\frac{\partial^2 p}{\partial  \xi_i \partial \xi_j}\big(0,z,\xi,\htilde \big)
$$
from which we deduce that 
$$
\frac{\partial^2 \phi }{\partial  \xi_i \partial \xi_j}(t,z,\xi,\htilde ) 
= -t \frac{\partial^2 p}{\partial  \xi_i \partial \xi_j}(0,z,\xi,\htilde ) +o(t).
$$
It follows from Proposition \ref{Hessp} that one can find $M_0>0$ such that
\begin{equation}\label{HessphiND}
\la \det \Big( \frac{\partial^2 \phi}{\partial  \xi_i \partial \xi_j}(t,z,\xi,\htilde )\Big) \ra \geq M_0 t^d,
\end{equation}
for $t\in I_{\htilde }, z \in \xR^d, \xi \in \mathcal{C}_0, 0<\htilde  \leq \htilde _0$.

Our goal now is to prove estimates of higher order on the phase (see Corollary \ref{estkappa} below.)
 \subsubsection{Classes of symbol and symbolic calculus}
Recall here that $\delta = \frac{2}{3}$ and that $ \mathcal{N}_{k}(\gamma)  $ has been defined in \eqref{norme:gamma}.
\begin{defi}
Let $m\in \xR, \mu_0 \in \xR^+ $ and $a =  a(t, z,\xi, \htilde)$ be a smooth function defined on $ \Omega =[0, \htilde^\delta]\times \xR^d \times \mathcal{C}_0\times (0, \htilde_0]$. We shall say that
 
 $(i)$ \quad  $a \in S^m_{\mu_0}$ if for every $k\in \xN$ one can find $\mathcal{F}_k:\xR^+ \to \xR^+$ such that for   all $(t,z,\xi,\htilde) \in \Omega $
 \begin{equation}\label{est:symb}
  \vert D_z^\alpha D_\xi^\beta a(t,z,\xi,\htilde)\vert \leq \mathcal{F}_k( \Vert V\Vert_{E_0} + \mathcal{N}_{k+1}(\gamma)) \, \htilde^{m - \vert \alpha \vert \mu_0},\quad \vert \alpha \vert + \vert \beta \vert =k,
  \end{equation}
 \\
 $(ii)$ \quad   $a \in \dot{S}^m_{\mu_0}$  if \eqref{est:symb} holds for every $k\geq1$.
     \end{defi}
\begin{rema}\label{remarque1}
\begin{enumerate}
\item If $m\geq m'$ then $ S^m_{\mu_0} \subset  S^{m'}_{\mu_0} $ and  $ \dot{S}^m_{\mu_0} \subset  \dot{S}^{m'}_{\mu_0}$.
\item Let $a(t,z, \xi, \htilde) = z$ and $b(t,z, \xi, \htilde) = \xi$. Then   $a \in \dot{S}^{\delta/2}_{2\delta -1}$,    $b \in \dot{S}^0_{2\delta -1}$.
\item If $a \in S^m_{\mu_0}$ with $m \geq 0$ then $b=e^a \in S^0_{\mu_0}.$
\end{enumerate}
\end{rema}
We study now the composition of such symbols.
 \begin{prop}\label{est:compo}
Let $m \in \xR , f \in S^{m}_{2\delta-1} (\text{resp.}\dot{S}^{m}_{2\delta-1}),  Ê U \in \dot{S}^{\delta/2}_{2\delta-1}, V\in \dot{S}^{0}_{2\delta-1} $ and assume that $V\in \mathcal{C}_0.$  Set 
$$F(t,z,\xi, \htilde) = f(t, U(t,z,\xi,\htilde), V(t,z,\xi,\htilde), \htilde).$$ Then $F\in S^{m}_{2\delta-1} (\text{resp.\ }\dot{S}^{m}_{2\delta-1}).$
\end{prop}
\begin{proof}
Let $\Lambda = ( \alpha, \beta)\in \xN^d \times \xN^d, \vert \Lambda\vert =k$. If $k=0$ the estimate of $F$ follows easily from the hypothesis on $f$. Assume  $k  \geq 1$. Then $D^\Lambda F$ is a finite linear combination of terms of the form
$$(1)= (D^Af)( \cdots) \prod_{j=1}^r (D^{L_j} U)^{p_j} (D^{L_j} V)^{q_j} $$
where $A=(a,b), \quad 1\leq \vert A \vert \leq  \vert \Lambda \vert, \quad L_j= (l_j,m_j)$ and
$$\sum_{j=1}^r p_j = a, \quad \sum_{j=1}^r q_j = b, \quad \sum_{j=1}^r (\vert p_j\vert + \vert q_j \vert)L_j = \Lambda$$
By the hypothesis on $f$   we have 
\begin{equation}\label{est:DAf3}
 \vert D^A f(\cdots)\vert \leq \mathcal{F}_k( \Vert V \Vert_{E_0} + \mathcal{N} _{k+1}(\gamma))\, \htilde^{m-\vert a \vert (2 \delta -1)}.
\end{equation}
By the hypotheses on $U,V$, the product occuring in the definition of $(1)$ is bounded by $  \mathcal{F}_k( \Vert V \Vert_{E_0} + \mathcal{N}_{k+1}(\gamma))\, \htilde^M$ where 
$$M= \sum_{j=1}^r \vert p_j \vert  \big(\frac{\delta}{2} -\vert l_j \vert (2 \delta -1) \big)  - \sum_{j=1}^r \vert q_j \vert ( \vert l_j \vert (2 \delta -1)) = -\vert \alpha \vert (2 \delta -1) + \frac{\delta}{2} \vert a \vert.$$
 Using \eqref{est:DAf3} and the fact that $1-2\delta +Ê\frac{\delta}{2} = 0$ we obtain the desired conclusion.
\end{proof}
 
  \subsubsection{Further estimates on the flow}
   We shall denote by  $z(s) = z(s; z, \xi, \htilde), \zeta(s) = \zeta(s; z, \xi, \htilde)$ the solution of \eqref{bicar} with $z(0) = z, \zeta(0) = \xi.$  Recall  that $\delta = \frac{2}{3} $. 
        \begin{prop}\label{estflow}
There exists $\mathcal{F}: \xR^+ \to \xR^+$ non decreasing such that
\begin{alignat*}{2}
&(i) \quad &&\la\frac{\partial z}{\partial {z }}(s) - Id \ra  \leq   \mathcal{F} \big(\Vert V \Vert_{E_0} +   \mathcal{N}_2(\gamma)\big)   \htilde^{\frac{\delta}{2}}, \quad \la\frac{\partial \zeta}{\partial {z }}(s)\ra\leq   \mathcal{F} \big(\Vert V \Vert_{E_0} +   \mathcal{N}_2(\gamma)\big),  \\
  &(ii) &&\la\frac{\partial z}{\partial \xi}(s) \ra   + \la\frac{\partial \zeta}{\partial \xi}(s) - Id \ra \leq \mathcal{F} \big(\Vert V \Vert_{E_0} +   \mathcal{N}_2(\gamma)\big) \, \htilde^{\frac{\delta}{2}}.
  \end{alignat*}
for all $s\in I_{\htilde } = \big[0, \htilde ^\delta\big], z\in \xR^d, \xi \in \mathcal{C}_0.$

    For any $k\geq 1$ there exists $\mathcal{F}_k: \xR^+ \to \xR^+$ non decreasing such that for $\alpha, \beta \in \xN^ d$ with $ \vert \alpha \vert + \vert \beta \vert =k $ 
  \begin{equation}\label{est:z:zeta}
  \left\{ 
  \begin{aligned}
   &  \la D_z^\alpha D_\xi^\beta  z(s) \ra \leq \mathcal{F}_k\big(\Vert V \Vert_{E_0} +   \mathcal{N}_{k+1}(\gamma)\big) \htilde ^{ \vert \alpha \vert  (1-2 \delta) + \frac{\delta}{2}},\\
  & \la D_z^\alpha D_\xi^\beta  \zeta(s)  \ra\leq  \mathcal{F}_k\big(\Vert V \Vert_{E_0} +   \mathcal{N}_{k+1}(\gamma)\big) \htilde ^{ \vert \alpha \vert  (1-2 \delta)}.
  \end{aligned}
  \right.
   \end{equation} 
\end{prop}
\begin{proof}
The estimates of the first terms in  $(i)$ and $(ii)$ have been proved in \eqref{estDx}.
By exactly the same argument one deduces the estimates  on the second terms. 

We shall prove \eqref{est:z:zeta} by induction on $k$. According to $(i)$ and $(ii)$ it is  true for $k =1.$
   Assume it is true up to the order $k$ and let $\vert \alpha \vert +\vert \beta \vert =k+1\geq 2$. 
Let us set $\Lambda =(\alpha, \beta),$ $D^\Lambda = D_z^\alpha D_\zeta^\beta$ and  $m(s) =(s, z(s),\zeta(s), \htilde ).$  
By the Faa-di-Bruno formula we have
\begin{equation}\label{F1-F2}
\left\{
\begin{aligned}
 &D^\Lambda[p'_\zeta(m(s))] = p''_{z\zeta}(m(s)) D^\Lambda z(s) + p''_{\zeta\zeta}(m(s)) D^\Lambda \zeta(s) + F_1(s),\\
 &D^\Lambda[p'_z(m(s))] = p''_{zz}(m(s)) D^\Lambda z(s) + p''_{z\zeta}(m(s)) D^\Lambda \zeta(s) + F_2(s).  
 \end{aligned}
\right.
\end{equation}
It follows that $U(s) = (D^\Lambda z(s),D^\Lambda \zeta(s)$ is the solution of the problem 
$$\dot{U}(s) = \mathcal{A}(s) U(s) + F(s), \quad U(0) = 0$$
where $\mathcal{A}(s)$ has been defined in \eqref{matrice} and $F(s) = (F_1(s), F_2(s)).$

According to the estimates of the symbol $p$ given in Lemma \ref{derp} the worse term is $F_2.$ By the   formula mentionned above we see that  $F_1$ is a finite linear combination of terms of the form
$$
\big(D^A p'_z \big)(m(s)) \prod_{i=1}^r \big(D^{L_i}z(s) \big)^{p_i} \prod_{i=1}^{r} \big(D^{L_i} \zeta(s) \big)^{q_i},
$$
where $ A=(a,b),\quad  2\leq \vert A \vert \leq \vert \Lambda \vert$ and
\begin{equation*}
  L_i=(l_i,l'_i), \quad 1 \leq \vert L_i \vert \leq \vert \Lambda \vert -1, \quad \sum_{i=1}^r p_i = a,\quad   \sum_{i=1}^{r} q_i = b,   \quad \sum_{i=1}^k (\vert p_i \vert +\vert q_i \vert )L_i   = \Lambda. 
  \end{equation*}
It follows from Corollary \ref{derp'} that for $s$ in $[0, \htilde^\delta]$ we have, 
\begin{equation}\label{estDqxi}
\int_0^s \la \big(D^Ap'_z \big)(m(\sigma)) \ra d \sigma \leq  \mathcal{N}_{\vert A \vert +1}(\gamma) \htilde^{(\vert a \vert +1)(1-2 \delta) + \delta}  \leq  \mathcal{N}_{\vert A \vert +1}(\gamma) \htilde^{ \vert a \vert (1-2 \delta) + \frac{\delta}{2}}.
 \end{equation}
 since $1-2 \delta +\delta =  \frac{\delta}{2}.$ Now since $1 \leq \vert L_i \vert \leq \vert \Lambda \vert -1= k$ we have, by the induction,
 \begin{align*}
 & \vert D^{L_i}z(s)\vert \leq \htilde^{\vert l_i\vert (1-2 \delta) + \frac{\delta}{2}} \mathcal{F}_k\big(\Vert V \Vert_{E_0} +   \mathcal{N}_{k+1}(\gamma)\big),\\
 &\vert D^{L_i}\zeta(s)\vert \leq \htilde^{\vert l_i\vert (1-2 \delta)}  \mathcal{F}_k\big(\Vert V \Vert_{E_0} +   \mathcal{N}_{k+1}(\gamma)\big). 
  \end{align*}
 It follows that
  $$
\int_0^s \vert F_2(\sigma) \vert d \sigma 
\leq   \Big(\int_0^s \la \big(D^Ap'_z \big)(m(\sigma)) \ra d \sigma \Big) \mathcal{F}_{k+1}\big(\Vert V \Vert_{E_0} +   \mathcal{N}_{k+1}(\gamma)\big) \htilde^{M} 
     $$
     where $M = \sum_{i=1}^r \big( (\vert p_i \vert + \vert q_i \vert) \vert l_i\vert (1- 2 \delta) \vert  + \vert p_i \vert \frac{\delta}{2}\big) = \vert \alpha \vert (1-2 \delta) + \vert a \vert \frac{\delta}{2}.$ It follows from \eqref{estDqxi} and the fact that $1 - \frac{3 \delta}{2} = 0$ that 
     $$\int_0^s \vert F_2(\sigma) \vert d \sigma  \leq \mathcal{F}_{k+1}\big(\Vert V \Vert_{E_0} +   \mathcal{N}_{k+1}(\gamma)\big)\htilde^{ \vert \alpha \vert (1-2 \delta) + \frac{\delta}{2}}.
 $$
 Since $D^Ap'_\zeta$ has even a better estimate, the same computation shows that 
 $$\int_0^s \vert F_1(\sigma) \vert d \sigma  \leq \mathcal{F}_{k+1}\big(\Vert V \Vert_{E_0} +   \mathcal{N}_{k+1}(\gamma)\big)\htilde^{ \vert \alpha \vert (1-2 \delta) + \frac{\delta}{2}}.
 $$ 
Then we write
 $$
U(s) = \int_0^s F(\sigma) d \sigma + \int_0^s \mathcal{A}(\sigma)U(\sigma)d \sigma 
$$    
 and we use the above estimates on $F_1, F_2$, \eqref{calA} and   the Gronwall lemma to see that the step $k+1$ of the induction is achieved. This completes the proof of Proposition \ref{estflow}.
  \end{proof}
\begin{coro}\label{estkappa}
For every $k \geq 1$ there exists $\mathcal{F}_k: \xR^+ \to \xR^+$ non decreasing such that for every   $(\alpha, \beta) \in \xN^d \times \xN^d$ with  $ \vert \alpha \vert +  \vert \beta \vert =k$     we have
 \begin{equation*}
\begin{aligned}
& (i) \quad &&\vert D_z^\alpha D_\xi^\beta \kappa(s,z,\xi,\htilde) \vert \leq \mathcal{F}_k\big(\Vert V \Vert_{E_0} +   \mathcal{N}_{k+1}(\gamma)\big) \htilde ^{ \vert \alpha \vert  (1-2 \delta) + \frac{\delta}{2}},\\
&(ii) \quad &&\Big\vert D_z^\alpha D_\xi^\beta \big(\frac{\partial \phi}{\partial z}\big)(s,z,\xi,\htilde )\Big\vert
 \leq \mathcal{F}_k\big(\Vert V \Vert_{E_0} +   \mathcal{N} _{k+1}(\gamma)\big) \htilde ^{ \vert \alpha \vert  (1-2 \delta)},\\
 &(iii)\quad && \vert D_\xi^\beta \phi (s,z,\xi,\htilde) \vert \leq \mathcal{F}_k\big(\Vert V \Vert_{E_0} +   \mathcal{N} _{k+1}(\gamma)\big)\vert s \vert, \quad \vert \beta \vert \geq 2,
  \end{aligned}
\end{equation*}
for all $s \in I_{\htilde }, z \in \xR^d, \xi \in \mathcal{C}_0.$   This implies that $\kappa \in \dot{S}_{2\delta-1}^{\delta/2}$ and $  \frac{\partial \phi}{\partial z}  \in  {S}_{2\delta-1}^0.$
\end{coro}
\begin{proof}
We first show $(ii)$ and  $(iii)$. Recall that 
$$
\frac{\partial \phi}{\partial z}(s,z,\xi,\htilde ) = \zeta(s; \kappa(s; z,\xi,\htilde ), \xi, \htilde ).
$$
By   Proposition \ref{estflow} (since $\zeta$ is bounded) we have $\zeta\in S^0_{2\delta-1}.$ By $(i)$ we have $\kappa \in \dot{S}^{\delta/2}_{2\delta-1}$ and by Remark \ref{remarque1} we have $\xi \in \dot{S}^0_{2\delta-1}.$  Then Proposition \ref{est:compo} implies that $\frac{\partial \phi}{\partial z} \in \dot{S}_{2\delta-1}^0.$ Moreover $\frac{\partial \phi}{\partial z} $ is bounded since $\vert \zeta(s) - \xi \vert \leq \int_0^s \vert\frac{ \partial p}{\partial z}(t, \ldots) \vert dt \leq \mathcal{F} \big(\Vert V \Vert_{E_0} +   \mathcal{N}_2(\gamma)\big) \widetilde{h}^\frac{\delta}{2}$ and $\xi \in \mathcal{C}_0.$ Now $(iii)$ follows from the definition \eqref{phi} of the phase, the facts that $p \in S^0_{2\delta-1}, z\in \dot{S}^\mez_{2\delta-1}, \zeta(s; \kappa(s; z,\xi,\htilde ), \xi, \htilde )\in \dot{S}^0_{2\delta-1}$ and Proposition \ref{est:compo}.
 
  We are left with the proof of $(i)$. We proceed by induction on   $\vert \alpha \vert + \vert \beta \vert = k \geq1.$ Recall that  by definition of $\kappa$ we have the equality $z(s; \kappa(s;z,\xi,\htilde ), \xi, \htilde ) = z$. It follows that  
 $$\frac{\partial z}{\partial z} \cdot \frac{\partial \kappa}{\partial z} = Id, \quad  \frac{\partial z}{\partial z} \cdot\frac{\partial \kappa}{\partial \xi} =  -\frac{\partial z}{\partial \xi} .$$
 Then the estimate for $k =1$ follows from $(i)$ in Proposition \ref{estflow}. Assume the estimate true up to the order $k$ and let $\Lambda =(\alpha, \beta), \vert \Lambda \vert = k+1 \geq 2.$ Then differentiating  $\vert \Lambda \vert $ times the first above equality we see that $\frac{\partial z}{\partial z} \cdot D^\Lambda \kappa$ is a finite linear combination of terms of the form
 $$ (2)=  D^A z(\cdots) \prod_{j=1}^r \big(D^{L_j}\kappa\big)^{p_j} \prod_{j=1}^r \big(D^{L_j} \xi)^{q_j}$$
where $A=(a,b) \quad 2 \leq \vert A \vert \leq \vert \Lambda \vert, \quad L_j =(l_j,m_j),\quad  1 \leq \vert L_j \vert \leq k$ and 
$$ \sum_{j=1}^r p_j = \alpha,\quad \sum_{j=1}^r q_j = \beta, \quad \sum_{j=1}^r (\vert p_j \vert +\vert q_j \vert )L_j = (\alpha, \beta).$$
 We use the estimate (given by Proposition \ref{estflow})
$$ \vert D^A z(\cdots) \vert \leq  \mathcal{F}_{k+1}\big(\Vert V \Vert_{E_0} +   \mathcal{N}^0_{k+2}(\gamma)\big)\htilde ^{ \vert a\vert  (1-2 \delta) + \frac{\delta}{2}},$$ 
 the induction, the fact that $\xi \in \dot{S}^0_{2\delta-1}$  and the equality $1-2\delta + \frac{\delta}{2} =0$ to see that  
$$ \vert (2) \vert \leq \mathcal{F}_{k+1}\big(\Vert V \Vert_{E_0} +   \mathcal{N}^0_{k+2}(\gamma)\big)\htilde ^{ \vert \alpha\vert  (1-2 \delta) + \frac{\delta}{2}}.$$
Then we use Proposition \ref{estflow} $(i)$ to conclude the induction.
\end{proof}
\begin{rema}\label{rem:est:theta}
Since $\theta(t,z,z',\xi,\htilde) =  
\int_0^1 \frac{\partial \phi}{\partial z}(t, \lambda z + (1-\lambda) z' , \xi,\htilde) \, d \lambda $ we have also the estimate
\begin{equation}\label{est:theta}
\Big\vert D_z^{\alpha_1} D_{z'}^{\alpha_2}D_\xi^\beta \theta(s,z,z',\xi,\htilde )\Big\vert
 \leq \mathcal{F}_k\big(\Vert V \Vert_{E_0} +   \mathcal{N} _{k+1}(\gamma)\big) \htilde ^{ (\vert \alpha_1 \vert + \vert \alpha_2 \vert) (1-2 \delta)}.
\end{equation}
for $ \vert \alpha_1 \vert + \vert \alpha_2 \vert + \vert \beta \vert =k$
\end{rema}

\subsection{The transport equations}

According to \eqref{Jfinal} and \eqref{conjugaison} if $\phi$ satisfies the eikonal equation we have
\begin{equation}\label{RS}
\begin{aligned}
&e^{-i{\htilde}^{-1}\phi }(\htilde  \partial_t + \htilde c+iP)\big(e^{i{\htilde}^{-1}\phi }\widetilde{b} \big) \\
&\quad=\htilde  \partial_t \widetilde{b}  +\htilde c \widetilde{b} +   i\sum_{\vert \alpha \vert=1}^{N-1}\frac{\htilde ^{\vert \alpha \vert}}{\alpha!}
D_{z'}^\alpha\Big[(\partial^\alpha_\eta \widetilde{p})\big(t,z,z',\theta(t,z,z',\htilde )),\htilde \big) \widetilde{b} (z')\Big]\Big\arrowvert_{z'=z} \\
&\quad\quad+R_N +S_N
\end{aligned}
\end{equation}

Recall (see \eqref{tildeb}) that $\widetilde{b}  = b \Psi_0$. Let us set 
\begin{equation} \label{CT1}
T_N=   \partial_t b + c +i \sum_{1\leq \vert \alpha \vert \leq {N-1}}\frac{\htilde ^{\vert \alpha \vert - 1}}{\alpha!} D_{z'}^\alpha\Big[(\partial^\alpha_\eta \widetilde{p})\big(t,z,z',\theta(t,z,z',\htilde ), \htilde \big) b(z')\Big]\Big\arrowvert_{z'=z}. 
\end{equation}
Then 
\begin{equation}\label{UN}
e^{-i{\htilde}^{-1}\phi }(\htilde  \partial_t + \htilde c +iP)\big(e^{i{\htilde}^{-1}\phi }\widetilde{b} \big) = \htilde T_N \Psi_0  +U_N + R_N+S_N.
\end{equation}

Our purpose is to show that one can find a symbol $b$ such that, in a sense to be explained,
\begin{equation}\label{OhN}
T_N = \mathcal{O}({\htilde }^M), \quad \forall M \in \xN.
\end{equation}
 We set
\begin{equation*}
\mathcal{L}b = \partial_t b + c + \sum_{i=1}^n \frac{\partial}{\partial {z'}_i} 
\Big[ \frac{\partial \widetilde{p}}{\partial \zeta_i}\big(t,z,z',\theta(t,z,z',\htilde ), \htilde \big) b(t,z',\xi,\htilde )\Big]\Big\arrowvert_{z'=z}
\end{equation*}
Then  we can write
\begin{equation}\label{mathcalL}
\mathcal{L} = \frac{\partial}{\partial t} + \sum_{i=1}^d a_j(t,z,\xi,\htilde ) \frac{\partial}{\partial {z}_i} + c_0(t,z,\xi,\htilde )
\end{equation}
where
\begin{equation}\label{coeffL}
\left\{
\begin{aligned}
a_i(t,z,\xi,\htilde ) &= \frac{\partial p}{\partial \zeta_i}\Big(t,z, \frac{\partial \phi}{\partial z}(t,z,\xi,\htilde ), \htilde \Big),\\
c_0(t,z,\xi,\htilde )&= \sum_{i=1}^d\frac{\partial}{\partial z'_i} 
\Big[ \frac{\partial \widetilde{p}}{\partial \zeta_i}\big(t,z,z',\theta(t,z,z',\xi, \htilde ), \htilde \big)  \Big]\Big\arrowvert_{z'=z} + c(t,z,\htilde),\\
\theta(t,z,z',\xi, \htilde )&= \int_0^1 \frac{\partial \phi}{\partial z}\big(t,\lambda z+(1-\lambda) z',\xi,\htilde \big)\, d\lambda  
\end{aligned}
\right.
\end{equation}
and $c$ has been defined in \eqref{c=}.

Notice that, with $m=(t,z,\xi,\htilde)$ we have 
\begin{equation}\label{c0}
c_0(m) = \sum_{i=1}^d \frac{\partial^2 \widetilde{p}}{\partial\zeta_i \partial z'_i}(t,z,z,\frac{\partial \phi}{\partial z}(m), \htilde) + \mez \sum_{i,j=1}^d \frac{\partial^2  {p}}{\partial\zeta_i \partial \zeta_j}(t,z,\frac{\partial \phi}{\partial z}(m),\htilde))\frac{\partial^2\phi}{\partial z_i \partial z_j}(m) +c.
\end{equation}
Then we can write
\begin{equation}\label{K=L}
T_N = \mathcal{L}b  + i \sum_{2\leq \vert \alpha \vert 
\leq {N-1}}\frac{\htilde ^{\vert \alpha \vert - 1}}{\alpha!} 
D_{z'}^\alpha\Big[(\partial^\alpha_\zeta \widetilde{p})\big(t,z,z', \theta(t,z,z',\htilde), \htilde \big) b(t,z',\xi, \htilde )\Big]\Big\arrowvert_{z'=z}. 
\end{equation}

We shall seek $b$ on the form
\begin{equation}\label{formedeb}
b = \sum_{j=0}^N \htilde ^j \, b_j. 
\end{equation}
Including this expression of $b$ in \eqref{K=L} after a change of indices we obtain
$$
T_N= \sum_{k=0}^N \htilde ^k \mathcal{L}b_k 
+ i  \sum_{k=1}^{N+1} \htilde ^k\sum_{2\leq \vert \alpha \vert \leq N-1}\frac{\htilde ^{\vert \alpha \vert - 2}}{\alpha!}D_{z'}^\alpha \big[(\partial^\alpha_\zeta \widetilde{p})(\cdots)b_{k-1}\big]\big\arrowvert_{z'=z}.
$$
We will take $b_j$ for  $ j=0, \ldots, N,$ as solutions of the following problems
\begin{equation}\label{bk}
\left\{
\begin{aligned}
\mathcal{L}b_0 &= 0, \quad b_0 \arrowvert_{t=0} = \chi, \quad \chi \in C_0^\infty(\xR^d),\\
\mathcal{L}b_j & = F_{j-1} := -i  \sum_{2\leq \vert \alpha \vert \leq N-1}
\frac{\htilde ^{\vert \alpha \vert - 2}}{\alpha!}D_{z'}^\alpha 
\big[(\partial^\alpha_\zeta \widetilde{p})(\cdots)b_{j-1}\big]\big\arrowvert_{z'=z}, \quad   b_j \arrowvert_{t=0} =0.
\end{aligned}
\right.
\end{equation}
This choice will imply that
\begin{equation}\label{n200}
T_N = i\htilde ^{N+1}\sum_{2\leq \vert \alpha \vert \leq N-1}
\frac{\htilde ^{\vert \alpha \vert - 2}}{\alpha!}
D_{z'}^\alpha \big[(\partial^\alpha_\zeta \widetilde{p})(\cdots)b_{N}\big]\big\arrowvert_{z'=z}.
\end{equation}

 \begin{prop}\label{b:existe}
 The system \eqref{bk} has a unique solution   with $ b_j \in S^0_{2\delta-1}.$
     \end{prop}
 We  prove this result by induction.   To solve these equations   we use the method of characteristics and we begin by preliminaries. 
\begin{lemm}\label{aj}
 We have $a_i \in S^0_{2\delta-1} $ for $i=1,\ldots,d.$
  \end{lemm}
\begin{proof}
This follows from Lemma \ref{derp}, Proposition \ref{est:compo} with $f = \frac{\partial p}{\partial \zeta_i}, U(t,z,\xi,\htilde) = z, V(\cdots) =  \frac{\partial \phi}{\partial z}$ and Corollary \ref{estkappa}.
  \end{proof}

Consider now the system of differential equations
$$
\dot{Z}_j(s) = a_j(s,Z(s),\xi,\htilde ), \quad Z_j(0) = z_{j}, \quad  1 \leq j \leq d.
$$
By Lemma \ref{aj}  $a_j   $   is bounded. Therefore this system has a unique solution 
defined on $I_{\htilde }$. Differentiating with respect to $z$   we obtain 
$$
\la \frac{\partial Z}{\partial z}(s) \ra \leq  C+   \mathcal{F}_{ 2 }(\Vert V \Vert_{E_0}+ \mathcal{N}_{2}(\gamma))\int_0^s \htilde ^{1- 2 \delta}\la \frac{\partial Z}{\partial z }(\sigma) \ra d\sigma, \quad 0< s \leq \htilde ^{\delta},
$$
since $ \vert s \vert \, \htilde ^{1- 2 \delta} \leq \htilde ^{1- \delta} = \htilde ^{ \frac{\delta}{2}},$  the Gronwall inequality  shows that $ \la \frac{\partial Z}{\partial z }(s) \ra$ is uniformly bounded.
Using again the equation satisfied by $\frac{\partial Z}{\partial z }(s)$ we deduce that
\begin{equation}\label{est:dZ}
\la \frac{\partial Z}{\partial z }(s) - Id \ra 
\leq     \mathcal{F}_{ 2 }(\Vert V \Vert_{E_0}+ \mathcal{N}_{2}(\gamma)) \, {\htilde }^{ \frac{\delta}{2}}, \quad 0< s \leq \htilde ^{\delta}.
\end{equation}
This shows that the map $z  \mapsto Z(s;z , \xi, \htilde)$ is a global diffeomorphism from $\xR^d$ to itself so
\begin{equation}\label{diffeo3}
Z(s;z , \xi, \htilde ) = z \Longleftrightarrow z  = \omega(s; Z;\xi,\htilde ).
\end{equation}
An analogue computation shows that 
\begin{equation}\label{est:dxi}
\la \frac{\partial Z}{\partial \xi}(s) \ra 
\leq     \mathcal{F}_{ 2 }(\Vert V \Vert_{E_0}+ \mathcal{N}_{2}(\gamma)) \, {\htilde }^{ \delta}, \quad 0< s \leq \htilde ^{\delta}.
\end{equation}
 \begin{lemm}\label{est:derZ}
 
 The function $(s,z,\xi,\htilde) \mapsto Z(s,z,\xi,\htilde)$ belongs to $\dot{S}^{\delta/2}_{2\delta-1}.$
 \end{lemm}
 \begin{proof}
We have to prove that for   $  \vert \alpha \vert + \vert \beta \vert =  k \geq 1 $ we have the estimate
\begin{equation}\label{est:dalphaZ}
\vert D^\alpha_z D^\beta_\xi Z(s;z,\xi,\htilde)  \vert \leq \mathcal{F}_{ k}(\Vert V \Vert_{E_0}+ \mathcal{N}_{k+1} (\gamma))\, \htilde^{\vert \alpha \vert (1-2\delta) + \frac{\delta}{2}}.
\end{equation}
  Indeed this is true for $k=1$ by \eqref{est:dZ}, \eqref{est:dxi}. Assume this is true up to the order $k$ and let $\vert \alpha \vert + \beta \vert =k+1\geq 2.$ Set $U(s) =D^\Lambda  Z(s;z,\xi,\htilde)$ where $\Lambda =(\alpha, \beta).$  It satifies the system $\dot{U}(s) = \frac{\partial a}{\partial z} (s;Z(s),\xi,\htilde)U(s) + F(s),    U(0)=0 $ where $F(s)$ is a finite linear combination of terms of the form 
$$(1) =(D^A a) (\cdots) \prod_{j=1}^r (\partial^{L_j}Z(s))^{p_j}  (\partial^{L_j}\xi)^{q_j} $$
where $A =(a,b),   2 \leq \vert A \vert \leq \vert \Lambda \vert, L_j = (l_j,m_j), 1 \leq \vert L_j \vert \leq k$ and
$$\sum_{j=1}^r (\vert p_j \vert + \vert q_j \vert)  L_j = (\alpha, \beta), \quad \sum_{j=1}^r p_j = \alpha, \quad \sum_{j=1}^r q_j = \beta.$$
 First of all, by Lemma \ref{aj} we can write
 $$ \vert (D^A a) (\cdots) \vert \leq \mathcal{F}_k\big(\Vert V \Vert_{E_0} +   \mathcal{N}_{k+1}(\gamma)\big)\htilde^{\vert a \vert (1-2 \delta)}.$$
 Using the induction and the fact that $\xi \in \dot{S}^0_{2\delta-1}$ we can estimate the product occuring in $(1)$ by $\mathcal{F}_k\big(\Vert V \Vert_{E_0} +   \mathcal{N}_{k+1}(\gamma)\big) \htilde^{M} $ where $$M = \sum_{j=1}^r \Big\{\vert p_j\vert  \big (\vert l_j \vert (1-2 \delta) + \frac{\delta}{2}\big ) + \vert q_j \vert \vert l_j \vert (1-2 \delta)\Big\} = \vert \alpha \vert   (1-2 \delta) + \vert a \vert \frac{\delta}{2}.$$
 It follows that $  \int_0^s \vert F(t)\vert dt  \leq \mathcal{F}_k\big(\Vert V \Vert_{E_0} +   \mathcal{N}_{k+1}(\gamma)\big)\htilde^{\vert \alpha \vert (1-2 \delta) + \delta}$ and we conclude by the Gronwall inequality.
 \end{proof}
 \begin{coro}\label{est:omega}
 The function $\omega$ defined in \eqref{diffeo3} belongs to $ \dot{S}^{\delta/2}_{2\delta-1}.$
\end{coro}
\begin{proof}
The proof is the same as that of Corollary \ref{estkappa}.
\end{proof}

 \begin{proof}[Proof of Proposition \ref{b:existe}]
Now, with the notations in \eqref{coeffL} and \eqref{bk} we have 
$$
\frac{d}{ds} \big[b_j(s,Z(s)) \big] = \big( \frac{\partial u}{\partial t} + a \cdot \nabla u \big)(s, Z(s) )= -(c_0 u)(s,Z(s)) +F_{j-1}(s,z(s)),\quad  j \geq 0
$$
with $F_{-1} =0.$ It follows that 
$$
\frac{d}{ds} \Big[e^ {  \int_0^s c_0(\sigma,Z(\sigma)) d \sigma } b_j(s,Z(s)) \Big]
= e^{  \int_0^s c_0(\sigma,Z(\sigma)) \, d \sigma  }F_{j-1}(s,Z(s)),
$$
Using \eqref{diffeo3} we see that the unique solution of \eqref{bk} is given by 
\begin{equation}\label{bj=}
\left\{
\begin{aligned}
b_0(s,z,\xi,\htilde ) &= \chi(\xi)  \,  \text{exp} \Big({ \int_0^s   c_0(t, Z(t; \omega(s,z,\xi,\htilde ), \xi, \htilde))  \,dt  }\Big), 
 \\
   b_j(s,z,\xi,\htilde ) &= \int_0^s 
e^{ \int_s^\sigma   c_0(t, Z(t; \omega(s,z,\xi,\htilde ), \xi, \htilde)  \,dt  } 
F_{j-1}( \sigma, Z(\sigma; \omega(s,z,\xi, \htilde ),\xi, \htilde)  \,d \sigma.
\end{aligned}
\right.
\end{equation}

 The last step in the proof of Proposition \ref{b:existe} is contained in the following lemma.
 \begin{lemm}\label{derb}
We have $b_j \in S^0_{2\delta-1}.$
\end{lemm}
\begin{proof}
  Step1: we show that 
  \begin{equation}\label{est:e^c0}   
 e^{ \int_s^\sigma   c_0(t, Z(t; \omega(s,z,\xi,\htilde ), \xi, \htilde)  \,dt  } \in S^0_{2\delta-1}.
 \end{equation}
 According to Remark \ref{remarque1} this will be implied by $ \int_s^\sigma c_0(t, Z(t; \omega(s,z,\xi,\htilde ), \xi, \htilde)  \,dt   \in S^{\delta/2}_{2\delta-1}.$  By Lemma \ref{est:derZ} we have $Z\in \dot{S}^{\delta/2}_{2\delta-1}$ and $\omega\in \dot{S}^{\delta/2}_{2\delta-1}.$ Moreover $\xi \in \dot{S}^{0}_{2\delta-1}$.  By Proposition \ref{est:compo} the function $Z(t; \omega(s; z,\xi,\htilde),\xi,\htilde)$ belongs to $\dot{S}^{\delta/2}_{2\delta-1}$. Now by Corollary \ref{estkappa} we have $\frac{\partial \phi}{\partial z} \in \dot{S}^0_{2\delta-1}$ and $\frac{\partial^2 \phi}{\partial z^2} \in  {S}^{-\delta/2}_{2\delta-1} $ (since $1-2\delta= -\delta/2.$)  It follows from Proposition \ref{est:compo} that  for $s\in[0, \htilde^\delta]$ 
\begin{equation}\label{U=}
\left\{
\begin{aligned}
 &U_1(t;  z,\xi,\htilde )=  \frac{\partial \phi}{\partial z}(t,Z(t; \omega(s,z,\xi,\htilde ), \xi, \htilde), \xi,\htilde)\in  \dot{S}^0_{2\delta-1},\\
  &U_2(t;  z,\xi,\htilde)=  \frac{\partial^2 \phi}{\partial z^2}(t,Z(t; \omega(s,z,\xi,\htilde ), \xi, \htilde), \xi,\htilde)\in   {S}^{-\delta/2}_{2\delta-1}.
 \end{aligned}
 \right.
 \end{equation}
Now by Lemma \ref{derp}   the functions 
$\frac{\partial^2 \widetilde{p}}{\partial\zeta \partial z' }(t,z,z,\zeta, \htilde) $ 
(resp.\ $\frac {\partial^2  {p}}{\partial\zeta  \partial \zeta }(t,z,\zeta,\htilde) $) 
satisfy the condition of Proposition \ref{est:compo} with $m = 1-2\delta$ 
(resp.\ $m=0.$) Using \eqref{U=} and the fact that 
$z\in  \dot{S}^{\delta/2}_{2\delta-1}$ 
we deduce that 
\begin{align*}
&\int_s^\sigma \frac{\partial^2 \widetilde{p}}{\partial\zeta \partial z' }\big(t,z,z,U_1\big(t,z,\xi,\htilde\big), \htilde\big)\, dt 
\in {S}^{ \frac{\delta}{2}}_{2\delta-1},\\
& \int_s^\sigma \frac{\partial^2  {p}}{\partial\zeta^2   }
\big(t,z, U_1\big(t,z,\xi,\htilde\big), \htilde\big)U_2\big(t;  z,\xi,\htilde\big)\, dt \in {S}^{\frac {\delta}{2}}_{2\delta-1}.
\end{align*}
This shows that $\int_s^\sigma   c_0\big(t, Z(t; \omega\big(s,z,\xi,\htilde \big), \xi, \htilde\big)  \,dt \in S^ {\frac{\delta}{2}}_{2\delta-1}$ as claimed.

Step 2: we show that for $\vert a \vert +\vert b \vert = k \geq 0$ we have, with $\Lambda =(a,b)\in \xN^d \times \xN^d$
\begin{equation}\label{est:Fj}
\int_0^s \la D^\Lambda \big [G_{j-1}\big( \sigma;Z\big(\sigma; \omega\big(s;z,\xi, \htilde\big ),\xi, \htilde\big)\big)\big]\ra \,d \sigma 
\leq \mathcal{F}_{ k}\big(\Vert V \Vert_{E_0}+ \mathcal{N}_{k+1} (\gamma)\big) \htilde^{\vert a \vert(1-2\delta)}
 \end{equation}
 where for $\vert \rho \vert \geq 2,$
\begin{equation}\label{FjGj}
G_{j-1}(\sigma, z, \xi, \htilde)=  
\htilde ^{\vert \rho \vert - 2} D_{z'}^\rho
\big[(\partial^\rho_\zeta \widetilde{p})\big(\sigma;z,z', \theta\big(\sigma;z,z',\xi, \htilde\big), \htilde \big)b_{j-1}(\sigma; z', \xi, \htilde)\big]\big\arrowvert_{z'=z}. 
  \end{equation}
 We claim that for $\Lambda = (\alpha, \beta), \vert \alpha \vert + \vert \beta\vert = k \geq 0,$
 \begin{equation}\label{est:Gj1}
\la D^\Lambda G_{j-1}(\sigma, z, \xi, \htilde)\ra \leq \htilde^{-\delta + \vert \alpha \vert (1-2\delta)} 
\mathcal{F}_{ k}\big(\Vert V \Vert_{E_0}+ \mathcal{N}_{k+1} (\gamma)\big) .\end{equation}
 Indeed $ D^\Lambda G_{j-1}$ is a finite sum of terms of the form $H_1 \times H_2$ with
     $$ H_1  =  \htilde ^{\vert \rho \vert - 2} D^{\Lambda_1}D_{z'}^{\rho_1}  \big[(\partial^\rho_\zeta \widetilde{p})(\sigma,,z,z', \theta(\sigma;z,z',\xi, \htilde), \htilde \big)) \big] \arrowvert_{z'=z}, \quad 
   H_2   = D^{\Lambda_2}D_{z}^{\rho_2}b_{j-1}(\sigma, z, \xi, \htilde),
$$
where $ \Lambda_i =(\alpha_i, \beta_i) \quad \vert \Lambda_1 \vert + \vert \Lambda_2 \vert = \vert \Lambda  \vert, \quad \vert \rho_1  \vert +\vert \rho_2 \vert = \vert \rho  \vert.$

By the induction we have Ä
$$
\vert H_2 \vert \leq 
\mathcal{F}_{ k}\big(\Vert V \Vert_{E_0}+ \mathcal{N}_{k+1} (\gamma)\big)
\, \htilde^{(\vert \alpha_2 \vert + \vert \rho_2\vert)(1-2\delta)}
$$
and since $z'\in \dot{S}_{2\delta-1}, \theta \in \dot{S}^0_{2\delta-1}$ 
using Proposition \ref{derp} we see that 
$$
\vert H_1 \vert \leq  
\mathcal{F}_{ k}\big(\Vert V \Vert_{E_0}+ \mathcal{N}_{k+1} (\gamma)\big) 
\, \htilde^{\vert \rho \vert -2 + (\vert \alpha_1 \vert + \vert \rho_1\vert)(1-2\delta)}.
$$
Now since $ \vert \rho \vert\geq 2$ and $\delta = \frac{2}{3},$ we have 
$\vert \rho \vert - 2 + \vert \alpha \vert (1-2\delta) + \vert \rho \vert (1-2\delta) \geq \vert \alpha \vert (1-2\delta) -\delta$ which proves \eqref{est:Gj1}.
   
Eventually since the function $Z\big(t; \omega(s; z,\xi,\htilde),\xi,\htilde\big)$ 
belongs to $\dot{S}^{\delta/2}_{2\delta-1}$ (see  Step 1), 
we deduce from  Proposition \ref{est:compo}, 
with $m=-\delta,$ that \eqref{est:Fj} holds. 
Then Step 1 and Step 2 prove the lemma. Notice that $b_j$ can be written  $ \chi(\xi)b_j^0.$
\end{proof}
Thus the proof of Proposition \ref{b:existe} is complete.
\end{proof}

Summing up we have proved that with the choice of $\phi$ and $b$ 
made in Proposition \ref{eqeiko} and in \eqref{formedeb} we have 
\begin{equation}\label{UN1}
e^{-i{\htilde}^{-1}\phi }(\htilde  \partial_t + \htilde c +iP)\big(e^{i{\htilde}^{-1}\phi }\widetilde{b} \big) = \htilde T_N \Psi_0  +U_N + R_N+S_N 
\end{equation}
 where $R_N$ is defined in \eqref{J=}, $S_N$ in \eqref{SN},   $T_N$ in \eqref{n200} and $U_N$ in \eqref{UN}.

\section{The dispersion estimate}

The purpose of this section is to prove the following result. Recall that $\delta = \frac{2}{3}.$
\begin{theo}\label{dispersive}
Let $\chi \in C_0^\infty(\xR^d)$ be such that $\supp \chi \subset \{\xi: \mez \leq \vert \xi \vert \leq 2\} $. Let $t_0 \in \xR$, 
$u_0 \in L^1(\xR^d)$ and set $u_{0,h} = \chi(hD_x)u_0$. 
Denote by $S(t,t_0)u_{0,h}$ the solution of the problem
\begin{equation*}
\Big(  \partial_t + \mez( T_{V_\delta} \cdot \nabla + \nabla \cdot T_{V_\delta}) 
+i T_{\gamma_\delta} \Big)U_{h}(t,x) =0,  \quad
U_{h}(t_0,x) = u_{0,h}(x).
\end{equation*}
Then there exist  $\mathcal{F}: \xR^+ \to \xR^+$ $ k= k(d) \in \xN$ and $h_0 >0$ such that  
$$
\Vert S(t,t_0)u_{0,h} \Vert_{L^\infty(\xR^d)} 
\leq  \mathcal{F}\big(\Vert V\Vert_{E_0} + \mathcal{N}_k(\gamma)\big) \, h^{-\frac{3d}{4}}\vert t- t_0 \vert^{-\frac{d}{2}} \Vert u_{0,h} \Vert_{L^1(\xR^d)},
$$
for all $0<\vert t-t_0 \vert \leq h^{\frac{\delta}{2}}$  and all $0<  h  \leq h_0.$
\end{theo}

This result will be a consequence of the following one in the variables $(t,z)$. 
\begin{theo}\label{dispersive1}
Let $\chi \in C_0^\infty(\xR^d)$ be such that $\supp \chi \subset \{\xi: \mez \leq \vert \xi \vert \leq 2\} $. 
Let $t_0 \in \xR$, 
$w_0 \in L^1(\xR^d)$ and set $w_{0,\htilde } = \chi(\htilde D_z)w_0$. 
Denote by $\widetilde{S}(t,t_0)w_{0,\htilde }$ the solution of the problem
\begin{equation*}
\big( \htilde  \partial_t + \htilde c+iP\big) \widetilde{U}(t,z) =0, \quad \widetilde{U}(t_0,z) = w_{0,\htilde }(z).
\end{equation*}
Then there exist  $\mathcal{F}: \xR^+ \to \xR^+$, $ k= k(d) \in \xN$ and $h_0 >0$ such that  
$$
\Vert \widetilde{S}(t, t_0)w_{0,\htilde } \Vert_{L^\infty(\xR^d)} 
\leq  \mathcal{F}\big(\Vert V\Vert_{E_0} + \mathcal{N}_k(\gamma)\big)\,   \htilde ^{-\frac{d}{2}}\vert t-t_0 \vert^{-\frac{d}{2}} \Vert w_{0,\htilde } \Vert_{L^1(\xR^d)}
$$
for all $0<\vert t -t_0 \vert\leq {\htilde }^{ \delta}$ and $0<{\htilde }  \leq \htilde _0$.
\end{theo}
Indeed suppose that Theorem \ref{dispersive1} is proved.  We can assume $t_0 =0$. 
According to the two change of variables $x =X(t,y)$ and $z= \widetilde {h}^{-1}y$ we have for any smooth function $U$ (see \eqref{semiclass} and \eqref{c=})
$$
\big( \htilde  \partial_t +\htilde c+ iP\big) \big [U (t,X(t, \htilde z)\big] 
= \htilde   \Big( \big( \partial_t + \mez( T_{V_\delta} \cdot \nabla + \nabla \cdot T_{V_\delta}) 
+i T_{\gamma_\delta}\big) U \Big)\big(t, X(t, \htilde z)\big).
$$
It follows that 
$$
\big( \widetilde{S}(t, 0)w_{0,\htilde }\big) (t,z) = \big( S(t,0)u_{0,h}\big) (t,X(t,\htilde z)).
$$
Moreover since $w_0(z) = u_0(\htilde z)$ we have
$$
w_{0,\htilde }(z)=(\chi(\htilde D_z)w_0) (z) = ( \chi(hD_x)u_0)( \htilde z) =u_{0,h}(\htilde z).
$$
 
Therefore using Theorem \ref{dispersive1} we obtain
\begin{equation*}
\begin{aligned}
\Vert  S(t, 0)u_{0,h} \Vert_{L^\infty(\xR^d)}  &= \Vert  \tilde{S}(t, 0)w_{0,\htilde }\Vert _{L^\infty(\xR^d)}  
\leq \mathcal{F}(\cdots)\,   \htilde ^{-\frac{d}{2}}t^{-\frac{d}{2}}\Vert  w_{0,\htilde } \Vert_{L^1(\xR^d)}   \\
& \leq \mathcal{F}(\cdots)\, \htilde ^{-\frac{d}{2}}t^{-\frac{d}{2}} \htilde ^{-d}\Vert u_{0,h} \Vert_{L^1(\xR^d)}= \mathcal{F}(\cdots)\,  h^{-\frac{3d}{4}}t^{-\frac{d}{2}} \Vert u_{0,h} \Vert_{L^1(\xR^d)}
\end{aligned}
\end{equation*}
since $\htilde  = h^\mez.$ Thus Theorem \ref{dispersive} is proved.
 \begin{proof}[Proof of Theorem \ref{dispersive1}]
We set 
\begin{equation}\label{K1}
\mathcal{K}w_{0,\htilde }(t,z) 
= (2 \pi \htilde )^{-d} \iint e^{i{\htilde}^{-1}(\phi(t,z,\xi,\htilde ) -y\cdot \xi)}\widetilde{b} \big(t,z,y,\xi,\htilde\big)
\chi_1(\xi)w_{0,\htilde}(y) \,dy \, d\xi
\end{equation}
where $\chi_1$ belongs to $C^\infty_0(\xR^d)$ with $\chi_1 \equiv 1$ 
on the support of $\chi$ and $\tilde{b}$ is defined in \eqref{tildeb}. We can write
\begin{equation} \label{K2}
\begin{aligned}
\mathcal{K}w_{0,\htilde }(t,z) &= \int K\big(t,z,y,\htilde \big) w_{0,\htilde }(y) dy \quad \text{with}\\
K(t,z,y,\htilde ) 
&=  (2 \pi \htilde )^{-d} \int e^{i\htilde ^{-1} (\phi (t,z,\xi,\htilde )-y\cdot \xi )}\widetilde{b}  (t,z,y,\xi,\htilde\big)
\chi_1(  \xi) \, d\xi.
\end{aligned}
\end{equation}
It follows from the \eqref{HessphiND} for the hessian, \eqref{formedeb}, 
Proposition~\ref{b:existe} and the stationary 
phase theorem that 
$$
\big\vert  K(t,z,y,\htilde ) \big\vert 
\leq \mathcal{F}(\cdots)\,  \htilde ^{-d} \htilde ^{\frac{d}{2}} t^{- \frac{d}{2}} \leq  \mathcal{F}(\cdots)  \htilde ^{-\frac{d}{2}} t^{- \frac{d}{2}}
$$
for all $0<t \leq h^{ \frac{\delta}{2}}$, $z,y \in \xR^d$ and $0< \htilde  \leq \htilde _0$.

Therefore we obtain
\begin{equation}\label{estimK}
\big\Vert  \mathcal{K}w_{0,\htilde }(t,\cdot) \big\Vert_{L^\infty(\xR^d)}
\leq \mathcal{F}(\cdots)\,   \htilde ^{-\frac{d}{2}} t^{- \frac{d}{2}} \big\Vert w_{0,\htilde } \big\Vert_{L^1(\xR^d)}
\end{equation}
for all $0<t \leq \htilde^{\delta} $ and $0< \htilde  \leq \htilde _0.$

We can state now the following result.
\begin{prop}\label{EQK}
Let $\sigma_0 $ be an integer such that $\sigma_0> \frac{d}{2}.$ Set
$$
\big( \htilde  \partial_t + \htilde c+ iP\big)\Big( \mathcal{K}w_{0,\htilde } \Big)(t,z) = F_{\htilde }(t,z).
$$
Then there exists  $k= k(d) \in \xN$ and  for any $N \in \xN,$  $\mathcal{F}_N:\xR^+ \to \xR^+$  such that
$$
\sup_{0< t \leq \htilde ^{\delta}} \big\Vert F_{\htilde }(t, \cdot ) \big\Vert_{H^{\sigma_0}(\xR^d)}
\leq  \mathcal{F}_N\big(\Vert V\Vert_{E_0} + \mathcal{N}_k(\gamma)\big)\,  \htilde ^N \Vert w_{0,\htilde } \Vert_{L^1(\xR^d)}.
$$
\end{prop}
 We shall use the following result.
\begin{lemm}\label{estint}
Let $k_0>\frac{d}{2}$. Let us set $m(t,z,y,\xi,\htilde ) = \frac{\partial \phi}{\partial \xi} (t,z ,\xi, \htilde ) -y$ 
and 
$$
\Sigma = \left\{(t,y,\xi,\htilde ):  0< \htilde \leq {\tilde {h}}_0, 0
\leq t \leq \htilde ^{\delta}, y \in \xR^d, \vert \xi \vert \leq C \right\}.
$$
Then
$$
\sup_ {(t,y,\xi,\htilde ) \in \Sigma} \int_{\xR^d} \frac{dz}{\langle m(t,z,y,\xi, \htilde )\rangle^{2k_0}} \leq  \mathcal{F}\big(\Vert V\Vert_{E_0} + \mathcal{N}_2(\gamma)\big).
$$
\end{lemm}
\begin{proof}
By \eqref{eiksuite} we have $\frac{\partial \phi}{\partial z}\big(t,z,\xi,\htilde\big) 
= \zeta\big(t;\kappa\big(t;z,\xi,\htilde\big),\xi,\htilde\big)$ so 
$\frac{\partial^2 \phi}{\partial\xi \partial z} = \frac{\partial \zeta}{\partial z}  \frac{\partial \kappa}{\partial \xi}  +  \frac{\partial \zeta}{\partial \xi} .$ We deduce from Proposition \ref{estflow} and Corollary \ref{estkappa} that  $\vert \frac{\partial^2 \phi}{\partial\xi \partial z} - Id\vert \leq  \mathcal{F} \big(\Vert V \Vert_{E_0} +   \mathcal{N}_2(\gamma)\big) \, \htilde^{\frac{\delta}{2}}. $ It follows  that the  map $z \mapsto  \frac{\partial \phi}{\partial \xi} (t,z,\xi, \htilde )$  is proper and therefore is a global diffeomorphism from $\xR^d $ to itself. Consequently, 
we can perform the change of variable $X=  \frac{\partial \phi}{\partial \xi} (t,z,\xi, \htilde )$ and the lemma follows.
\end{proof}
\begin{proof}[Proof of Proposition \ref{EQK}]
 
According to \eqref{K2} Proposition \ref{EQK} will be proved if we have
\begin{equation}\label{estK}
\sup_{(t,y,\tilde{h}) \in \Sigma}  \big\Vert \big(\tilde{h} \partial_t +\htilde c+  iP\big)K(t,\cdot,y,\tilde{h}\big)\big\Vert_{H^{\sigma_0}}
\leq \mathcal{F}_N(\cdots) \, \tilde{h}^N.
\end{equation}
Now, setting $L= \tilde{h} \partial_t + \htilde c +  iP(t,z,D_z)$, we have
\begin{equation}\label{egalite1}
LK\big(t,z,y,\tilde{h}\big)= \big(2\pi \tilde{h}\big)^{-d}\int  e^{-i{\htilde}^{-1}  y\cdot \xi}
L\Big(e^{i{\htilde}^{-1} \phi(t,z,\xi,\tilde{h})}\widetilde{b} \big(t,z,y,\xi,\tilde{h}\big)\Big) \chi_1(\xi)\, d\xi 
\end{equation}
and according to \eqref{RS}, \eqref{CT1}, \eqref{n200} we have,
\begin{equation}\label{egalite2}
L\Big(e^{i{\htilde}^{-1}\phi(t,z,\xi,\tilde{h})}\widetilde{b} \big(t,z,y,\xi,\tilde{h}\big)\Big)
= e^{i{\htilde}^{-1}\phi(t,z,\xi,\tilde{h})}\big(R_N +S_N + \tilde{h}T_N \Psi_0 +U_N\big)(t,z,y,\xi,\htilde) 
\end{equation}
where $R_N$, $S_N$, $T_N$, $U_N$ are defined in \eqref{J=}, \eqref{SN}, \eqref{CT1}, \eqref{UN}.

\begin{lemm}\label{estRNSN}
Let $\sigma_0,k_0 $ be   integers, $\sigma_0 > \frac{d}{2},$  $ k_0 >\frac{d}{2}$. There exists  a fixed integer  $N_0(d)$   such that for any $N \in \xN$  there exists $C_N >0$ such that, if we set $\Xi = (t,z,y,\xi,\htilde),$ then
\begin{equation}\label{estRN}
  \big\langle m\big(\Xi\big)\big\rangle^{k_0} 
\left\{\vert D^\beta_z R_N(\Xi) \vert 
+ \vert D^\beta_z S_N (\Xi)\vert + \vert D_z^\beta (\tilde{h}(T_N \Psi_0) (\Xi)\big)\vert  \right\} 
 \leq\mathcal{F}_N(\cdots) \, {\tilde{h}}^{\delta N-N_0}, 
\end{equation}
 for all $\vert \beta \vert \leq \sigma_0$,  all $(t,y,\xi, \tilde{h}) \in \Sigma $ and all $z\in \xR^d$. 
  \end{lemm}
\begin{proof}
According to \eqref{J=}, $  D^\beta_z R_N(\Xi)$ is a finite linear combination 
of terms of the form \begin{equation*}
\begin{aligned}
R_{N, \beta}(\Xi)  = \tilde{h} ^{-d-\vert \beta_1\vert} 
\iint e^{i{\htilde}^{-1}(z-z')\cdot\eta} \eta^{\beta_1}\kappa(\eta) D_z^{\beta_2}&r_N(t,z, z',\eta, \xi,\htilde)\\
&{b}(t,z', \xi,\tilde{h}) \Psi_0\Big(\frac{\partial \phi}{\partial \xi} \big(t,z',\xi, \tilde{h}\big) -y\Big)\,dz' \,d\eta\\
\end{aligned}
\end{equation*}
where  $\beta_1 + \beta_2 = \beta$  and 
\begin{equation*}
D_z^{\beta_2}r_N (\cdots) = \sum_{\vert \alpha \vert =N} 
\frac{N}{\alpha!}\int_0^1 (1-\lambda)^{N-1} D_z^{\beta_2}
\Big[\big(\partial^\alpha_\eta \widetilde{p}\big)(t,z,z',\eta+ \lambda\theta(t,z,z', \xi,\htilde)) \Big ]\eta^\alpha \,d\lambda.
\end{equation*}

Using the equality $\eta^{\alpha + \beta_1}  e^{i{\htilde}^{-1}(z-z')\cdot\eta} 
= (-\tilde{h}D_{z'})^{\alpha+ \beta_1}   e^{i{\htilde}^{-1}(z-z')\cdot\eta}$ 
we see that $R_{N, \beta}$ is a linear combination of  terms of the form
\begin{equation}
\begin{aligned}
R'_{N, \beta}(\Xi)= \tilde{h}^{N   -d}\int_0^1 
\iint  e^{i{\htilde}^{-1}(z-z')\cdot\eta}&\kappa(\eta)D_{z'}^{\alpha + \beta_1} 
D_z^{\beta_2}\Big[(\partial_\eta^\alpha  \widetilde{p})(t,z,z',\eta+ \lambda\theta(t,z,z', \xi,\htilde))\\
&b(t,z',\xi,\tilde{h}) \Psi_0\big(\frac{\partial \phi}{\partial \xi} (t,z'\xi, \tilde{h}) -y\big)\Big] 
\, d\lambda \, dz' \, d\eta.
\end{aligned}
\end{equation}
Then we insert in the integral  the quantity 
$  \big (\frac{\partial \phi}{\partial {\xi }} (t,z ,\xi, \tilde{h}) -y \big)^{\gamma }$ where $  \vert \gamma\vert =  k_0. $  It is a finite linear combination of terms of the form
$$
\Big(\frac{\partial \phi}{\partial {\xi }} (t,z ,\xi, \tilde{h}) 
- \frac{\partial \phi}{\partial{\xi }} (t,z',\xi, \tilde{h})\Big)^{\gamma_1 } 
\Big(\frac{\partial \phi}{\partial {\xi }} (t,z',\xi, \tilde{h}) -y\Big)^{\gamma_2 }.
$$
Using the Taylor formula we see that $ \langle m(\Xi) \rangle^{k_0} R'_{N, \beta}(\Xi)$ 
is a finite linear combination of terms of the form
\begin{equation*}
\begin{aligned}
& {\tilde{h}}^{N -d} \int_0^1 \iint  (z-z')^{\nu}  e^{i{\htilde}^{-1}(z-z')\cdot\eta} F(t,z,z',\xi,\tilde{h}) 
\kappa(\eta)\Big(\frac{\partial \phi}{\partial {\xi }} (t,z',\xi, \tilde{h}) -y\Big)^{l}\\ 
&\qquad\qquad D_{z'}^{\alpha+ \beta_1} D_z^{\beta_2} \Big[(\partial_\eta^\alpha  \widetilde{p})(t,z,z',\eta+ \lambda\theta(t,z,z',\xi,\htilde)) 
b(t,z',\xi,\tilde{h}) \Psi_0\big(\frac{\partial \phi}{\partial \xi} (t,z',\xi, \tilde{h}) -y\big)\Big] \, d\lambda \, dy \, d\eta,  
\end{aligned}
\end{equation*}
where, by Corollary \ref{estkappa} (ii), $F$  is a bounded function.

Eventually we use  the identity 
$(z-z')^{\nu}  e^{i{\htilde}^{-1}(z-z')\cdot\eta} = (\htilde D_\eta)^\nu  e^{i{\htilde}^{-1}(z-z')\cdot\eta},$ we 
integrate by parts in the integral with respect to $\eta$ and we use Remark \ref{rem:est:ptilde}, Remark \ref{rem:est:theta}, the estimate $(ii)$ in Corollary \ref{estkappa}, the fact that $b\in S^0_{2\delta-1}$  and the fact that $N+ N(1-2\delta) = \delta N$ to deduce that   
\begin{equation}\label{est:RN}
  \big\langle m\big(\Xi\big)\big\rangle^{k_0}  \vert D_z^\beta R_N(\Xi) \vert \leq \mathcal{F}_N(\cdots) \, {\htilde }^{\delta N -N_d}
\end{equation}
where    $N_d$ is a fixed number depending only on the dimension.
 
Let us consider the term $S_N.$ Recall that $D_z^\beta S_N$ 
is a finite linear combination for $\vert \alpha \vert \leq N-1$ and $ \vert \gamma \vert = N$ of terms of the form
\begin{multline*}
  S_{N,\alpha, \beta,Ê\gamma} = {\htilde }^{N+\vert \alpha \vert}\int_0^1 \int (1-\lambda)^{N-1}\mu^\gamma \widehat{\kappa}(\mu) D_{z'}^{\alpha + \beta + \gamma}
\Big[\big(\partial^\alpha_\zeta \widetilde{p}\big)\big(t,z,z',\theta (t,z,z',\xi,\htilde),\htilde\big)\\
b\big(t,z',\xi,\htilde\big) \Psi_0\Big(\frac{\partial \phi}{\partial \xi }\big(t,z',\xi,\htilde\big)
-y\Big)\Big] \Big\arrowvert_{z'=z+\lambda \htilde  \mu}\, d\lambda \, d\mu.
  \end{multline*}
Then we multiply  $S_{N,\alpha, \beta, \gamma}$ by $ \langle m(\Xi) \rangle^{ k_0}$ 
and we write
$$
\frac{\partial \phi}{\partial \xi }\big(t,z,\xi,\htilde\big) -y
= \frac{\partial \phi}{\partial \xi }\big(t,z,\xi,\htilde\big)-
\frac{\partial \phi}{\partial \xi }\big(t,z+\lambda \htilde \mu,\xi,\htilde\big) 
+ \frac{\partial \phi}{\partial \xi }\big(t,z+\lambda \htilde \mu,\xi,\htilde\big)-y.
$$
By the Taylor formula the first term will give rise to a power of $\lambda \htilde  \mu$ 
which will be absorbed by $\hat{\kappa}(\mu)$ and the second term will be absorbed by $\Psi_0.$ Then we use again Remark \ref{rem:est:ptilde}, Remark \ref{rem:est:theta} to conclude that 
\begin{equation}\label{est:SN}
 \big\langle m\big(\Xi\big)\big\rangle^{k_0}  \vert D_z^\beta S_N(\Xi) \vert \leq \mathcal{F}_N(\cdots) \,{\htilde }^{\delta N -N'_d}.
\end{equation}
Let us look now to the term $\htilde D_z^\beta T_N \Psi_0$. 
According to \eqref{n200} this expression is a linear combination for $2 \leq \vert \alpha \vert  \leq N$ of terms of the form
\begin{equation*}
{\htilde }^{\vert \alpha\vert +N }
D_z^\beta \Big\{D_{z'}^ \alpha \Big[(\partial_\zeta^\alpha \widetilde{p}(t,x,z ,z',\theta(t,z,z',\xi,\htilde),\htilde )
b_N(t,z',\xi,\htilde )\Big]\Big\arrowvert_{z'=z} \Psi_0\Big(\frac{\partial \phi}{\partial \xi }(t,z',\xi,\htilde ) -y\Big)\Big\}.
\end{equation*}
It follows from Remark \ref{rem:est:ptilde}, Remark \ref{rem:est:theta} and Corollary \ref{estkappa} that we have 
\begin{equation}\label{est:TN}
\big \langle m (\Xi ) \big\rangle^{ k_0}
\big\vert \htilde D_z^\betaÊ\big[T_N \Psi_0 \big] \big\vert \leq \mathcal{F}_N(\cdots) \, {\htilde }^{\delta N}.
\end{equation}
Lemma \ref{estRNSN} follows from \eqref{est:RN}, \eqref{est:SN},  \eqref{est:TN}.
\end{proof}
From   Lemma \ref{estRNSN} we can write
\begin{equation}\label{estint2}
{\htilde }^{-d}\int  \big\Vert e^{i{\htilde}^{-1}( \phi(t,z,\xi,\htilde )- y\cdot \xi)} 
\big(R_N + S_N + \htilde  T_N \Psi_0\big)\big\Vert_{H^{\sigma_0}_z} \vert \chi_1(\xi) \vert \, d \xi 
\leq \mathcal{F}_N(\cdots) \, {\htilde }^{\delta N-N_1(d)}
\end{equation}
where $N_1(d)$ is a fixed number depending only on the dimension.

To conclude the proof of Proposition \ref{EQK} we have to estimate the integrals
$$
I_{N, \beta} =   (2\pi \htilde )^{-d}\int D_z^\beta \big[ e^{i{\htilde}^{-1}( \phi(t,z,\xi,\htilde )- y\cdot \xi)}
U_N\big(t,z,y,\xi,\htilde\big)\big] \chi_1(\xi) \,d \xi.
$$
 
Now according to \eqref{RS}, \eqref{CT1} and \eqref{UN} on the support of $U_N$ 
the function $\Psi_0$ 
is differentiated at least one time. Thus on this support 
one has $\big\vert \frac{\partial \phi}{\partial \xi} (t,z ,\xi, \htilde ) -y\big\vert \geq 1$.  
Then we can use the vector field
$$
X= \frac{\htilde }{\big\vert \frac{\partial \phi}{\partial \xi} (t,z ,\xi, \htilde ) -y\big\vert^2} 
\sum_{j=1}^d \Big(\frac{\partial \phi}{\partial {\xi_j}} (t,z ,\xi, \htilde ) -y_j\Big) D_{\xi_j}
$$
to integrate by parts in $I_{N,\beta}$   to obtain
$$
\vert I_{N,\beta} \vert \leq \mathcal{F}_N(\cdots) \,{\htilde }^{\delta N}.
$$
The proof of Proposition \ref{EQK} is complete.
\end{proof}

We show now that
\begin{equation}\label{rh}
\mathcal{K}w_{0,\htilde }(0,z)=  w_{0,\htilde }(z) + r_{\htilde }(z)
\end{equation}
with
\begin{equation}\label{estrh}
\big\Vert r_{\htilde }\big\Vert_{H^{\sigma_0}(\xR^d)}
\leq \mathcal{F}_N(\cdots) \, {\htilde }^N \big\Vert w_{0,\htilde }\big\Vert_{L^1(\xR^d)} \quad \forall N \in \xN.
\end{equation}
 
It follows from the initial condition on $\phi$ given in \eqref{eikonale} and the initial condition on $b$ that \eqref{rh} is true with
$$
r_{\htilde }(z)=  (2 \pi \htilde )^{-d} \iint e^{i \htilde ^{-1}(z -y) \cdot \xi} \chi_1(\xi)(1-\Psi_0(z-y))w_{0,\htilde }(y) \, dy \, d\xi.
$$
We see easily that for $\vert \beta \vert \leq \sigma_0, D_z^\beta r_{\htilde }(z)$ is a finite linear combination of terms of the form
$$
r_{\htilde ,\beta}(z) = {\htilde }^{-d-\vert \beta_1\vert}  
\iint e^{i \htilde ^{-1}(z -y) \cdot \xi} \xi^{\beta_1}\chi_1(\xi)\Psi_{\beta}(z-y)
w_{0,\htilde }(y) \, dy \,d\xi, \quad \vert \beta_1  \vert \leq \vert \beta \vert
$$
where $\Psi_\beta \in C^\infty_b(\xR^d)$ and $\vert z-y \vert \geq 1$ on the support of $\Psi_\beta $. Then one can write
$$
r_{\htilde ,\beta}(z) = \big(F_{\htilde }\ast w_{0, \htilde }\big)(z)
$$
where 
$$
F_{\htilde }(X) = {\htilde }^{-d-\vert \beta_1\vert}  \int e^{i \htilde ^{-1}X \cdot \xi} \xi^{\beta_1}\chi_1(\xi)\Psi_{\beta}(X)\, d\xi
$$
and $\vert X \vert \geq 1$ on the support of $\Psi_\beta(X).$
Then we remark that if we set $L = \frac{1}{\vert X \vert^2} \sum_{j=1}^d X_j \frac{\partial}{\partial \xi_j}$ we have \, $\htilde L e^{i \htilde ^{-1}X \cdot \xi} =e^{i \htilde ^{-1}X \cdot \xi}$. 
Therefore one can write
$$
F_{\htilde }(X) = {\htilde }^{M-d-\vert \beta_1\vert}  
\int e^{i \htilde ^{-1}X \cdot \xi} (-L)^M\big[\xi^{\beta_1}\chi_1(\xi)\big]\Psi_{\beta}(X) \, d\xi
$$ 
from which we deduce
$$
\vert F_{\htilde }(X) \vert \leq \mathcal{F}_M(\cdots) \, {\htilde }^{M-d-\vert \beta_1 \vert} 
\frac{\vert \tilde{\Psi} (X)\vert}{\vert X \vert^M}, \quad \forall M\in \xN
$$
where $\tilde{\Psi} \in C^\infty_b(\xR^d)$ is equal to $1$ on the support of $\Psi_\beta.$

It follows then that $ \Vert F_{\htilde }\Vert_{L^2(\xR^d)} \leq \mathcal{F}_M(\cdots) \, {\htilde }^{M-d-\vert \beta_1 \vert}$ 
from which we deduce that for $\vert \beta \vert \leq \sigma_0$, 
$$
\big\Vert D^\beta r_{\htilde }\big\Vert_{L^2(\xR^d)}
\leq \mathcal{F}_N(\cdots) \, \big\Vert F_{\htilde }\big\Vert_{L^2(\xR^d)}
\big\Vert w_{0,\htilde }\big\Vert_{L^1(\xR^d)}
\leq \mathcal{F}_N(\cdots) \, {\htilde }^{M-d-\vert \beta_1 \vert}\big\Vert w_{0,\htilde }\big\Vert_{L^1(\xR^d)}
$$
which proves \eqref{estrh}.

Using Proposition \ref{EQK} and the Duhamel formula one can write
$$
\tilde{S}(t,0)w_{0,\htilde }(z)
=\mathcal{K}w_{0,\htilde }(t,z)  - \tilde{S}(t,0) r_{\htilde }(z) - \int_0^t  \tilde{S}(t,s)[F_{\htilde }(s,z)] \,ds.
$$
Now we can write
\begin{equation*}
\begin{aligned}
 \lA \int_0^t  \tilde{S}(t,s)\big[F_{\htilde }(s,z)\big] \,ds\rA_{L^\infty(\xR^d)}
 &\leq \int_0^t  \lA \tilde{S}(t,s)\big[F_{\htilde }(s,z)\big]\rA_{H^{\sigma_0}(\xR^d)} \, ds \\
 &\leq C \int_0^t\lA  F_{\htilde }(s,z) \rA_{H^{\sigma_0}(\xR^d)} \, ds 
 \leq \mathcal{F}_N(\cdots) \, \htilde ^N \big\Vert w_{0,\htilde }\big\Vert_{L^1(\xR^d)}
\end{aligned}
\end{equation*}
and, for every $N \in \xN$,
$$
\big\Vert \tilde{S}(t,0) r_{\htilde}\big\Vert_{L^\infty(\xR^d)}
\leq C \big\Vert \tilde{S}(t,0) r_{\htilde}\big\Vert_{H^{\sigma_0}(\xR^d)}
\leq C'  \big\Vert r_{\htilde}\big\Vert_{H^{\sigma_0}(\xR^d)}
\leq \mathcal{F}_N(\cdots) \,{\htilde }^N \Vert w_{0,\htilde }\big\Vert_{L^1(\xR^d)}.
$$
Then Theorem \ref{dispersive1} follows from these estimates and \eqref{estimK}. 
\end{proof}

  \section{The Strichartz estimates}

\begin{theo}\label{semicl}
Consider the problem
\begin{equation*}
\Big(  \partial_t + \mez( T_{V_\delta} \cdot \nabla + \nabla \cdot T_{V_\delta}) +i T_{\gamma_\delta} \Big)u_{h}(t,x) =f_h (t,x),  \quad
u_{h}(t_0,x) = u_{0,h}(x).
\end{equation*}
where $u_h, u_{0,h} $ and $f_h$ 
have spectrum  in  $\big\{ \xi : c_1h^{-1} \leq \vert \xi \vert \leq c_2 h^{-1}\big\}$. Let  $I_h = (0, h^{ \frac{\delta}{2}}).$
 
Then there exists  $k = k(d), h_0>0$ such that for any $s \in \xR$ and $ \eps>0$ there exists $\mathcal{F}, \mathcal{F}_\eps :\xR^+ \to \xR^+$,  such that  
\begin{align*}
& (i) \quad \text{if }  d=1: \\
  &  \quad \Vert u_h \Vert_{L^{4}(I_h, W^{s - \frac{3}{8} , \infty}(\xR ))} 
\leq \mathcal{F}\big(\Vert V\Vert_{E_0} + \mathcal{N}_k(\gamma)\big) \, \Big( \Vert u_{0,h}\Vert_{H^s(\xR)} + \Vert f_h \Vert_{L^1(I_h, H^s (\xR))} \Big), \\
& (ii) \quad \text{if }  d\geq 2:\\
&   \quad \Vert u_h \Vert_{L^{2+\eps}(I_h, W^{s - \frac{d}{2} + \frac{1}{4} - \eps, \infty}(\xR^d))} 
\leq \mathcal{F}_\eps\big(\Vert V\Vert_{E_0} + \mathcal{N}_k(\gamma)\big) \,  \Big( \Vert u_{0,h}\Vert_{H^s(\xR^d)} + \Vert f_h \Vert_{L^1(I_h, H^s (\xR^d))} \Big),    
  \end{align*}
  for any  $0<h \leq h_0$  
 \end{theo}
\begin{proof}
If $d =1,$   by the $TT^*$ argument  we deduce from the dispersive estimate   given in Theorem \ref{dispersive} that 
\begin{equation*}
\Vert u_h \Vert_{L^4(I_h, L^\infty(\xR))} \leq \mathcal{F}(\cdots)\, h^{-\frac{3}{8}}
\Big(  \Vert u_{0,h}\Vert_{L^2(\xR)} + \Vert f_h \Vert_{L^1(I_h,L^2 (\xR))} \Big).
\end{equation*}
Then multiplying this estimate by $h^s$   and using the fact that $u_h,u_{0,h},f_h$ 
are spectrally supported in     $ \{ \xi : c_1h^{-1} \leq \vert \xi \vert \leq c_2 h^{-1} \}$  we deduce
$(i).$  
 
If $d\geq 2$ we use the same argument. Then   if $(q,r)\in \xR^2$ is such that $ q>2$ and 
$\frac{2}{q} = \frac{d}{2} - \frac{d}{r}$  we obtain
\begin{equation}\label{strich1}
\Vert u_h \Vert_{L^q(I_h, L^r(\xR^d))} \leq \mathcal{F}(\cdots)\, h^{-\frac{3}{2q}}
\Big(  \Vert u_{0,h}\Vert_{L^2(\xR^d)} + \Vert f_h \Vert_{L^1(I_h,L^2 (\xR^d))} \Big).
\end{equation}
Taking $q = 2+ \eps$ we find $r= 2+ \frac{8}{(2+\eps)d-4}.$  Moreover $ h^{-\frac{3}{2q}} \leq h^{-\frac{3}{4}}.$ Then multiplying both members of \eqref{strich1} by $h^s$ we obtain 
$$
\Vert u_h \Vert_{L^{2+ \eps}(I_h,W^{s - \frac{3}{4}, r}(\xR^d))} 
\leq \mathcal{F}(\cdots)\,  \Big(  \Vert u_{0,h}\Vert_{H^s(\xR^d)} + \Vert f_h \Vert_{L^1(I_h,H^s (\xR^d))} \Big).
$$
On the other hand the Sobolev embedding shows that $W^{a+b,r}(\xR^d) \subset W^{b,\infty} (\xR^d)$ 
provided that $a >  \frac{d}{r} =  \frac{d}{2} -1 + \frac{\eps}{2+ \eps}$. 
In particular we can take $a = \frac{d}{2} -1 + \eps$. Taking $ \frac{d}{2} -1+ \eps +b =s - \frac{3}{4}$
we obtain the conclusion of the Theorem.
\end{proof}

\begin{coro}\label{strich2}
With  the notations in Theorem \ref{semicl} and  $\delta = \frac{2}{3}, I=[0,T]$ we have

  $(i)  \quad   \text{if  }   d=1$: 
  \begin{equation*}
  \begin{aligned}
     \quad  \Vert  u_{h} \Vert_{L^4(I; W^{s - \frac{3}{8}-\frac{\delta}{8}, \infty}(\xR))} \leq \mathcal{F}\big(\Vert V\Vert_{E_0} +  \mathcal{N}_k(\gamma)\big) \,  
     &\Big(
     \Vert  f_h\Vert_{L^4(I; H^{s-\frac{\delta}{2}}(\xR))} 
     +  \Vert u_h\Vert_{C^0(I; H^{s }(\xR))}\Big),
  \end{aligned}
  \end{equation*} 
     $(ii)  \quad     \text{if  }  d\geq 2$: 
     \begin{equation*}
  \begin{aligned}
     \quad   \Vert  u_{h} \Vert_{L^2(I; W^{s - \frac{d}{2} +  \frac{1}{4} -\frac{\delta}{4} - \eps, \infty}(\xR^d))}  
  \leq \mathcal{F}_\eps\big(\Vert V\Vert_{E_0} +  \mathcal{N}_k(\gamma)\big) \,  &\Big(\Vert  f_h\Vert_{L^2(I; H^{s-\frac{\delta}{2}}(\xR^d))} 
  +  \Vert u_h\Vert_{C^0(I; H^s(\xR^d))}\Big)
    \end{aligned}
 \end{equation*}
   for any $\eps>0.$ 
\end{coro}
\begin{proof}
Let $T>0$ and  $\chi \in C_0^\infty(0,2)$ equal to one on $[ \mez, \frac{3}{2}].$ For $0 \leq k \leq [Th^{-\frac{\delta}{2}}]-2$ define
$$
I_{h,k} = [kh^\frac{\delta}{2}, (k+2)h^\frac{\delta}{2}], \quad \chi_{h,k}(t) = \chi\Big( \frac{t-kh^\frac{\delta}{2}}{h^\frac{\delta}{2}}\Big), \quad  u_{h,k} = \chi_{h,k}(t)u_h.
$$
Then
\begin{equation*} 
\Big(  \partial_t + \mez( T_{V_\delta} \cdot \nabla + \nabla \cdot T_{V_\delta}) 
+i T_{\gamma_\delta} \Big)u_{h,k} 
= \chi_{h,k}f_h   +h^{-\frac{\delta}{2}}\chi'\Big( \frac{t-kh^\frac{\delta}{2}}{h^\frac{\delta}{2}}\Big)u_h
\end{equation*}
and $u_{h,k}(kh^\frac{\delta}{2},\cdot)=0$. 

Consider first the case $d=1$. Applying Theorem \ref{semicl}, $(i)$ to each $u_{h,k}$ on the interval $I_{h,k}$   we obtain, since $\chi_{h,k}(t) = 1$ for $(k+\mez)h^\frac{\delta}{2}\leq t \leq (k+\frac{3}{2})h^\frac{\delta}{2}$, 
\begin{equation*} 
\begin{aligned}
\Vert &u_{h } \Vert_{L^4(  (k+\mez)h^\frac{\delta}{2},(k+\frac{3}{2})h^\frac{\delta}{2});W^{s - \frac{3}{8} , \infty}(\xR))}  \\
& \leq \mathcal{F}(\cdots) \Big( \Vert   f_h \Vert_{L^1(  (kh^\frac{\delta}{2},(k+2)h^\frac{\delta}{2} );H^s(\xR))} 
+ h^{-\frac{\delta}{2}} \Vert \chi'\Big( \frac{t-kh^\frac{\delta}{2}}{h^\frac{\delta}{2}}\Big)u_h \Vert_{L^1(\xR;H^s(\xR))}\Big)\\
&\leq \mathcal{F}(\cdots)  \Big(h^{ \frac{3\delta}{8}} 
\Vert   f_h \Vert_{L^4(   (kh^\frac{\delta}{2},(k+2)h^\frac{\delta}{2} );H^s(\xR))} 
+ \Vert u_h\Vert_{L^\infty(I; H^s(\xR))}\Big).
\end{aligned}
\end{equation*}
    Multiplying both members of the above inequality by $h^\frac{\delta}{8}$ 
and taking into account that $u_h$ and $f_h$ 
are spectrally supported in a ring of size $h^{-1}$ we obtain 
\begin{equation}\label{global3}
\begin{aligned}
\Vert &u_{h } \Vert_{L^4((k+\mez)h^\frac{\delta}{2} ,(k+\frac{3}{2})h^\frac{\delta}{2}); W^{s - \frac{3}{8}-\frac{\delta}{8}, \infty}(\xR))} \\
&\leq    \mathcal{F}(\cdots)   \Big(\Vert  f_h\Vert_{L^4((kh^\frac{\delta}{2},(k+2)h^\frac{\delta}{2}  ); H^{s-\frac{\delta}{2}}(\xR))} 
+ h^\frac{\delta}{8}  \Vert u_h\Vert_{L^\infty(I; H^s(\xR))}\Big).
\end{aligned}
\end{equation}
Taking the power $4$ of \eqref{global3}, summing in $k$ from $0$ to  $[Th^{-\frac{\delta}{2}}]-2$ 
  we obtain (since there are $\approx T h^{-\frac{\delta}{2}}$ intervals)
\begin{equation*}\label{global4}
 \Vert  u_{h} \Vert_{L^4(I; W^{s - \frac{3}{8}-\frac{\delta}{8}, \infty}(\xR))}\\
  \leq   \mathcal{F}(\cdots)  \Big(
  \Vert  f_h\Vert_{L^4(I; H^{s-\frac{\delta}{2}}(\xR))} 
+  \Vert u_h\Vert_{C^0(I; H^s(\xR))}\Big).
  \end{equation*}
This completes the proof of $(i)$. 

The proof of $(ii)$ follows exactly the same path.  We apply Theorem \ref{semicl}, $(ii)$ to each $u_{h,k}$ on the interval $I_{h,k}.$ The only difference with the  case $d=1$ is that,   passing from the $L^1$ norm in $t$ of $f_h$  to the  $L^2$ norm, it  gives rise to a $h^\frac{\delta}{4}$ factor. Therefore we multiply the inequality by  $h^\frac{\delta}{4},$ we take the square of the new inequality and we sum in $k$.
\end{proof}

\begin{coro}\label{strich3}
Consider the problem
\begin{equation*}
\Big(  \partial_t + \mez( T_{V } \cdot \nabla + \nabla \cdot T_{V }) 
+i T_{\gamma } \Big)u_{h}(t,x) =F_h (t,x),  \quad
u_{h}(t_0,x) = u_{0,h}(x).
\end{equation*}
where $u_h, u_{0,h} $ and $F_h$ is spectrally supported in 
$ \{ \xi : c_1h^{-1} \leq \vert \xi \vert \leq c_2 h^{-1} \}.$  

  Then there exists  $k = k(d), h_0>0$ such that for any $s \in \xR$, for any $T>0$, $\eps>0$  one can find $\mathcal{F}, \mathcal{F}_\eps: \xR^+ \to\xR^+ $ such that with $I =[0,T]$
  
  $(i)$ if $d=1$: 
\begin{equation*}
\begin{aligned}
  \Vert  u_{h} \Vert_{L^4(I; W^{s - \frac{3}{8}-\frac{\delta}{8}, \infty}(\xR))} \leq \mathcal{F}\big(\Vert V\Vert_{E_0} +  \mathcal{N}_k(\gamma)\big) \,  
  \Big( \Vert  F_h\Vert_{L^4(I; H^{s }(\xR))}
+  \Vert u_h\Vert_{C^0(I; H^{s }(\xR))}\Big), 
\end{aligned} 
\end{equation*}
   $(ii)$ if $d\geq 2$: 
 \begin{equation*}
 \begin{aligned}
 \Vert  u_{h} \Vert_{L^2(I; W^{s - \frac{d}{2} +  \frac{1}{4} -\frac{\delta}{4} - \eps, \infty}(\xR^d))}  
  \leq \mathcal{F}\big(\Vert V\Vert_{E_0} +  \mathcal{N}_k(\gamma)\big) \,  
  \Big( \Vert  F_h\Vert_{L^2(I; H^{s }(\xR^d))}
+  \Vert u_h\Vert_{C^0(I; H^s(\xR^d))}\Big) .
\end{aligned}
   \end{equation*}
    \end{coro}
 \begin{proof}
Applying Corollary \ref{strich2} we see that Corollary \ref{strich3} will follow from the following estimates with $ p =  4$ if $d=1$ and $p=2$ if $d\geq 2.$
\begin{equation}\label{diffV} 
\left\{ 
\begin{aligned}
&\Vert  \big( T_V - T_{V_\delta}\big).\nabla u_h \Vert_{L^p(I; H^{s-\frac{\delta}{2}}(\xR^d))} \leq \mathcal{F}(\cdots)  \Vert u_h\Vert_{L^\infty(I; H^s(\xR^d))} \\
&\Vert  \big( T_{\nabla V} - T_{{\nabla V}_\delta}\big)  u_h \Vert_{L^p(I; H^{s-\frac{\delta}{2}}(\xR^d))} \leq \mathcal{F}(\cdots)  \Vert u_h\Vert_{L^\infty(I; H^s(\xR^d))} \\
&\Vert  \big( T_\gamma - T_{\gamma _\delta}\big) u_h \Vert_{L^p(I; H^{s-\frac{\delta}{2}}(\xR^d))} \leq \mathcal{F}(\cdots)  \Vert u_h\Vert_{L^\infty(I; H^s(\xR^d))}.  
\end{aligned}
\right.
\end{equation}
  These three estimates are proved along the same lines. Let us prove the first one.

We have 
$$
\big(T_V - T_{V_\delta}\big).\nabla u_h 
= \sum_{k=-1}^{+\infty} \big( (S_k(V) - S_{k \delta}(V))\cdot \nabla\Delta_k u_h \big).
$$
Since $u_h = \Delta_j u$ where $h=2^{-j}$ we see easily that the above sum is reduced 
to a finite number of terms where $\vert k-j \vert \leq 3.$ All the terms beeing estimated in the same way 
it is sufficient to consider for simplicity the term where $k=j.$ 
Let us set 
$$
A_j(t,x) = \big((S_j(V_i) - S_{j \delta}(V_i))\cdot \partial_i\Delta_ju_h\big)(t,x) \quad i=1,\ldots, d 
$$
where $V=(V_1,\ldots,  V_d)$.

Since the spectrum of $A_j  $ is contained in a ball of radius $C\, 2^j$ we can write for fixed $t$
\begin{equation*}
\begin{aligned}
\Vert A_j(t,\cdot) \Vert_{H^{s-\frac{\delta}{2}}(\xR^d)} &\leq C\, 2^{j(s-\frac{\delta}{2})} 
\Vert  (S_j(V_i) - S_{j \delta}(V_i))\cdot \partial_i\Delta_ju_h)(t,\cdot) \Vert_{L^2(\xR^d)}\\
& \leq C\, 2^{j(s-\frac{\delta}{2})}\Vert  (S_j(V_i) - S_{j \delta}(V_i))(t,\cdot) \Vert_{L^\infty(\xR^d)} 
2^{j(1-s)}  \Vert u_h(t,\cdot) \Vert_{H^s(\xR^d)}. 
\end{aligned}
\end{equation*}
  Now we can write  
$$
(S_j(V_i) - S_{j \delta}(V_i))(t,x) 
= \int_{\xR^d} \hat{\psi}(z) \big(V_i(t,x-2^{-j }z) - V_i(t,x-2^{-j\delta} z) \big) \,dz.
$$
where $\psi \in C_0^\infty(\xR^d)$ has its support contained in a ball of radius $1$. 
It follows easily, since $0<\delta<1,$ that 
$$
\Vert (S_j(V_i) - S_{j \delta}(V_i))(t,\cdot) \Vert_{L^\infty(\xR^d)} 
\leq \mathcal{F}(\cdots)\, 2^{-j \delta} \Vert V_i(t,\cdot) \Vert_{W^{1,\infty}(\xR^d)}.
$$
Therefore we obtain with $p= 2,4$ and $I=[0,T],$
$$
\Vert A_j\Vert_{L^p(I; H^{s-\frac{\delta}{2}}(\xR^d))} \leq \mathcal{F}(\cdots)\,  2^{j(1-\frac{3\delta}{2})}  
\Vert V_i\Vert_{L^p(I; W^{1,\infty}(\xR^d))}  \Vert u_h \Vert_{L^\infty(I; H^s(\xR^d))},
$$
and the estimate \eqref{diffV} for the first term follows from the fact that $ \frac{3\delta}{2} = 1$.
\end{proof}
The theorem below has been stated in Theorem \ref{T4} but for the convenience of the reader we state it again.
\begin{theo}\label{T20}
Let $I= [0,T]$, $d\geq 1$. 
Let $\mu$ be such that 
$\mu=\frac{1}{24}$ if $d=1$ and $\mu<\frac{1}{12}$ if $d\ge 2$.

Let  $s>1+\frac{d}{2}-\mu$ 
and $f \in L^\infty(I; H^s(\xR^d))$. 
Let $u\in C^0(I; H^s(\xR^d))$ be a solution of the problem 
$$
\Big(  \partial_t + \mez( T_{V } \cdot \nabla + \nabla \cdot T_{V }) +i T_{\gamma } \Big)u  =f.
$$
Then one can find $k = k(d)$ such that
\begin{equation*}
\begin{aligned}
\Vert  u  \Vert_{L^p(I;  W^{s-\frac{d}{2}+\mu, \infty}(\xR^d))}
\leq    \mathcal{F}\big(\Vert V\Vert_{E_0} +  \mathcal{N}_k(\gamma)\big) 
\Big\{  \Vert  f\Vert_{L^p(I; H^{s }(\xR^d))} 
+ \Vert  u \Vert_{C^0(I; H^s(\xR^d))}\Big\}
\end{aligned}
\end{equation*}
where   $p=4$ if $ d=1 $ and $p=2$ if $d \geq 2.$  
\end{theo}
\begin{proof}
The function  $  \Delta_j u  $ is a solution of 
$$
\Big(  \partial_t  + \mez( T_{V } \cdot \nabla + \nabla \cdot T_{V })  
+i T_{\gamma }  \Big)(\Delta_ju)  =F_j ,  \quad (\Delta_ju)\arrowvert_{t=0} = \Delta_ju_0, 
$$
where
$$
F_j = \Delta_j f +\mez [ T_{V } \cdot \nabla + \nabla \cdot T_{V }, \Delta_j]u +i[T_\gamma, \Delta_j]u.
$$
Then $F_j$ is spectrally supported in a ring 
$\big\{\xi: c_1 2^{-j} \leq \vert \xi \vert \leq c_2 2^{-j}\big\}$ 
and since
$$
V_i \in L^p(I; W^{1 , \infty}(\xR^d)),\quad \gamma \in L^2\big(I; \Gamma_{1/2}^{1/2}(\xR^d)\big)\cap  L^\infty\big(I; \Gamma_{0}^{1/2}(\xR^d)\big)
$$ 
it follows that, with constants independent of $j$ we 
have 
\begin{equation*}
\begin{aligned}
&\Vert [ T_{V } \cdot \nabla + \nabla \cdot T_{V }, \Delta_j]u\Vert_{L^p(I; H^s(\xR^d))} 
\leq C\Vert V \Vert_{L^p(I; W^{1, \infty}(\xR^d))} \Vert \widetilde \Delta_j u \Vert_{L^\infty(I; H^s(\xR^d))}\\
& \Vert [T_\gamma, \Delta_j]u \Vert_{L^p(I; H^s(\xR^d))} \leq C\Vert \gamma \Vert_{E_1} 
\Vert \widetilde \Delta_j u  \Vert_{L^\infty(I; H^s(\xR^d))} 
\end{aligned}
\end{equation*}
with $\widetilde \Delta_j u = \widetilde{\varphi}(2^{-j}D)u$  where $\widetilde{\varphi}$ 
is supported in a ring $\{\xi: 0< c'_1  \leq \vert \xi \vert \leq c'_2 \}$ and $\widetilde{\varphi} = 1$ on the support of $\varphi$.

It follows then from Corollary \ref{strich3} that, with constants independent of $j$ we have 
\begin{equation}\label{finale}
\begin{aligned}
\Vert  \Delta_j u  \Vert_{L^p(I;  W^{s -\sigma_d, \infty}(\xR^d))}
&\leq    \mathcal{F}(\cdots) \big\{\Vert \Delta_j f\Vert_{L^p(I; H^{s }(\xR^d))} 
+   \Vert \Delta_j u \Vert_{C^0(I; H^s(\xR^d))} \big\}
\end{aligned}
 \end{equation}
 where $\sigma_1 = \frac{3}{8} + \frac{\delta}{8}= \mez - \frac{1}{24}, \,  \sigma_d = \frac{d}{2} - \frac{1}{4} +\frac{\delta}{4}= \frac{d}{2}- \frac{1}{12}$ if $ d\geq 2.$
 
The right hand side of \eqref{finale} is bounded by 
\begin{equation}\label{finale1}
A :=  \mathcal{F} (\cdots) \Big\{ 
 \Vert f\Vert_{L^p(I; H^{s }(\xR^d))} 
+   \Vert  u \Vert_{C^0(I; H^s(\xR^d))} 
\Big\}.
\end{equation}

On the other hand we have
$$
\Vert  \Delta_j u  \Vert_{L^p(I;  W^{s - \sigma_d - \eps, \infty}(\xR^d))} 
\leq 2^{-j\eps} \Vert  \Delta_j u  \Vert_{L^p(I;  W^{s - \sigma_d, \infty}(\xR^d))}\leq 2^{-j\eps}A.
$$
Since $\sum_{j=-1}^\infty 2^{-j\eps}<+\infty$, summing on $j$, this completes the proof.
\end{proof}

\section{A priori estimates}

In this section we prove Theorem~\ref{T2}. To do so, 
we shall combine the Strichartz estimate proved above, together with the estimates already established in \cite{ABZ3}, to prove a priori estimates for the solutions of the Craig-Sulem-Zakharov system:
\begin{equation}\label{n400}
\left\{
\begin{aligned}
&\partial_{t}\eta-G(\eta)\psi=0,\\[1ex]
&\partial_{t}\psi+g \eta
+ \smash{\frac{1}{2}\la\partialx \psi\ra^2  -\frac{1}{2}
\frac{\bigl(\partialx  \eta\cdot\partialx \psi +G(\eta) \psi \bigr)^2}{1+|\partialx  \eta|^2}}
= 0.
\end{aligned}
\right.
\end{equation}
For the sake of clarity we recall here our assumptions and notations.  
\begin{assu}\label{T:22}
We consider smooth solutions of \eqref{n400} such that
\begin{enumerate}[i)]
\item $(\eta,\psi)$ belongs to $C^1([0,T_0]; H^{s_0}(\xR^d)\times H^{s_0}(\xR^d))$ for some 
$T_0$ in $(0,1]$ and some 
$s_0$ large enough;
\item there exists $h>0$ such that \eqref{n1} holds for any $t$ in $[0,T_0]$ (this is the assumption 
that there exists a curved strip of width $h$ separating the free surface from the bottom);
\item there exists $c>0$ such that the Taylor coefficient $a(t,x)=-\partial_y P\arrowvert_{y=\eta(t,x)}$ is bounded from below by $c$ 
from any $(t,x)$ in $[0,T_0]\times \xR^d$.
\end{enumerate}
\end{assu}
We work with the horizontal and vertical traces of the velocity 
on the free boundary, namely
\begin{gather*}
B= (\partial_y \phi)\arrowvert_{y=\eta},\quad 
V = (\nabla_x \phi)\arrowvert_{y=\eta},
\end{gather*}
which can be defined in terms of $\eta$ and $\psi$ by means of
\begin{equation}\label{n405}
B\defn \frac{\partialx \eta \cdot\partialx \psi+ G(\eta)\psi}{1+|\partialx  \eta|^2},
\qquad
V\defn \partialx \psi -B \partialx\eta.
\end{equation}
Let $s$ and $r$ be two positive real numbers such that
\begin{equation}\label{n406}
s>\frac{3}{4}+\frac{d}{2}, \quad s+\frac{1}{4}-\frac{d}{2} > r>1,\quad 
r\not\in\mez \xN. 
\end{equation}
Define, for $T$ in $(0,T_0]$, the norms
\begin{equation}\label{n410}
\begin{aligned}
M_s(T)&\defn 
\lA (\psi,\eta,B,V)\rA_{C^0([0,T];H^{s+\mez}\times H^{s+\mez}\times H^s\times H^s)},\\
Z_r(T)&\defn \lA \eta\rA_{L^p([0,T];W^{r+\mez,\infty})}
+\lA (B,V)\rA_{L^p([0,T];W^{r,\infty}\times W^{r,\infty})},
\end{aligned}
\end{equation}
where as above $p=4$ if $d=1$ and $p=2$ for $d\ge 2$. Recall also that, for $\rho= k + \sigma$ with $k\in\xN$ and $\sigma \in (0,1)$, one denotes 
by~$W^{\rho,\infty}(\xR^d)$ 
the space of functions whose derivatives up to order~$k$ are bounded and uniformly H\"older continuous with exponent~$\sigma$. Hereafter, we always consider indexes $\rho\not\in\xN$. 

Our goal is to estimate $M_s(T)+Z_r(T)$, for some $s,r$ satisfying \eqref{n406} and 
for $T$ in $(0,T_0]$, in terms of
\begin{equation}\label{n411}
M_{s,0}\defn \lA (\psi(0),\eta(0),B(0),V(0))\rA_{H^{s+\mez}\times H^{s+\mez}\times H^s\times H^s}.
\end{equation}

In \cite{ABZ3} we proved that, for any $s$ and $r$ satisfying \eqref{n406}, there exists 
a continuous non-decreasing 
function~$\mathcal{F}\colon \xR^+\rightarrow\xR^+$ such that, for 
all smooth solution 
$(\eta,\psi)$ of \eqref{n400} defined on the time interval~$[0,T_0]$ and satisfying 
Assumption~\ref{T:22} on this time interval, for any~$T\in (0,T_0]$,
\begin{equation}\label{n415}
M_s(T)\le \mathcal{F}\bigl(\mathcal{F}(M_{s,0})+T\mathcal{F}\bigl(M_s(T)+Z_r(T)\bigr)\bigr).
\end{equation}
If~$s>1+d/2$ then the Sobolev embedding implies that $Z_r(T)\les M_s(T)$ for~$r=s-d/2$. 
Therefore, for $s>1+d/2$, \eqref{n415} implies that  
$M_s(T)\le \mathcal{F}\bigl(\mathcal{F}(M_{s,0})+T\mathcal{F}\bigl(M_s(T)\bigr)\bigr)$, 
which implies (using classical arguments recalled in the proof of Corollary~\ref{T:25.5} below) that there exist two positive real numbers $T_1$ and $B$ depending only 
on $M_{s,0}$ such that
$M_s(T_1)\le B$. 

\begin{rema}
For the sake of completeness, let us mention that in \cite{ABZ3} we proved the estimate \eqref{n415} with $Z_r(T)$ replaced by 
$$
\widetilde{Z}_r(T)\defn \lA \eta\rA_{L^2([0,T];W^{r+\mez,\infty})}
+\lA (B,V)\rA_{L^2([0,T];W^{r,\infty}\times W^{r,\infty})}.
$$
This expression coincides with the one given above for $Z_r(T)$ for $d\ge 2$. However, for 
$d=1$, these two expressions differ because the definition of $Z_r(T)$ 
in \eqref{n410} involves the $L^4$-norm in time. 
However, for $T\le 1$, the $L^2$-norm in time is smaller than the $L^4$-norm in time 
and hence $\widetilde{Z}_r(T)\le Z_r(T)$. The estimate \eqref{n415} is thus an immediate corollary of the corresponding estimate in \cite{ABZ3}.
\end{rema}

We shall consider here the case where $s$ might be smaller than $1+d/2$. More precisely, we consider the case when $s>1+d/2-\mu$ for some positive 
real number $\mu$ given by the Strichartz estimate established previously (see Theorem~\ref{T20}). 
We shall show in this section that Theorem~\ref{T20} implies that $Z_r(T)$ satisfies an estimate 
analogous to \eqref{n415} for some $r>1$. This in turn will allow us 
to prove an {\em a priori} estimate for $M_s(T)+Z_r(T)$.

Denote by $\mu$ the ``gain'' provided by the Stichartz estimate established previously. 
Namely, fix a positive real number $\mu$ 
such that $\mu=\frac{1}{24}$ if $d=1$ and $\mu<\frac{1}{12}$ for $d\ge 2$. 
We also fix two real positive 
numbers $s$ and $r$ such that
\begin{equation}\label{n416}
s>1+\frac{d}{2}-\mu, \quad 1<r<s+\mu-\frac{d}{2},\quad 
r\not\in\mez \xN. 
\end{equation}
(\eqref{n406} corresponds to \eqref{n416} with $\mu$ replaced by $1/4$ so 
\eqref{n416} is stronger than \eqref{n406}.) 
We want to prove that there exists a continuous non-decreasing 
function~$\mathcal{F}\colon \xR^+\rightarrow\xR^+$ such that
\begin{equation}\label{n417}
M_s(T)+Z_r(T)\le \mathcal{F}\bigl(\mathcal{F}(M_{s,0})+T\mathcal{F}\bigl(M_s(T)+Z_r(T)\bigr)\bigr).
\end{equation}
Since $\mu<1/4$ and since 
the estimate \eqref{n415} is proved in \cite{ABZ3} under the general assumption~\eqref{n406}, 
it remains only to prove that $Z_r(T)$ satisfies an estimate similar to 
\eqref{n415}. 
\begin{prop}\label{T25}
Let~$d\ge 1$ and consider~$s,r$ satisfying \eqref{n416}. 
There exists 
a continuous non-decreasing 
function~$\mathcal{F}\colon \xR^+\rightarrow\xR^+$ such that, for all~$T_0\in (0,1]$ 
and all smooth solution 
$(\eta,\psi)$ of \eqref{n400} defined on the time interval~$[0,T_0]$ and satisfying 
Assumption~\ref{T:22} on this time interval, there holds
\begin{equation}\label{n425}
Z_r(T)\le \mathcal{F}\bigl(T\mathcal{F}\bigl(M_s(T)+Z_r(T)\bigr)\bigr),
\end{equation}
for any $T$ in $[0,T_0]$.
\end{prop}

Let us admit this result and deduce Theorem~\ref{T2}Ê
stated in the introduction, which is 
an {\em a priori} estimate in low norms. 

\begin{coro}\label{T:25.5}
Let $T_0>0$. For any $A>0$ there exist $B>0$ and $T_1\in (0,T_0]$ such that, 
for all smooth solution 
$(\eta,\psi)$ of \eqref{n400} defined on the time interval~$[0,T_0]$ and satisfying 
Assumption~\ref{T:22} on this time interval, if 
$M_{s,0}\le A$ then $M_s(T_1)+Z_r(T_1)\le B$.
\end{coro}
\begin{proof}
Introduce for $T$ in $[0,T_0]$, 
$f(T)=M_s(T)+Z_r(T)$. It follows from \eqref{n415}Ê
and \eqref{n425} that \eqref{n417}Ê
holds. This means that there exists a continuous non-decreasing 
function~$\mathcal{F}\colon \xR^+\rightarrow\xR^+$ such that, for all~$T\in (0,T_0]$, 
\begin{equation}\label{n430}
f(T)\le \mathcal{F}\bigl(\mathcal{F}(A)+T\mathcal{F}\bigl(f(T)\bigr)\bigr).
\end{equation}
Now fix $B$ such that $B>\max\big\{ A,\mathcal{F}\bigl(\mathcal{F}(A)\bigr)\big\}$
and then chose $T_1\in (0,T_0]$ such that
$$
\mathcal{F}\bigl(\mathcal{F}(A)+T_1\mathcal{F}(B)\bigr)<B.
$$
We claim that $f(T)<B$ for any $T$ in $[0,T_1]$. Indeed, since $f(0)=M_{s,0}\le A <B$, 
assume that there exists $T'\in (0,T_1]$ such that $f(T')>B$. Since $f$ is continuous, this implies that 
there is $T''\in (0,T_1)$ such that $f(T'')=A$. Now it follows from~\eqref{n430}, the fact that 
$\mathcal{F}$ is increasing, and the definition of $T_1$ that
$$
B=f(T'')\le  \mathcal{F}\bigl(\mathcal{F}(A)+T''
\mathcal{F}\bigl(f(T'')\bigr)\bigr)
\le \mathcal{F}\bigl(\mathcal{F}(A)+T_1
\mathcal{F}\bigl(B\bigr)\bigr)<B.
$$
Hence the contradiction which proves that $f(T)\le B$ 
for any $T$ in $[0,T_1]$. \end{proof}

It remains to prove Proposition~\ref{T25}. This will be the purpose of the following subsections.

\subsection{Reduction}

We begin by using an interpolation 
inequality to reduce the proof of Proposition~\ref{T25} to the proof of the following proposition.

\begin{prop}\label{T26}
Let~$d\ge 1$ and consider~$\mu,s,r$ as in \eqref{n416}. 
Consider in addition $r'$ such that
$$
r<r'<s+\mu-\frac{d}{2},\quad r'\not\in\mez \xN
$$
and set
$$
Z_{r'}(T)\defn \lA \eta\rA_{L^p([0,T];W^{r'+\mez,\infty})}
+\lA (B,V)\rA_{L^p([0,T];W^{r',\infty}\times W^{r',\infty})}
$$
where $p=4$ if $d=1$ and $p=2$ for $d\ge 2$. 
There exists 
a continuous non-decreasing 
function~$\mathcal{F}\colon \xR^+\rightarrow\xR^+$ such that, for all~$T_0\in (0,1]$ 
and all smooth solution 
$(\eta,\psi)$ of \eqref{n400} defined on the time interval~$[0,T_0]$ and satisfying 
Assumption~\ref{T:22} on this time interval, there holds
\begin{equation}\label{n435}
Z_{r'}(T)\le \mathcal{F}\bigl(M_s(T)+Z_r(T)\bigr),
\end{equation}
for any $T$ in $[0,T_0]$.
\end{prop}
We prove in this paragraph that Proposition~\ref{T26} implies 
Proposition~\ref{T25}. Proposition~\ref{T26} will be proved in the next 
paragraph.

\begin{proof}[Proof of Proposition~\ref{T25} given Proposition~\ref{T26}] 
Consider a function $v=v(t,x)$. By interpolation, since $1-\mu<1<r<r'$, 
there exists a real number $\theta\in (0,1)$ such that
$$
\lA v(t,\cdot)\rA_{W^{r,\infty}}\les \lA v(t,\cdot)\rA_{W^{1-\mu,\infty}}^\theta
\lA v(t,\cdot)\rA_{W^{r',\infty}}^{1-\theta}.
$$
This implies that
$$
\int_0^T \lA v(t,\cdot)\rA_{W^{r,\infty}}^p\, dt
\les \lA v\rA_{C^0([0,T];W^{1-\mu,\infty})}^{p\theta}
\int_0^T \lA v(t,\cdot)\rA_{W^{r',\infty}}^{p(1-\theta)}\, dt.
$$
The H\"older inequality then implies that
$$
\lA v\rA_{L^p([0,T];W^{r,\infty})}
\les T^{\frac{\theta}{p}}
\lA v\rA_{C^0([0,T];W^{1-\mu,\infty})}^{\theta}
\lA v\rA_{L^p([0,T];W^{r',\infty})}^{1-\theta}.
$$
By the same way, there holds
$$
\lA v\rA_{L^p([0,T];W^{r+\mez,\infty})}
\les T^{\frac{\theta'}{p}}
\lA v\rA_{C^0([0,T];W^{1-\mu+\mez,\infty})}^{\theta'}
\lA v\rA_{L^p([0,T];W^{r'+\mez,\infty})}^{1-\theta'}.
$$
Since $s>(1-\mu)+d/2$, the Sobolev embedding implies that
$$
\lA v\rA_{C^0([0,T];W^{1-\mu,\infty})}\les \lA v\rA_{C^0([0,T];H^{s})},
\quad \lA v\rA_{C^0([0,T];W^{1-\mu+\mez,\infty})}\les \lA v\rA_{C^0([0,T];H^{s+\mez})}.
$$
This proves that
$$
Z_{r}(T)\le T^{\frac{\theta}{p}} M_s(T)^\theta (Z_{r'}(T))^{1-\theta}
$$
for some $\theta>0$. This in turn proves that \eqref{n435} implies \eqref{n425}.
\end{proof} 

\subsection{Proof of Proposition~\ref{T26}}

Recall that the positive real number 
$\mu$ has been chosen such that $\mu<1/24$ if $d=1$ and $\mu<1/12$ 
for $d\ge 2$, and $s,r,r'$ are such that
\begin{equation*}
s>1+\frac{d}{2}-\mu, \quad 1<r<r'<s+\mu-\frac{d}{2},\quad 
r\not\in\mez \xN. 
\end{equation*}
Let $T>0$ and set $I=[0,T]$.

The proof of Proposition~\ref{T26} is based on Corollary~\ref{T10} and Theorem~\ref{T20}. 
By combining these two results we shall deduce in the first step of the proof that 
\begin{equation}\label{n440}
\lA u\rA_{L^p(I;W^{r',\infty})}\le \mathcal{F}\bigl(M_s(T)+Z_r(T)\bigr)
\end{equation}
where $u$ is defined in terms of $(\eta,V,B)$ by (see \eqref{eq:r0})
\begin{equation}\label{n440.5}
\begin{aligned}
&u=\lDx{-s}(U_s-i \theta_s),\\[0.5ex]
&U_s \defn \lDx{s} V+T_\zeta \lDx{s}\B \qquad (\zeta=\partialx\eta),\\[0.5ex]
&\theta_s\defn T_{\sqrt{\ma/\lambda}}\lDx{s}\nabla \eta,
\end{aligned}
\end{equation}
where recall that $a$ is the Taylor coefficient 
and $\lambda$ is the principal symbol of 
$G(\eta)$, given by $\lambda=\sqrt{(1+\la\partialx\eta\ra^2)\la\xi\ra^2-(\partialx\eta\cdot\xi)^2}$. 

In the next steps of the proof we show 
how to recover estimates for the original unknowns 
$(\eta,V,\B)$ in $L^p([0,T];W^{r'+\mez}\times W^{r',\infty}\times W^{r',\infty})$. 
 
\paragraph{Step1: proof of \eqref{n440}.}

It follows from Theorem~\ref{T20} that 
\begin{equation*}
\Vert  u  \Vert_{L^p(I;W^{r', \infty}(\xR^d))}
\leq    \mathcal{F}\big(\Vert V\Vert_{E_0} +  \mathcal{N}_k(\gamma)\big) 
\Big\{\Vert  f\Vert_{L^p(I; H^{s }(\xR^d))} 
+ \Vert  u \Vert_{C^0(I; H^s(\xR^d))}\Big\}
\end{equation*}
Clearly we have
$$
\Vert V\Vert_{E_0}\le Z_1(T)\le Z_r(T),\quad \mathcal{N}_k(\gamma)
\le M_s(T).
$$
Moreover, \eqref{eq:r2} and \eqref{n201} imply that 
$$
\Vert  u \Vert_{C^0(I; H^s(\xR^d))}\le \mathcal{F}(M_s(T)),\quad 
\Vert  f\Vert_{L^p(I; H^{s }(\xR^d))} \le \mathcal{F}\bigl(M_s(T)+Z_r(T)\bigr).
$$
By combining the previous estimates we deduce the desired estimate \eqref{n440}.

\paragraph{Step 2: estimate for $\eta$.} 

Separating real and imaginary parts, directly from the definition \eqref{n440.5} 
of $u$, we get
$$
\big\lVert \lDx{-s}T_{\sqrt{a/\lambda}}\lDx{s}\partialx\eta \big\rVert_{W^{r',\infty}}
\le \lA u\rA_{W^{r',\infty}}.
$$

We shall make repeated uses of the following elementary properties 
of paradifferential operators.
\begin{prop}\label{theo:scZ}
Let~$m,m'\in\xR$ and~$\rho\in (0,1)$. Given a symbol~$a\in \Gamma^m_{\rho}(\xR^d)$ 
(see Definition~\ref{T:5}), recall the notation
\begin{equation}\label{defi:norms-2}
M_{\rho}^{m}(a)= 
\sup_{\la\alpha\ra\le 1+2d+\rho~}\sup_{\la\xi\ra \ge 1/2~}
\lA (1+\la\xi\ra)^{\la\alpha\ra-m}\partial_\xi^\alpha a(\cdot,\xi)\rA_{W^{\rho,\infty}(\xR)}.
\end{equation}

$(i)$ If~$a \in \Gamma^m_0({\mathbf{R}}^d)$, then for all~$\mu\in[\max(0,m),+\infty)$ (with $\mu\not\in\xN$ and $\mu-m\not\in\xN$) 
there exists a constant~$K$ such that
\begin{equation}\label{esti:quant1}  
\lA T_a \rA_{W^{\mu,\infty}\rightarrow W^{\mu-m,\infty}}\le K M_{0}^{m}(a).
\end{equation}
$(ii)$ If~$a\in \Gamma^{m}_{\rho}({\mathbf{R}}^d), b\in \Gamma^{m'}_{\rho}({\mathbf{R}}^d)$ then 
for all~$\mu\in[0,+\infty)$ there exists a constant~$K$ such that
\begin{equation}\label{esti:quant2}
\lA T_a T_b  - T_{a b}   \rA_{W^{\mu,\infty}\rightarrow W^{\mu-m-m'+\rho,\infty}}\le 
K M_{\rho}^{m}(a)M_{0}^{m'}(b)+K M_{0}^{m}(a)M_{\rho}^{m'}(b)
\end{equation}
provided $\mu\not\in\xN$, $0\le \mu-m-m'+\rho\not\in\xN$

$(iii)$ 
Let $(r,r_1,r_2)\in [0,+\infty)^3$ be such that $r\le \min (r_1+\rho,r_2+m)$, $r\not\in\xN$. 
Consider the equation $T_\tau v=f$ where $\tau=\tau(x,\xi)$ is a symbol such that 
$\tau$ and $1/\tau$ belongs to $\Gamma^m_\rho$. 
Then
$$
\lA v\rA_{W^{r,\infty}}\le K \lA v\rA_{W^{r_1,\infty}}+K \lA f\rA_{W^{r_2,\infty}}
$$
for some constant $K$ depending only on 
$\mathcal{M}^m_\rho(\tau)+\mathcal{M}^m_\rho(1/\tau)$.
\end{prop}
The first two points are classical (these points are discussed in \cite{ABZ3} with the following difference: in \cite{ABZ3} we worked in Zygmund spaces $C^r_*$. This makes no difference here 
since $C^r_*=W^{r,\infty}$ for any $r$ in $[0,+\infty)\setminus \xN$). To prove the third one, write
$$
v=\big(I -T_{1/\tau}T_\tau\big)v+T_{1/\tau}f
$$
and use the first (resp.\ second) point to estimate the first (resp.\ second) term.

Now write
$$
\lDx{-s}T_{\sqrt{a/\lambda}}\lDx{s}\partialx=
T_{\sqrt{a/\lambda}}\partialx+R
$$
where $R=\big[ \lDx{-s},T_{\sqrt{a/\lambda}}\lDx{s}\partialx\bigr]$. 
Since $\sqrt{a/\lambda}$ is a symbol of order $-1/2$ in $\xi$, it follows from the 
second point in Proposition~\ref{theo:scZ} that, for any $\rho\in (0,1)$, 
$$
\lA R\eta \rA_{W^{r',\infty}}\le K
\mathcal{M}^{-\mez}_\rho\Big(\sqrt{\frac{a}{\lambda}}\Big)\lA \eta\rA_{W^{r'+\mez-\rho,\infty}}
$$
and hence
$$
\lA R\eta \rA_{W^{r',\infty}}\le \mathcal{F}\bigl(\lA\partialx\eta\rA_{W^{\rho,\infty}},
\lA a\rA_{W^{\rho,\infty}}\big)
\lA \eta\rA_{W^{r'+\mez-\rho,\infty}}.
$$
Now by assumption on $s$ and $r'$ we can chose $\rho$ (say $\rho=1/4$) so that
$$
\rho<s-\mez-\frac{d}{2},\quad r'+\mez-\rho<s+\mez-\frac{d}{2}
$$
and hence 
$$
\lA R\eta \rA_{W^{r',\infty}}\le\mathcal{F}\bigl(\lA\eta\rA_{H^{s+\mez}},
\lA a\rA_{H^{s-\mez}}\big).
$$
Recalling (see \cite[Prop. 4.6]{ABZ3}) that 
$\lA a\rA_{H^{s-\mez}}\le \mathcal{F}(M_s)$ for any $s>3/4+d/2$, we obtain 
$\lA R\eta \rA_{W^{r',\infty}}\le \mathcal{F}(M_s)$. 

We thus deduce that
$$
\lA T_{\sqrt{a/\lambda}}\partialx\eta\rA_{W^{r',\infty}}\le 
\lA u\rA_{W^{r',\infty}}+\mathcal{F}(M_s).
$$
Using statement $(iii)$ in Proposition~\ref{theo:scZ} with 
$m=1/2$ and $\rho=1/4$, this yields
$$
\lA \eta\rA_{W^{r'+\mez,\infty}}\le 
K\lA u\rA_{W^{r',\infty}}+K\lA \eta\rA_{W^{r'+\mez-\rho,\infty}}+K \mathcal{F}(M_s)
$$
for some constant $K$ depending only on $\mathcal{M}^{-\mez}_\rho\Big(\sqrt{\frac{a}{\lambda}}\Big)$. As already seen, $K\le \mathcal{F}(M_s)$ and 
$\lA \eta\rA_{W^{r'+\mez-\rho,\infty}}\le M_s$ (using the Sobolev embedding). 
We conclude that
$$
\lA \eta\rA_{W^{r'+\mez,\infty}}\le 
\mathcal{F}(M_s)\lA u\rA_{W^{r',\infty}}+\mathcal{F}(M_s).
$$
Therefore \eqref{n440}Ê
implies that
\begin{equation}\label{n445}
\lA \eta\rA_{L^p(I;W^{r'+\mez,\infty})}\le \mathcal{F}\bigl(M_s(T)+Z_r(T)\bigr).
\end{equation}

\paragraph{Step 3: estimate for $V+T_{\zeta}B$.} 

We proceed as above: starting from \eqref{n440} one deduces an estimate for 
the $W^{r',\infty}$-norm of $\lDx{-s}\bigl(\lDx{s}V+T_{\zeta}\lDx{s}B\bigr)$. One 
rewrite this term as $V+T_\zeta B$ plus a commutator which is estimated by 
statement $(ii)$ in Proposition~\ref{theo:scZ} and the Sobolev embedding. 
It is find that
$$
\lA V+T_\zeta B\rA_{W^{r',\infty}}\le \mathcal{F}(M_s)\lA u\rA_{W^{r,\infty}}+\mathcal{F}(M_s)
$$
so that \eqref{n440}Ê
implies that
\begin{equation}\label{n447}
\lA V+T_\zeta B\rA_{L^p(I;W^{r',\infty})}\le \mathcal{F}\bigl(M_s(T)+Z_r(T)\bigr).
\end{equation}

\paragraph{Step 4: estimate for $V$ and $B$.} We shall estimate the 
$L^p(I;W^{r',\infty})$-norm of $B$. The estimate for 
the $L^p(I;W^{r',\infty})$-norm of $V$ will follow from 
$V=(V+T_\zeta B)-T_\zeta B$ since the first term 
$V+T_\zeta B$ is already estimated (see \eqref{n447}) and since, 
for the second term, 
one can use the rule \eqref{esti:quant1} to obtain 
$\lA T_\zeta B\rA_{W^{r',\infty}}\les 
\lA \zeta\rA_{L^\infty}\lA B\rA_{W^{r',\infty}}\les \lA \eta\rA_{H^{s+\mez}}\lA B\rA_{W^{r',\infty}}$.

To estimate the $W^{r',\infty}$-norm of $B$, as in \cite{ABZ3}, we use the identity 
$G(\eta)B=-\cnx V+\gamma$ where (see \cite[Prop. 4.5]{ABZ3}) 
$$
\lA \gamma\rA_{W^{r'-1,\infty}}\le \lA \gamma\rA_{H^{s-\mez}}\le \mathcal{F}
\big(\lA (\eta,V,B)\rA_{H^{s+\mez}\times H^\mez\times H^\mez}\big).
$$
In order to use this identity, write
\begin{equation}\label{n450}
\begin{aligned}
\cnx \big( V+T_\zeta B\big)&=\cnx V+T_{\cnx \zeta}B+T_{\zeta}\cdot\partialx B\\
&=-G(\eta)B+T_{\cnx \zeta}B+T_{\zeta}\cdot\partialx B+\gamma\\
&=T_{-\lambda+i\zeta\cdot\xi}B+r
\end{aligned}
\end{equation}
where
$$
r=T_{\cnx \zeta} B+\gamma+(T_\lambda-G(\eta))B.
$$
The first term in the right-hand side is estimated by means of 
\begin{alignat*}{2}
\lA T_{\cnx \zeta} B\rA_{W^{r'-1,\infty}}
&\les \lA T_{\cnx \zeta} B\rA_{H^{s-1+a}} \qquad &&\text{since }
r'<s+\mu-\frac{d}{2}\\
&\les \lA \cnx \zeta\rA_{C^{a-1}_*}\lA B\rA_{H^s} && \text{(see chapter 2 in~\cite{BCD})}\\
&\les \lA \eta\rA_{W^{1+a,\infty}}\lA B\rA_{H^s}&&\text{since } \cnx \zeta=\Delta\eta \\
&\les \lA \eta\rA_{H^{s+\mez}}\lA B\rA_{H^s}&&\text{since }1+\mu<1+\frac{1}{4}
<s+\mez-\frac{d}{2}\cdot
\end{alignat*}
The key point is to estimate $(T_\lambda-G(\eta))B$. 
To do so we shall use the following
\begin{prop}[Prop.\ 3.19 in \cite{ABZ3}]
Let~$d\ge 1$ and~$s>\mez+\frac{d}{2}$. For any~$\frac{1}{2}\leq \sigma \leq s$ and any
$$
0<\eps\leq \mez, \qquad \eps< s-\mez-\frac{d}{2},
$$
there exists a non-decreasing function~$\mathcal{F}\colon\xR^+\rightarrow\xR^+$ such that 
$R(\eta)f\defn G(\eta)f-T_\lambda f$ satisfies
\begin{equation*}
\lA R(\eta)f\rA_{H^{\sigma-1+\eps}({\mathbf{R}}^d)}\le \mathcal{F} \bigl(\| \eta \|_{H^{s+\mez}({\mathbf{R}}^d)}\bigr)\lA f\rA_{H^{\sigma}({\mathbf{R}}^d)}.
\end{equation*}
\end{prop}
Using the Sobolev embedding, this implies that
$$
\lA (T_\lambda-G(\eta))B\rA_{W^{r'-1,\infty}}\le \mathcal{F}
\big(\lA (\eta,B)\rA_{H^{s+\mez}\times H^s}\big).
$$
We end up with
$$
\lA r\rA_{W^{r'-1,\infty}}\le 
\mathcal{F}
\big(\lA (\eta,V,B)\rA_{H^{s+\mez}\times H^s\times H^s}\big).
$$
Writing (see \eqref{n450}) 
$$
T_{-\lambda+i\zeta\cdot\xi}B=\cnx \big( V+T_\zeta B\big)-r,
$$ 
it follows from \eqref{n447} and
statement $(iii)$ in Proposition~\ref{theo:scZ} 
that 
$$
\lA B\rA_{L^p(I;W^{r'+\mez,\infty})}\le \mathcal{F}\bigl(M_s(T)+Z_r(T)\bigr).
$$
This completes the proof.

\vspace{1cm}

\noindent\textbf{Thomas Alazard}\\
\noindent DMA, \'Ecole normale sup\'erieure et CNRS UMR 8553 \\ 
\noindent 45 rue dÕUlm, 75005 Paris, France

\vspace{2mm}

\noindent\textbf{Nicolas Burq}\\
\noindent Univiversit\'e Paris-Sud  \\ 
\noindent D\'epartement de Math\'ematiques  \\ 
\noindent 91405 Orsay, France

\vspace{2mm}

\noindent\textbf{Claude Zuily}\\
\noindent 
Universit\'e Paris-Sud \\ 
\noindent D\'epartement de Math\'ematiques  \\ 
\noindent 91405 Orsay, France

\end{document}